\documentclass[12pt,oneside,reqno]{article}
\usepackage{amsmath}
\usepackage{paralist}
\usepackage{graphics}
\usepackage{epsfig}
\usepackage{float}
\usepackage{graphicx}
\usepackage{epstopdf}
\usepackage{amssymb}
\usepackage{cite}
\usepackage{mathrsfs}
\usepackage{amsthm}
\usepackage{xcolor}
\usepackage{tikz}
\usepackage{caption}
\usepackage{subfigure}
\usepackage[colorlinks=true]{hyperref}
\hypersetup{urlcolor=blue, citecolor=blue}
\usepackage[utf8]{inputenc}
\usepackage[shortlabels]{enumitem}
\usepackage{cancel} 
\usepackage[margin=1.0in]{geometry}
\usepackage{overpic}
\usepackage[normalem]{ulem}

\numberwithin{equation}{section}

\newtheorem{theorem}{Theorem}[section]

\newtheorem{lemma}[theorem]{Lemma}
\newtheorem{proposition}{Proposition}

\newtheorem{physical conclusion}{Physical Conclusion}
\newtheorem{example}{Example}
\newtheorem{definition}[theorem]{Definition}
\newtheorem{remark}{Remark}

\definecolor{darkgreen}{rgb}{0,0.35,0}

\title{A center manifold reduction approach to the Darcy–B\'{e}nard convection problem with non-zero Prandtl number}

\author{
Liang Li$^{a}$
\thanks{liliang187@gxu.edu.cn}
\quad\quad
Quan Wang$^{b}$
\thanks{Corresponding author:xihujunzi@scu.edu.cn }
\\ \footnotesize $^a$ School of Mathematics,
\footnotesize Guangxi University,
 Nanning, Guangxi, 530000, P.R.China
  \\ \footnotesize $^{b}$ College of Mathematics, Sichuan University,
  \footnotesize
 Chengdu, Sichuan, 610065,  China
}
\medskip

\begin{document}
\maketitle
\begin{abstract}
We study the bifurcation of two-dimensional Darcy–B\'{e}nard convection (DBC) in a rectangular domain, a canonical model for thermal convection in porous media with applications in geophysics and engineering. The momentum equation lacks advection and viscous dissipation, being regularized solely by a linear Darcy damping term. As a result, the linearized operator generates a semigroup that is neither analytic nor compact and the nonlinear term fails to be Lipschitz. The system is thus placed outside the scope of the standard center-manifold theorem.

To overcome these obstructions, we develop a center-manifold reduction adapted to DBC system. Our main result is a constructive proof of the existence of the center manifold function $h$ and local attractivity of the center manifold—the exponential convergence of small solutions toward it—without relying on analyticity of the full linear semigroup and Lipschitz nonlinearity. We circumvent these difficulties by exploiting the partially dissipative structure: the temperature equation is governed by the Laplacian, which generates an analytic semigroup and provides the smoothing needed to compensate for the lack of regularity in the velocity and the absence of global Lipschitz bounds. Through carefully designed inequalities, we establish both the construction of the center manifold function $h$ and the exponential convergence of nearby solutions.

Explicit approximations for the center manifold are derived in two scenarios—one simple eigenvalue and two distinct eigenvalues—yielding reduced systems of ordinary differential equations whose analysis determines the bifurcation type. Numerical simulations are presented to corroborate the theoretical results.



\end{abstract}

\newpage
\tableofcontents
\date{}
\maketitle
\section{Introduction}
The Rayleigh–B\'{e}nard convection (RBC) problem, concerning the buoyancy-driven instability of a horizontal fluid layer heated from below, constitutes one of the paradigmatic settings for the study of pattern formation in dissipative systems. Its mathematical structure—a nonlinear system of partial differential equations coupling the Navier–Stokes equations with a thermal energy balance—has served as a fertile ground for the development of bifurcation theory and nonlinear stability analysis  over the past century\cite{Feireisl2024,Venturi2010,Haragus20211,Wang2004,Wang2007 ,Chossat1996,MT2004,MT2007,Ma2019}. When the fluid saturates a porous medium of low permeability, the momentum balance reduces to Darcy's law, giving rise to the so-called Darcy–B\'{e}nard convection (DBC) problem, historically also referred to as the Horton–Rogers–Lapwood problem following the independent derivations by Horton and Rogers 
\cite{Horton1945} and Lapwood \cite{Lapwood1948}. For the classical configuration with isothermal, impermeable boundaries, linear stability theory predicts the existence of a critical Darcy–Rayleigh number, denoted by $R_{c}$, above which the motionless conduction state becomes linearly unstable\cite{Chandrasekhar1961,Nield2013}. This linear criterion, however, leaves open the fundamental nonlinear question of whether the bifurcating branches are supercritical or subcritical, and whether they are stable with respect to finite-amplitude perturbations. 

The system under consideration is a two-dimensional dimensionless model for the DBC problem in a rectangular porous cavity $\Omega=\left[0,L\right]\times\left[0,1\right]$. The fluid motion is governed by Darcy’s law, with velocity  $\mathbf{u}=\mathbf{u}\left(x,t\right)=\left(u_{1},u_{2}\right)$, pressure $p=p\left(x,t\right)$ and a modified temperature $\theta=\theta\left(x,t\right)$ satisfying an advection–diffusion equation. Under the Boussinesq approximation, the coupled system reads
\begin{align}\label{model}
    \begin{cases}
        \frac{1}{V_{a}}\frac{\partial \mathbf{u}}{\partial t}=-\nabla p -\mathbf{u}+R\theta \mathbf{e}_{2},
        \\
        \nabla\cdot\mathbf{u}=0,
        \\
        \frac{\partial \theta}{\partial t}+\mathbf{u}\cdot\nabla\theta=Ru_{2}+\Delta \theta,
    \end{cases}
\end{align}
where $\mathbf{e}_{2}=(0,1)$ is the upward unit vector.
The positive parameters $\frac{1}{V_{a}}$ and $R^2$ denote, respectively, Darcy-Prandtl number and the Rayleigh number. 
The boundary conditions are chosen to reflect a standard configuration: the top and bottom boundaries are isothermal, while the lateral walls are thermally insulating. In addition, the velocity satisfies the impermeability condition on the entire boundary. These conditions take the explicit form
\begin{align}\label{bianjian0814}
\mathbf{u}\cdot\mathbf{n}|_{\partial\Omega}=0,~
\theta|_{x_{2}=0,1}=0,~\nabla\theta\cdot\mathbf{\mathbf{n}}|_{x_{1}=0,L}=0,
\end{align}
where $\mathbf{n}$ denotes the outward unit normal on $\partial \Omega$.

The mathematical structure of system \eqref{model} is governed by the Darcy–Prandtl number $\frac{1}{V_{a}}$. In the limiting case $V_{a}\rightarrow +\infty$ (zero Darcy-Prandtl number), the momentum equation becomes stationary, after Leray projection, the linearized operator has a compact resolvent. The nonlinear dynamics of \eqref{model} have been extensively studied in this setting \cite{Turkyilmazoglu2023,Capone2023,Straughan2001,Han2020,Ly1999,Gupta1973,Liliang2025}. For the physically relevant case of non-zero Darcy–Prandtl number ($V_{a}<\infty$), several works have addressed the DBC problem from the viewpoint of dynamical behaviour \cite{VADASZ1998,Vadasz2000,Siddheshwar2021,Siddheshwar2024}. However, a rigorous bifurcation analysis for this regime is still lacking. The underlying difficulty is that, as observed by Fabrie\cite{Fabrie1996} and Oliver–Titi\cite{Oliver2000}, the solution semigroup for the DBC system is not smoothing, and the linearized operator fails to have a compact inverse on the natural function spaces. This loss of compactness is not a mere technicality: it excludes the system from the domain of applicability of the classical center-manifold theorem, which is the standard tool for studying bifurcations in infinite-dimensional systems. The aim of this work is to develop a rigorous bifurcation theory for this physically important yet mathematically non-standard setting. To this end, we establish a center-manifold reduction tailored to the system  \eqref{model}.

We now briefly recall the classical theory of center manifolds and explain why the system under consideration falls outside its applicability. To this end, we introduce the following abstract framework,
\begin{align}\label{ziji0730}
    \frac{dw}{dt}=\mathcal{L}w+\mathcal{G}(w)
\end{align}
where $\mathcal{L}:\mathcal{X}_{1}\rightarrow\mathcal{X}_{2}$ is a linear and bounded operator, $\mathcal{G}:\mathcal{X}_{3}\rightarrow\mathcal{X}_{2}$ is a nonlinear operator, $\mathcal{X}_{i}$($i=1,2,3$) are three Banach spaces with $\mathcal{G}\left(0\right)=0$ of higher order. Generally speaking, the existence of invariant manifolds requires only mild hypotheses. The construction of the center-manifold function $h$, however, is system-dependent and it is typically obtained as the fixed point of a contraction mapping, whose existence hinges on a local Lipschitz condition for the nonlinearity. Below we recall three  settings in which such a contraction is guaranteed:
\begin{enumerate}[label=(\arabic*)]
    \item Analytic semigroup setting\cite{Henry1993}. The operator $\mathcal{L}$ is sectorial and $\mathcal{G}$ satisfies $$\left\|\mathcal{G}\left(w_{1}\right)-\mathcal{G}\left(w_{2}\right)\right\|_{\mathcal{X}_{2}}\leq C\left\|w_{1}-w_{2}\right\|_{\mathcal{X}_{\beta}},$$ 
    where $\mathcal{X}_{\beta}$ is a suitable fractional Sobolev space.
    \item Strongly continuous semigroup with smooth nonlinearity \cite{Carr1981}. The operator $\mathcal{L}$ generates a strongly continuous semigroup, and $\mathcal{G}:\mathcal{X}_{2}\rightarrow\mathcal{X}_{2}$ is twice continuously Fr\'{e}chet differentiable.
    \item Completely continuous field setting\cite{ma2007}. The operator $\mathcal{L}$ is a completely continuous field, and $\mathcal{G}:\mathcal{X}_{2}\rightarrow\mathcal{X}_{2}$ is $C^{r}\left(r\geq 1\right)$ bounded mapping.
\end{enumerate}
A careful comparison shows that system \eqref{model}-\eqref{bianjian0814} does not fall within any of the three frameworks above. Indeed, the semigroup generated by the linear operator $\mathbf{L}$ (defined in \eqref{juyou0205}) is strongly continuous but not analytic: its spectrum has two accumulation points—one at $-\infty$ and the other at the finite negative value $-V_{a}$. Moreover, the local Lipschitz continuity of $\mathcal{G}$, that is 
\begin{align}\label{jiushi0822}
\left\|\mathcal{G}\left(w_{1}\right)
-\mathcal{G}\left(w_{2}\right)\right\|_{\mathcal{X}_{\beta}}\leq C\left\|w_{1}-w_{2}\right\|_{\mathcal{X}_{\beta}},
\end{align}
can be deduced from the conditions on $\mathcal{G}$ in \cite{Carr1981,ma2007}.
However, obviously,
the nonlinearity operator $G\left(\Phi\right)=\left(\mathbf{0},
        -\mathbf{u}\cdot\nabla\theta\right)$ in \eqref{model} does not satisfy 
 the above condition \eqref{jiushi0822} and this failure is readily verified.
 Additionally, Haragus and Iooss \cite{Haragus2011}  have developed an abstract framework for equations of the form \eqref{ziji0730}, in which the linear operator $\mathcal{L}$ is assumed to generate a strongly continuous semigroup and decompose as $\mathcal{L}=\mathcal{L}_{0}\oplus\mathcal{L}_{h}$, where $\mathcal{L}_{0}$ and $\mathcal{L}_{h}$ denote, respectively, the restrictions of 
$\mathcal{L}$ to the subspaces associated with eigenvalues having zero and non-zero real parts. They further require that the equation
 $$
 \frac{du_{h}}{dt}=\mathcal{L}_{h}u_{h}+f(t),~f \in \mathscr{C}_\eta(\mathbb R,\mathscr Y_h), 
 $$
 has a unique solution $u_{h}=\mathbf{K}_{h}f\in \mathscr{C}_\eta(\mathbb R,\mathscr{Z}_h)$ satisfying 
 $$
\left\|\mathbf{K}_{h}\right\|_{\mathscr{L}\bigl(\mathscr{C}_\eta(\mathbb R,\mathscr Y_h),\,\mathscr{C}_\eta(\mathbb R,\mathscr Z_h)\bigr)} \le \mathbf{C}(\eta),
 $$
 where $\eta\geq 0 $ is a constant, $\mathbf{C}\left(\cdot\right)\geq 0$ is a continuous mapping and the definitions of$ \mathscr{C}_\eta(\mathbb R,\mathscr Y_h)$ and $\mathscr{C}_\eta(\mathbb R,\mathscr{Z}_h)$ can be found in pp. 29 of \cite{Haragus2011}.
 While this framework is quite general, its verification requires detailed spectral information and delicate estimates that are not readily available for the coupled momentum–temperature system \eqref{model}. We therefore do not pursue this route. Thus, none of the classical routes to center-manifold reduction is available, and a genuinely new approach is required.
Our main results concern the existence, structure, exponential attractivity of the invariant manifold and bifurcation analysis for the system \eqref{model}-\eqref{bianjian0814}.
\begin{theorem}\label{zhongxinliuxing07276}System \eqref{model}-\eqref{bianjian0814} possesses a locally invariant manifold
$$
M_{l}\left(R\right)=\left\{\xi+h\left(\xi\right)|\xi\in \mathbf{X}_{c},~\left\|\xi\right\|_{\mathbf{X}_{\alpha}}~\text{is~small~sufficiently}\right\},
$$
 where $\mathbf{X}_{\alpha}=\mathbf{H}_{\alpha}\times H_{\alpha}$ is a fractional Sobolev space defined in \eqref{chouxiang0928}, $\frac{1}{2}<\alpha<1$, $\mathbf{X}_{c}$ and $\mathbf{X}_{s}$ are the critical and stable subspaces of $\mathbf{L}$, respectively and $h:~\mathbf{X}_{c}\rightarrow\mathbf{X}_{s}$ is the invariant manifold function defined in \eqref{hdedansheng1215}. 
 Specifically, for any initial value $\Phi_{0}\in M_{l}\left(R\right)$, there exists a $T=T\left(\Phi_{0}\right)>0$, such that 
the solution $\Phi\left(t;\Phi_{0}\right)$ to the system \eqref{model}-\eqref{bianjian0814} remains on $ M_{l}\left(R\right)$ as $t<T$. Additionally, for any initial value $\Phi_{0}$, if  $\left\|\Phi\left(t;\Phi_{0}\right)\right\|_{\mathbf{X}_{\alpha}}$ is sufficiently small for all $t>0$, the solution $\left\|\Phi\left(t;\Phi_{0}\right)\right\|_{\mathbf{X}_{\alpha}}$ converges to $M_{l}\left(R\right)$  exponentially as $t\rightarrow+\infty$.
\end{theorem}
By virtue of center-manifold reduction, we obtain the following conclusions.
\begin{theorem}\label{fenqi0729}
For system \eqref{model}-\eqref{bianjian0814} with 
$R$ in a neighborhood of the critical value 
  defined in \eqref{linjie0926}, the following conclusions hold.
\begin{enumerate}[label=(\arabic*)]
    \item 
    When the first eigenvalue of 
$\mathbf{L}$ is simple, the bifurcation is of pitchfork type. Whether the bifurcation is supercritical or subcritical is determined by the sign of the nondimensional coefficient 
$a$ defined in \eqref{diyige0210}, as presented in \autoref{diyidingli0211}.
    \item When the first eigenvalue of $\mathbf{L}$ is double, the bifurcation is again of pitchfork type. Its direction—supercritical or subcritical— is determined by the signs and relative magnitudes of the four nondimensional coefficients $a_{11}$, $a_{22}$, $b_{11}$ and $b_{22}$ defined in \eqref{xishu0210}, as described in \autoref{dierdingli0211}.
\end{enumerate}    
\end{theorem}

\autoref{zhongxinliuxing07276} follows from \autoref{yashuo0727} and \autoref{bijin0727}. In the classical setting, the invariant manifold function $h$ is obtained by constructing a contraction mapping $\mathbf{F}$. When the semigroup generated by the linear operator is analytic—or, alternatively, when the nonlinearity is assumed to be local Lipschitz—the contraction property follows in a standard way. In the present system, however, neither of these conditions holds. To overcome this difficulty, we prove an auxiliary inequality
\begin{align}\label{houmianxuyao0729}
\sup\limits_{t\leq 0}e^{\gamma t}
\left\|\theta\right\|_{H_{\frac{3}{2}}}\leq C\left(R,r\right)\sup\limits_{t\leq 0}e^{\gamma t}\left\|\mathbf{u}\right\|_{\mathbf{H}_{\alpha}},
\end{align}
whose validity relies on the fact that the temperature subsemigroup—generated by the Laplacian—is analytic. Indeed, for initial data in $\mathbf{X}_{\alpha}$, the temperature is not only bounded in $H_{\alpha}$
  but is actually upgraded to $H_{\frac{3}{2}}$
  regularity which is a key step to obtain \eqref{houmianxuyao0729}. Inequality \eqref{houmianxuyao0729} allows us to establish Proposition \ref{babay0120}, which in turn implies the existence of contraction mapping. Concerning the attractivity of the invariant manifold, the classical proof proceeds via an inequality 
\begin{align}\label{guanjianbudnegshi0729}
    \begin{aligned}
        \left\|v\left(t\right)\right\|_{\mathbf{X}_{\alpha}}
        \leq Ce^{-\gamma_{0} t}+
        Cr\int_{0}^t e^{-\gamma_{0}\left(t-\tau\right)}\left\|v\left(\tau\right)\right\|_{\mathbf{X}_{\alpha}}d\tau,
\end{aligned}
\end{align}
where $v\left(t\right)=\Phi_{s}\left(t\right)-h\left(\Phi_{c}\left(t\right)\right)$, $\Phi_{s}\left(t\right)\in \mathbf{X}_{s}$ and $\Phi_{c}\left(t\right)\in \mathbf{X}_{c}$, $\gamma_{0}>0$ is a constant defined in \eqref{zuida1212} and $r>0$ is a sufficiently small constant. Inequality \eqref{guanjianbudnegshi0729} follows either from analyticity of the full semigroup $e^{\mathcal{L}t}$ or from locally Lipschitz assumption on the nonlinearity $\mathcal{G}$ since the following inequality holds:
\begin{align}\label{zaiyi0822}
    \left\|\int_{0}^{t}e^{\mathcal{L}\left(t-\tau\right)}\left(\mathcal{G}\left(w_{1}\right)-\mathcal{G}\left(w_{2}\right)\right)d\tau\right\|_{\mathbf{X}_{\alpha}}\leq C \int_{0}^{t}e^{-\gamma_{0}\left(t-\tau\right)}\left\|w_{1}-w_{2}\right\|_{\mathbf{X}_{\alpha}}d\tau.
\end{align}
In the present setting, such an estimate \eqref{zaiyi0822} is not directly available. To circumvent this, we find that in system \eqref{model} there holds
$$
\int_{0}^{t}T\left(t-\tau\right)\left\|G\left(\Phi\right)-G\left(\widetilde{\Phi}\right)\right\|_{\mathbf{X}_{\alpha}}d\tau\leq \int_{0}^{t}C\left(t,\tau\right)\left\|\Phi-\widetilde{\Phi}\right\|_{\mathbf{X}_{\alpha}}d\tau,
$$
where $T\left(t\right)$ is the strongly continuous semigroup generated by $\mathbf{L}$ and $C\left(t,\tau\right)>0$ depends on $t$ and $\tau$ if $\left\|\theta\right\|_{H_{\alpha+\frac{1}{2}}}$ and $\left\|\widetilde{\theta}\right\|_{\alpha+\frac{1}{2}}$ are bounded and controlled by $\left\|\Phi\right\|_{\mathbf{H}_{\alpha}}$ and $\left\|\widetilde{\Phi}\right\|_{\mathbf{H}_\alpha}$. Thus,
we derive two auxiliary estimates--Lemmas \ref{xuyaode1231} (provides the boundedness of $\left\|\theta\right\|_{H_{\alpha+\frac{1}{2}}}$ in terms of $\left\|\Phi\right\|_{\mathbf{X}_{\alpha}}$) and \ref{xuyaode0116} (controls the difference  $\left\|\theta-\widetilde{\theta}\right\|_{H_\alpha+\frac{1}{2}}$ by $\left\|\Phi-\widetilde{\Phi}\right\|_{\mathbf{X}_{\alpha}}$)--again exploiting the analytic smoothing of the temperature subsemigroup. These estimates ultimately yield the desired inequality \eqref{guanjianbudnegshi0729}, and hence the exponential attraction to the invariant manifold. Our approach provides a general mechanism for proving the structure and attractivity of center-manifold in systems where only a subsemigroup is analytic—a feature that may be of independent interest beyond the DBC problem.

 This paper is organized as follows. Section \ref{dierjie0204} presents the functional setting and the spectral properties of the linear operator.  Section \ref{jie0205} establishes the well-posedness of system \eqref{model}. Section \ref{buxian0205} proves  the existence of invariant manifold, its exponential attracting property and gives its approximation formula. Section \ref{yuehua0205} applies the center-manifold reduction to derive the bifurcation theorem. Section \ref{shuzhi0205} presents numerical simulations and Section \ref{jielun0719} provides concluding remarks.
\section{Preliminaries}\label{dierjie0204}
\subsection{The abstract mathematical model}\label{2.1xiaojie0204}


In what follows, we reformulate the system of equations given by \eqref{model}–\eqref{bianjian0814} within an abstract functional framework, which provides a convenient setting for investigating the system from the perspective of dynamical transitions. First, define 
\begin{align}\label{kongjian0211}
    \begin{aligned}
        &\mathcal{D}_{1}=\left\{\mathbf{u}\in \left[C^{\infty}\left(\Omega\right)\right]^2|\nabla\cdot\mathbf{u}=0,~\mathbf{u}\cdot\mathbf{n}|_{\partial\Omega}=\frac{\partial u_{1}}{\partial x_{2}}|_{x_{2}=0,1}=\frac{\partial u_{2}}{\partial x_{1}}|_{x_{1}=0,L}=0\right\},
        \\
        &\mathcal{D}_{2}=\left\{\theta\in C^{\infty}\left(\Omega\right)|\theta|_{x_{2}=0,1}=\nabla\theta\cdot\mathbf{n}|_{x_{1}=0,L}=0\right\}.
    \end{aligned}
\end{align}

 $L^2\left(\Omega\right)$, $H^{1}\left(\Omega\right)$, and $H^{2}\left(\Omega\right)$ denote the standard Lebesgue space and Sobolev spaces, respectively.  
Furthermore, let $\mathbf{E}_{1}$ and $\widetilde{\mathbf{E}}_{1}$ be the closure of $\mathcal{D}_{1}$ with respect to the $L^2$-norm and $H^{1}$-norm, respectively. Similarly, let $\mathbf{E}_{2}$ and $\widetilde{\mathbf{E}}_{2}$ be the closure of $\mathcal{D}_{2}$  with respect to the $H^1$-norm and $H^{2}$-norm, respectively. Then, according to \cite{Temam2013} and Trace Theorem \cite{Evans2010,Galdi2011},  
\begin{align}\label{lingwai0211}
    \begin{cases}
        \mathbf{E}_{1}
        =\left\{\mathbf{u}
        \in \left[L^2\left(\Omega\right)\right]^2|
\nabla\cdot\mathbf{u}=0~\text{and~}\mathbf{u}\cdot\mathbf{n}|_{\partial\Omega}=0~\text{in~distributional~sense}
        \right\},
        \\
\widetilde{\mathbf{E}}_{1}
        =\left\{\mathbf{u}
        \in \left[H^1\left(\Omega\right)\right]^2|
\nabla\cdot\mathbf{u}=0~\text{and~}\mathbf{u}\cdot\mathbf{n}|_{\partial\Omega}=0
        \right\},
        \\
\mathbf{E}_{2}=
        \left\{\theta\in H^{1}\left(\Omega\right)|\theta|_{x_{2}=0,1}=0\right\},
        \\
 \widetilde{\mathbf{E}}_{2}=\{\theta\in H^2\left(\Omega\right)|\theta|_{x_{2}=0,1}=\nabla\theta\cdot\mathbf{n}|_{x_{1}=0,L}=0\}.
    \end{cases}
\end{align}

Then,
 defining \(\Phi=\left(\mathbf{u},\theta\right)\), we give the Hilbert space as follows
 \begin{align}\label{kongjian0120}
        \mathbf{H}=\left\{\Phi\in [L^2(\Omega)]^3|\nabla\cdot\mathbf{u}=0\text{~and~}
\mathbf{u}\cdot\mathbf{n}|_{\partial\Omega}=0~\text{in~distributional~sense}\right\}.
\end{align}
All Hilbert spaces \eqref{lingwai0211}-\eqref{kongjian0120} are endowed with their natural inner products.

Next,
we introduce the linear operator $\mathbf{L}:\mathbf{E}_{1}\times\widetilde{\mathbf{E}}_{2}\rightarrow \mathbf{H}$ and nonlinear operator \(G:\widetilde{\mathbf{E}}_{1}\times\widetilde{\mathbf{E}}_{2}\rightarrow \mathbf{H}\) as follows: $\forall \Phi=(\mathbf{u},\theta)$ and 
\(\widetilde{\Phi}=\left(\widetilde{\mathbf{u}},\widetilde{\theta}\right)\), 
\begin{align}\label{juyou0205}
    \begin{aligned}
    \mathbf{L}\Phi=\begin{pmatrix}
        -V_{a}\mathbf{u}+V_{a}\mathbf{P}(R\theta \mathbf{e}_{2})
        \\
        \Delta\theta+Ru_{2}
    \end{pmatrix},~
    G\left(\Phi,\widetilde{\Phi}\right)
    =\begin{pmatrix}
        \mathbf{0}
        \\
        -\mathbf{u}\cdot\nabla\widetilde{\theta}
    \end{pmatrix},
\end{aligned}
\end{align}
where $\mathbf{P}:\left[L^2\left(\Omega\right)\right]^2\rightarrow \left[L^2\left(\Omega\right)\right]^2$ is the Leray projection operator. 

With the help of the operators $\mathbf{L}$ and $G$, the system 
\eqref{model}-\eqref{bianjian0814} can be rewritten in forms
\begin{align}\label{suanzixing0708}
\begin{cases}
\frac{d\Phi}{dt}=\mathbf{L}\Phi+
\mathbf{G}\left(\Phi\right),~\Phi\in \widetilde{\mathbf{E}}_{1}\times \widetilde{\mathbf{E}}_{2}
\\
\Phi\left(t_{0}\right)=\Phi_{0}=\left(
\mathbf{u}_{0},\theta_{0}\right),
\end{cases}
\end{align}
where \(\mathbf{G}\left(\Phi\right)
=G\left(\Phi,\Phi\right)\) and 
\(\Phi\left(t_{0}\right)\) is the initial data.

\subsection{The eigenvalue problem}\label{puwenti0204}
This section proves that the linear operator $\mathbf{L}$ defined in \eqref{juyou0205} has a complete eigenvector system, using the completeness of the eigenvector systems of $\mathbf{P}\left(\Delta,\Delta\right)^{T}$ and $\Delta$. Specifically, for the following two eigenvalue-problems:
\begin{align}\label{fang10205}
    \begin{aligned}
        \mathbf{P}\Delta \mathbf{u}=\beta \mathbf{u},~\nabla\cdot\mathbf{u}=\mathbf{u}\cdot\mathbf{n}|_{\partial\Omega}=\frac{\partial u_{1}}{\partial x_{2}}|_{x_{2}=0,1}=\frac{\partial u_{2}}{\partial x_{1}}|_{x_{1}=0,L}=0,
    \end{aligned}
\end{align}
and
\begin{align}\label{fang20205}
    \begin{aligned}
        \Delta\theta=\widetilde{\beta}\theta, ~\theta|_{x_{2}=0,1}=\nabla\theta\cdot\mathbf{n}|_{x_{1}=0,L}=0.
    \end{aligned}
\end{align}
By the spectral theory of compact symmetric linear operators \cite{Evans2010}, operators $\mathbf{P}\left(\Delta,\Delta\right)^{T}$ and $\Delta$  have complete eigenvector systems, whose eigenvectors are directly given by calculation as follows:
\begin{align}\label{diderzu0814}
\mathbf{e}_{kl}=\left(l\pi\sin{\frac{k\pi x_{1}}{L}}\cos{l\pi x_{2}},-\frac{k\pi}{L}\cos{\frac{k\pi x_{1}}{L}}\sin{l\pi x_{2}}\right), ~k,l=1,2,\cdots,
\end{align}
and 
\begin{align}\label{tezhenxiangliang0926}
e_{kl}=\cos{\frac{k\pi x_{1}}{L}}\sin{l\pi x_{2}},~k=0,1,2,\cdots,l=1,2,\cdots.
\end{align}
Thus, for any $\mathbf{u}\in \mathbf{E}_{1}$ and $\theta\in L^2\left(\Omega\right)$ satisfying $\theta|_{x_{2}=0,1}=\nabla\theta\cdot\mathbf{n}|_{x_{1}=0,L}=0$, we can rewritten them in forms:
\begin{align}\label{jishu0205}
    \mathbf{u}=\sum\limits_{k,l=1}^{+\infty}u_{kl}\mathbf{e}_{kl},~\theta=\sum\limits_{k=0,l=1}^{+\infty}\theta_{kl}e_{kl}.
\end{align}
Subsequently, we consider the following eigenvalue problem:
\begin{align*}
    \begin{cases}
        -V_{a}\mathbf{u}+V_{a}\mathbf{P}(R\theta\mathbf{e}_{2})=\lambda \mathbf{u},
        \\
    \Delta\theta+Ru_{2}=\lambda\theta,
        \\
        \nabla\cdot\mathbf{u}=0,
        \\
\Phi=\left(\mathbf{u},\theta\right)~\text{satisfies~}\eqref{bianjian0814}.
    \end{cases}
\end{align*}
which is equivalent to 
\begin{align}\label{dengjia0708}
    \begin{cases}
        -V_{a}\frac{\partial p}{\partial x_{1}}=(\lambda+V_{a})u_{1},~
        -V_{a}\frac{\partial p}{\partial x_{2}}+V_{a}R\theta=(\lambda+V_{a})u_{2},
        \\
        \Delta\theta+Ru_{2}=\lambda \theta,
        ~
        \frac{\partial u_{1}}{\partial x_{1}}+\frac{\partial u_{2}}{\partial x_{2}}=0,
        \\
\mathbf{u}\cdot\mathbf{n}|_{\partial\Omega}=0,~
\theta|_{x_{2}=0,1}=0,~\nabla\theta\cdot\mathbf{\mathbf{n}}|_{x_{1}=0,L}=0.
    \end{cases}
\end{align}
Substituting \eqref{jishu0205} into \eqref{dengjia0708}, we have the following two conclusions: 

(1) when $k\neq 0$,
the characteristic polynomial is as follows,
\begin{align}\label{danshi0814}
\lambda^2+\left(V_{a}+r_{kl}^2\right)\lambda 
+V_{a}r_{kl}^2-\frac{V_{a}R^2k^2\pi^2}{L^2r_{kl}^2}=0,
\end{align}
where \(r_{kl}^2=\frac{k^2\pi^2}{L^2}+l^2\pi^2\).
A simple calculation gives that 
\(
\Delta_{kl}=\left(V_{a}-r_{kl}^2\right)^2
+\frac{4V_{a}R^2k^2\pi^2}{L^2r_{kl}^2}>0,
\)
which indicates that 
all eigenvalues of \(\mathbf{L}\) are real.  We also obtain the eigenvalues and corresponding 
eigenvectors as follows
\begin{align}\label{tezhenzhi0814}
\lambda_{k,l}^1=\frac{-\left(V_{a}+r_{kl}^2\right)
+\sqrt{\Delta_{kl}}}{2},~
\lambda_{k,l}^2=\frac{-\left(V_{a}+r_{kl}^2\right)
-\sqrt{\Delta_{kl}}}{2},
\end{align}
and 
\begin{align}\label{tehzenxiangliang0814}
\begin{aligned}
\Phi_{k,l}^{j}=\frac{2}{\sqrt{\left[1+\left(u_{k,l}^j\right)^2r_{kl}^2\right]L}}
\begin{pmatrix}
    u_{k,l}^{j}l\pi \sin{\frac{k\pi x_{1}}{L}}\cos{l\pi x_{2}},
    \\
    -u_{k,l}^{j}\frac{k\pi}{L} \cos{\frac{k\pi x_{1}}{L}}\sin{l\pi x_{2}},
    \\
    \cos{\frac{k\pi x_{1}}{L}}\sin{l\pi x_{2}}
\end{pmatrix},
p_{k,l}^{j}=\widetilde{p}_{k,l}^{j}\cos{\frac{k\pi x_{1}}{L}}\cos{l\pi x_{2}},
\end{aligned}
\end{align}
where \(k,l=1,2,\cdots,\) \(u_{k,l}^{j}=-\frac{\left(\lambda_{kl}^{j}+r_{kl}^2\right)L}{Rk\pi}\), \(\widetilde{p}_{kl}^{j}=\frac{\left(\lambda_{kl}^{j}+V_{a}\right)u_{kl}^{j}lL}{V_{a}k}\) and 
\(j=1,2\).

(2) when $k=0$, the eigenvalues $\lambda_{0l}$ and corresponding eigenvectors are as follows,
\begin{align}\label{ling0925}
    \begin{aligned}
        \lambda_{0,l}=-l^2\pi^2,~\Phi_{0l}=\begin{pmatrix}
            0\\
            0\\
            \frac{\sqrt{2}}{\sqrt{L}}\sin{l\pi x_{2}}
        \end{pmatrix},
        ~p_{0,l}=-\frac{\sqrt{2}R}{\sqrt{L}l\pi}\cos{l\pi x_{2}},
    \end{aligned}
\end{align}
where $l=1,2,3,\cdots$.

The following lemma shows that the linear operator $\mathbf{L}$ has a complete eigenvector system. That is, for any $\Phi\in \mathbf{E}_{1}\times \widetilde{\mathbf{E}}_{2}$, $\Phi$ can be expressed by the eigenvectors of $\mathbf{L}$.
\begin{lemma}\label{wanbei0205}
    All eigenvectors $\left\{\Phi_{k,l}^{1},\Phi_{k,l}^{2},\Phi_{0,l}\right\}$, where $k,l=1,2,\cdots,$ ,of $\mathbf{L}$ forms a basis for the space $\mathbf{E}_{1}\times \widetilde{\mathbf{E}}_{2}$.
\end{lemma}
\begin{proof}
    From \eqref{jishu0205}, 
    $$
    \Phi=\begin{pmatrix}
        \sum\limits_{k,l=1}^{+\infty}u_{kl}\mathbf{e}_{kl}
        \\
\sum\limits_{k=0,l=1}^{+\infty}\theta_{kl}e_{kl}
    \end{pmatrix},
    $$
    which can be expressed by $\left\{\Phi_{k,l}^{1},\Phi_{k,l}^{2},\Phi_{0,l}\right\}$, where $k,l=1,2,\cdots$, through a direct computation. 

    This completes the proof.
\end{proof}
\begin{remark}\label{fangbian0205}
    Since the eigenvectors of $\Phi_{k,l}^{1}$ and $\Phi_{k,l}^2$, where $k,l=1,2,\cdots$,  are not necessarily orthogonal, we additionally compute the following two vectors:
    \begin{align*}
        \widetilde{\Phi}_{k,l}^{1}=\begin{pmatrix}
           \frac{2\sqrt{\left[1+\left(u_{k,l}^1\right)^2\gamma_{kl}^2\right]L}}{l\pi L\left(u_{k,l}^1-u_{k,l}^2\right)}\sin{\frac{k\pi x_{1}}{L}}\cos{l\pi x_{2}}
           \\
           0
           \\
           \frac{2u_{k,l}^{2}\sqrt{\left[1+\left(u_{k,l}^1\right)^2\gamma_{kl}^2\right]L}}{ L\left(u_{k,l}^2-u_{k,l}^1\right)}\cos{\frac{k\pi x_{1}}{L}}\sin{l\pi x_{2}}
        \end{pmatrix},~
        \widetilde{\Phi}_{k,l}^{2}=\begin{pmatrix}
            \frac{2\sqrt{\left[1+\left(u_{k,l}^2\right)^2\gamma_{kl}^2\right]L}}{l\pi L\left(u_{k,l}^2-u_{k,l}^1\right)}\sin{\frac{k\pi x_{1}}{L}}\cos{l\pi x_{2}}
           \\
           0
           \\
           \frac{2u_{k,l}^{1}\sqrt{\left[1+\left(u_{k,l}^2\right)^2\gamma_{kl}^2\right]L}}{ L\left(u_{k,l}^1-u_{k,l}^2\right)}\cos{\frac{k\pi x_{1}}{L}}\sin{l\pi x_{2}}
        \end{pmatrix}.
    \end{align*}
    One can easily check that 
    \begin{align*}
        \widetilde{\Phi}_{p,q}^{j}\cdot\Phi_{k,l}^{s}=\begin{cases}
            1,~(p,q,j)=(k,l,s),
            \\
            0,~(p,q,j)\neq(k,l,s).
        \end{cases}
    \end{align*}
\end{remark}
\subsection{The principle of exchange of stability}\label{chongyaotiaojian0205}
In this section, we present a sufficient condition for the occurrence of transition, referred to as the principle of exchange of stability. To this end, from $\lambda_{k,l}^{1}$ in \eqref{tezhenzhi0814}  we introduce the critical number 
\begin{align}\label{linjie0926}
R_{c}=\min\limits_{k,l=1,2,\cdots}F(k,l),
\end{align}
where $F(k,l)=L\pi \left(\frac{k}{L^2}+\frac{l^2}{k}\right)$. Obviously, $F(k,l)$ increases with $l$. Thus, 
\begin{align*}
    R_{c}=\min\limits_{k=1,2,\cdots}F(k,1).
\end{align*}
A simple calculation reveals that there exists at least one and at most two critical indices $k_{0}$ satisfying $R_{c}=F(k_{0},1)$. Specifically, these critical indices $k_{0}$ correspond to $[L]$ or $[L]+1$ where $[L]$ denotes the greatest integer less than or equal to $L$ (i.e., the floor function of $L$).

\begin{figure}[htp]
		\begin{minipage}[t]{0.45\linewidth}
			\centering
\includegraphics[width=1.1\linewidth]{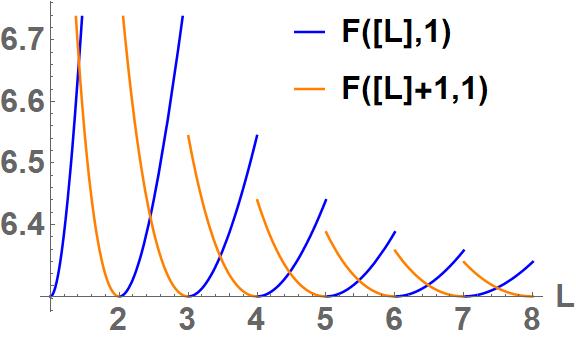}
		\end{minipage}
		\hfill
		\begin{minipage}[t]{0.5\linewidth}
			\centering
\includegraphics[width=0.95\linewidth]{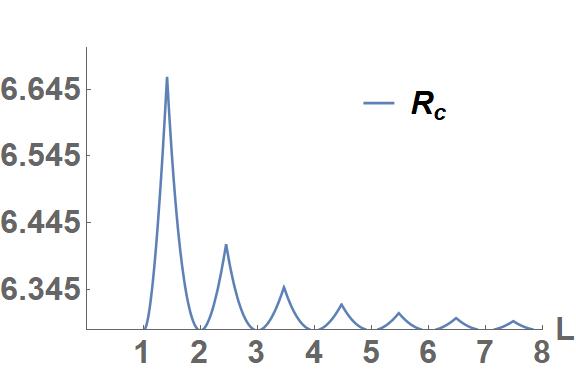}
		\end{minipage}
        \caption{ the figures of $F([L],1)$ and $F([L]+1,1)$ (left), the figure of $R_{c}$(right), as $L\in [1,8]$.}
        \label{zhaolinjiezhi0318}
	\end{figure}
\begin{remark}\label{remark0723}
  A direct computation shows that the first eigenvalue of the linear operator $\mathbf{L}$
  is double if and only if
  
   $$L=\sqrt{n(n+1)},~\text{for}~n\in\mathbf{N}^{*}.$$
   In this case, we have
    $$
    R_{c}=F(n,1)=F(n+1,1)=\sqrt{n(n+1)}\pi \left(\frac{1}{n}+\frac{1}{n+1}\right).
    $$
    Consequently, the set of $\mathbf{L}$ for which the first eigenvalue is double is given by 
    $$
    \mathcal{S}=\left\{\sqrt{n\left(n+1\right)}\bigg{|}n\in\mathbf{N}^{*}\right\},
    $$
    is a set of measure zero in $R^{+}$.
\end{remark}

\begin{lemma}\label{tadepu0925}
For \(R\) in the vicinity of \(R_{c}\), we have 
\begin{enumerate}
    \item[\rm(1)] when \(0<R<R_{c}\), 
    all eigenvalues of \(\mathbf{L}\) are negative; 
    \item[\rm(2)] when \(R=R_{c}\), 
    there is one or two zero eigenvalues of \(\mathbf{L}\);
    \item[\rm(3)] when \(R>R_{c}\), there is one or two positive eigenvalues of \(\mathbf{L}\);
    \item[\rm(4)] when \(k\) or \(l\) \(\rightarrow+\infty\), we have
    \(\lambda^{1}_{k,l}\rightarrow-V_{a}\), \(\lambda^{2}_{k,l}\rightarrow-\infty\) and as \(l\rightarrow+\infty\), we have \(\lambda_{0,l}\rightarrow-\infty\).
\end{enumerate}
\end{lemma}
\begin{proof}
    The first three conclusions can be directly derived by analyzing expressions \eqref{tezhenzhi0814} and \eqref{linjie0926}. Let $k$ or $l$ go to infinity in $\lambda_{k,l}^1$, $\lambda_{k,l}^{2}$ and $\lambda_{0,l}$, we can obtain the last conclusion.
\end{proof}

Lemma \ref{tadepu0925} says that as 
\(R\) passes through \(R_{c}\), there exist finite eigenvalues changing signs. Therefore, we call these eigenvalues the first eigenvalues and denote the corresponding critical index by 
\begin{align}\label{linjiezhibiaoji0709}
B=\left\{\left(k_{0},1\right)|\left(k_{0},1\right)=\left(k_{1},1\right),\cdots, 
\left(k_{m},1\right)\right\}.
\end{align}
When \(R\) is close to \(R_{c}\)
, one can see that \(m=1\) or \(2\), that is, Card(\(B\))\(=1\) or \(2\). Furthermore, the first eigenvectors are denoted by \(\left\{\Phi_{k_{1},1}^{1},\cdots,\Phi_{k_{m},1}^{1}\right\}\). In addition, we readily deduce that there exists a number $\gamma_{0}>0$ such that 
\begin{align}\label{zuida1212}
\lambda_{k,l}^{1},\lambda_{k,l}^{2},\lambda_{0,l}\leq-\gamma_{0}
\end{align}
for $\lambda_{k,l}^{1}\neq \lambda_{k_{0},1}^{1}$.
\section{The well-posedness to system \eqref{model}–\eqref{bianjian0814}}\label{jie0205}
This section investigates several properties of the solutions to the equation, which form the foundation for the existence of locally invariant manifolds of the system \eqref{model}–\eqref{bianjian0814}. The first subsection presents the existence theorem of solutions without proof; the second subsection demonstrates that the linear operator $\mathbf{L}$ can generate a strongly continuous semigroup and provides its specific action form; the third subsection improves the regularity of solutions by virtue of the operator semigroup, where the strategy for regularity improvement is implemented separately for the velocity field and the temperature field.
\subsection{The existence and uniqueness of weak solution}\label{diyibu0205}
We first give the definition of weak solution to the system \eqref{model}–\eqref{bianjian0814} as follows,
\begin{definition}
    $\Phi=\left(\mathbf{u},\theta\right),$ where $\mathbf{u}\in L^{2}\left(t_{0},T;\mathbf{E}_{1}\right)$, $\theta\in L^{2}\left(t_{0},T;\mathbf{E}_{2}\right)$ and $T>t_{0}$ is a arbitrary constant, is the weak solution to the system \eqref{model}–\eqref{bianjian0814}, if $\forall\widetilde{\Phi}=\left(\widetilde{\mathbf{u}},\widetilde{\theta}\right)$, where $\widetilde{\mathbf{u}}=\left(\widetilde{u}_{1},\widetilde{u}_{2}\right)\in C^{1}\left(t_{0},T;\mathcal{D}_{1}\right)$, $\widetilde{\theta}\in C^{1}\left(t_{0},T;\mathcal{D}_{2}\right)$ and $\widetilde{\Phi}\left(T\right)=0$, we have 
    \begin{align*}
        \begin{aligned}
            -&\int_{\Omega}\Phi_{0}\cdot\widetilde{\Phi}\left(t_{0}\right)dx-\int_{t_{0}}^{T}\int_{\Omega}\Phi\cdot\partial_{t}\widetilde{\Phi}dxdt
            -\int_{t_{0}}^{T}\int_{\Omega}\mathbf{u}\cdot\nabla\widetilde{\theta}\theta dxdt
            \\
            &=
           \int_{t_{0}}^{T}\int_{\Omega}\left[
           -V_{a}\mathbf{u}\cdot\widetilde{\mathbf{u}}+V_{a}R\theta\widetilde{u}_{2}-\nabla\theta\cdot\nabla\widetilde{\theta}+Ru_{2}\widetilde{\theta}\right]dxdt,
        \end{aligned}
    \end{align*}
    where $\Phi_{0}=\left(\mathbf{u}_{0},\theta_{0}\right)$ is the initial value.
\end{definition}

Using the Galerkin approximation method, we can obtain that 
\begin{lemma}\label{shidingxing0111} For any \(\mathbf{u}_{0}\in \widetilde{\mathbf{E}}_{1}\) and 
\(\theta_{0}\in \mathbf{E}_{2}\), then the weak solution \(\left(\mathbf{u},\theta\right)\) to the system \eqref{model}-\eqref{bianjian0814} with the initial value \(\left(\mathbf{u}_{0},\theta_{0}\right)\) is unique and have the following properties, 
\begin{align}\label{xingzhi0111}
\begin{aligned}
&~~~~~~~~~~~~~~\mathbf{u}\in L^{\infty}\left(t_0,T;\widetilde{\mathbf{E}}_{1}\right),~
\partial_{t}\mathbf{u}\in L^{2}\left(t_0,T;\mathbf{E}_{1}\left(\Omega\right)\right),
\\
&\theta\in L^{\infty}\left(t_0,T; H^{1}\left(\Omega\right)\right)
\cap L^{2}\left(t_0,T;H^{2}\left(\Omega\right)\right),~\partial_{t}\theta \in L^{2}\left(t_0,T;L^{2}\left(\Omega\right)\right).
\end{aligned}
\end{align}
\end{lemma}
\begin{proof}
The Galerkin approximation method is standard, thus we omit it. We just give the proof of uniqueness. Assume that $\left(\mathbf{u},\theta\right)$ and $\left(\widetilde{\mathbf{u}},\widetilde{\theta}\right)$ are two solutions to the system \eqref{model}-\eqref{bianjian0814} with the same initial value $\left(\mathbf{u}_{0},\theta_{0}\right)\in \widetilde{\mathbf{E}}_{1}\times \mathbf{E}_{2}$. Let $\mathbf{U}=\mathbf{u}-\widetilde{\mathbf{u}}=\left(U_{1},U_{2}\right)$ and $\Theta=\theta-\widetilde{\theta}$, then by standard energy method, the following estimate is obtained 
\begin{align}\label{gujiyixia0930}
\begin{aligned}
\frac{1}{2}\frac{d}{dt}\left[\left\|\Theta\right\|_{L^2\left(\Omega\right)}^2+\left\|\mathbf{U}\right\|_{L^2\left(\Omega\right)}^2\right]&
+V_{a}\left\|\mathbf{U}\right\|_{L^2\left(\Omega\right)}^2
+\left\|\nabla\Theta\right\|_{L^2\left(\Omega\right)}^2
\\
&=\left(V_{a}R+R\right)\int_{\Omega}U_{2}\Theta dx+\int_{\Omega}\mathbf{U}\cdot\nabla\widetilde{\theta}\Theta dx
\\
&\leq M(t)\left[\left\|\Theta\right\|_{L^2\left(\Omega\right)}^2+\left\|\mathbf{U}\right\|_{L^2\left(\Omega\right)}^2\right]+\frac{1}{2}\left\|\nabla\Theta\right\|_{L^2\left(\Omega\right)}^2,
\end{aligned}
\end{align}
where $M(t)=\frac{V_{a}R+R}{2}+\frac{1}{2}\left\|\nabla\widetilde{\theta}\right\|_{L^4\left(\Omega\right)}^2$ is integrable. With Gronwall's inequality and $\left(\mathbf{U},\Theta\right)|_{t=t_{0}}=\left(\mathbf{0},0\right)$, it follows that 
\begin{align*}
\left\|\Theta\right\|_{L^2\left(\Omega\right)}^2+\left\|\mathbf{U}\right\|_{L^2\left(\Omega\right)}^2=0,
\end{align*}
which implies the uniqueness.
\end{proof}
\begin{remark}
    Furthermore, if $\theta_{0}\in \widetilde{\mathbf{E}}_{2}\left(\Omega\right)$, the above weak solution $\Phi=\left(\mathbf{u},\theta\right)$ satisfies the following properties,
    \begin{align*}
        \begin{aligned}
            &\partial_{t}\Phi\in L^{\infty}\left(t_{0},T;\left[H^{1}\left(\Omega\right)\right]^{3}\right);~\partial_{t}\theta\in L^{2}\left(t_{0},T;H^{1}\left(\Omega\right)\right);
            \\
            &\theta\in L^{2}\left(t_{0},T;W^{2,q}\left(\Omega\right)\right)\cap L^{\infty}\left(t_{0},T;H^{2}\left(\Omega\right)\right),~1\leq q<\infty.
        \end{aligned}
    \end{align*}
\end{remark}

\subsection{The strong semigroup generated by $\mathbf{L}$}\label{banqun0205}

In order to obtain the semigroup $T(t)$ generated by $\mathbf{L}$, we consider the following linear problem 
\begin{align}\label{xianxingsuanzi0928}
\begin{cases}
\frac{d\Phi}{dt}=\mathbf{L}\Phi,
\\
\Phi\left(t_{0}\right)=\Phi_{0}=\left(
\mathbf{u}_{0},\theta_{0}\right)\in\mathbf{H}.
\end{cases}
\end{align}
From previous discussion, we conclude that 
$$\Phi_{0}=\sum\limits_{k,l\in\mathbf{N}^{*}}
\left[
C^{1}_{k,l}\left(0\right)\Phi^{1}_{k,l}+C^{2}_{k,l}\left(0\right)\Phi^{2}_{k,l}+C_{0,l}\left(0\right)\Phi_{0,l}
\right],$$ 
where $C_{k,l}^{i}(0)=\int_{\Omega}\Phi_{0}\cdot \widetilde{\Phi}_{k,l}^{i}dx$ ($i=1,2$) and $C_{0,l}(0)=\int_{\Omega}\Phi_{0}\cdot\widetilde{\Phi}_{0,l}dx$. Then, the above linearized problem can be transformed as ODEs 
\begin{align*}
\begin{aligned}
\frac{dC^{1}_{k,l}}{dt}=\lambda^{1}_{k,l}C^{1}_{k,l},~
\frac{dC^{2}_{k,l}}{dt}=\lambda^{2}_{k,l}C^{2}_{k,l},~
\frac{dC_{0,l}}{dt}=\lambda_{0,l}C_{0,l}.
\end{aligned}
\end{align*}
Therefore, the solution to linearized problem is 
\begin{align}\label{zuoyongfangshi0709}
T\left(t\right)\Phi_{0}=
\sum\limits_{k,
l\in\mathbf{N}^{*}}
\left[
C^{1}_{k,l}\left(0\right)e^{\lambda^{1}_{k,l}t}\Phi^{1}_{k,l}+C^{2}_{k,l}\left(0\right)e^{\lambda^{2}_{k,l}t}\Phi^{2}_{k,l}+C_{0,l}\left(0\right)e^{\lambda_{0,l}t}\Phi_{0,l}
\right]\in\mathbf{H}.
\end{align}
Then according to the definition of strongly continuous semigroup of operators \cite{Pazy1983} , it is easy to verify that $T(t):\mathbf{H}\rightarrow\mathbf{H}$ is a strongly continuous semigroup of operators generated by $\mathbf{L}$. From \cite{ma2011} and \eqref{xingzhi0111}, we conclude that the solution $\Phi=\left(\mathbf{u},\theta\right)$ to the system \eqref{model}-\eqref{bianjian0814} with initial value $\Phi_{0}\in \widetilde{\mathbf{E}}_{1}\times \mathbf{E}_{2}$ can be expressed in the form 
\begin{align}\label{anxin0930}
\Phi\left(t\right)=T\left(t-t_{0}\right)\Phi\left(t_{0}\right)+\int_{t_{0}}^{t}T\left(t-\tau\right)\mathbf{G}\left(\Phi\left(\tau\right)\right)d\tau.
\end{align}

\subsection{Regularity of solution}
From Lemma \ref{shidingxing0111}, there exists a weak solution $\left(\mathbf{u},\theta\right)$ to the system \eqref{model}-\eqref{bianjian0814} with initial value  \(\left(\mathbf{u}_{0},\theta_{0}\right)\in \widetilde{\mathbf{E}}_{1}\times\mathbf{E}_{2}\). 
This subsection is devoted to improving the regularity of the solution $\left(\mathbf{u},\theta\right)$ to the system \eqref{model}-\eqref{bianjian0814} with the initial value \(\left(\mathbf{u}_{0},\theta_{0}\right)\in \widetilde{\mathbf{E}}_{1}\times\mathbf{E}_{2}\) through the semigroup of operators. For any $\alpha\in\mathbf{R}$, we define the following fractional Sobolev spaces,
\begin{align}\label{chouxiang0928}
\begin{cases}
\mathbf{H}_{\alpha}=\left\{\mathbf{u}\in [L^2\left(\Omega\right)]^2|\mathbf{u}=\sum\limits_{k,l=1}^{+\infty}\mathbf{u}_{kl}\widetilde{\mathbf{e}}_{kl}~\text{and}~\sum\limits_{k,l=1}^{+\infty}\gamma_{kl}^{2\alpha}|\mathbf{u}_{kl}|^2<+\infty\right\},
\\
H_{\alpha}=\left\{\theta\in L^2\left(\Omega\right)|\theta=\sum\limits_{k=0,l=1}^{+\infty}\theta_{kl}\widetilde{e}_{kl}~\text{and}~\sum\limits_{k,l=1}^{+\infty}\gamma_{kl}^{2\alpha}|\theta_{kl}|^2<+\infty\right\},
\end{cases}
\end{align}
where $\widetilde{\mathbf{e}}_{kl}=\frac{2}{\gamma_{kl}}\mathbf{e}_{kl}$, 
$\widetilde{e}_{kl}=\frac{2}{\gamma_{kl}}e_{kl}$, 
$\widetilde{e}_{0l}=\sqrt{\frac{2}{L}}e_{0l}$ and the definitions of $e_{kl}$ and $\mathbf{e}_{kl}$ can be found in \eqref{tezhenxiangliang0926} and \eqref{diderzu0814}. Then, one can see
\begin{align}\label{fenshuci0928}
\begin{cases}
\mathbf{H}_{0}=\left\{\mathbf{u}=\left(u_{1},u_{2}\right)\in \left[L^2\left(\Omega\right)\right]^2\right\},~H_{0}=\left\{\theta\in L^2\left(\Omega\right)\right\},
\\
\mathbf{H}_{\frac{1}{2}}=\left\{\mathbf{u}=\left(u_{1},u_{2}\right)\in \left[H^1\left(\Omega\right)\right]^2\right\},~H_{\frac{1}{2}}=\left\{\theta\in H^1\left(\Omega\right)\right\},
\\
\mathbf{H}_{1}=\left\{\mathbf{u}=\left(u_{1},u_{2}\right)\in \left[H^2\left(\Omega\right)\right]^2\right\},~H_{1}=\left\{\theta\in H^2\left(\Omega\right)\right\}.
\end{cases}
\end{align}

We adopt the following strategy to improve the regularity of the solution: first, we consider the equations of the velocity field and the temperature field separately, and confirm that the solution can be expressed in the form of a semigroup of operators. Then, we first improve the regularity of the velocity; after that, we enhance the regularity of the temperature; subsequently, we improve the regularity of the velocity again. We repeat this process and elevate the regularity of the solution to the highest level.

From \cite{Pazy1983,ma2011}, we conclude that for $\eqref{model}_{1,2}$ and $\eqref{model}_{3}$, the solutions $\mathbf{u}$ and $\theta$ can be expressed in form of the semigroup of operators as follows 
\begin{align}\label{jiedebiaoshi0111}
\begin{aligned}
&\mathbf{u}\left(t\right)=e^{-V_{a}\left(t-t_{0}\right)}\mathbf{u}_{0}+V_{a}R\int_{t_0}^{t}
e^{-V_{a}\left(t-\tau\right)}\mathbf{P}\left(\theta\left(\tau\right)\mathbf{e}_{2}\right)d\tau,
\\
&\theta\left(t\right)=
e^{\left(t-t_{0}\right)\Delta }\theta_{0}+\int_{t_0}^{t}
e^{\left(t-\tau\right)\Delta}f\left(\tau\right)d\tau,
\end{aligned}
\end{align}
where  \(f\left(t\right)=Ru_{2}-\mathbf{u}\cdot\nabla\theta\). 

\begin{theorem}\label{dingxia1002}
   If the initial value $\left(\mathbf{u}_{0},\theta_{0}\right)$ belongs to $ \mathbf{H}_{\alpha}\times H_{\frac{1}{2}}$, where $\alpha\in (\frac{1}{2},1)$, then the corresponding solution $\left(\mathbf{u},\theta\right)$ to the system \eqref{model}-\eqref{bianjian0814} has the following properties 
    \begin{align*}
    \begin{aligned}
     \mathbf{u}\in C^{1}\left(t_0,T;\mathbf{H}_{\alpha}\right),~\theta\left(t\right)\in C^{1}\left(t_0,T;H_{\widetilde{\alpha}}\right)\cap C\left(t_0,T;H_{\frac{3}{2}}\right),
    \end{aligned}
    \end{align*}
    where $\frac{1}{2}<\widetilde{\alpha}<\alpha$.
\end{theorem}
\begin{proof}
According to \cite{Evans2010}, from $\theta\in L^{\infty}\left(t_{0},T;H_{\frac{1}{2}}\right)\cap L^{2}\left(t_{0},T;H_{1}\right)$ and $\partial_{t}\theta\in L^{2}\left(t_{0},T;H_{0}\right)$, we have 
$\theta\in C\left(t_{0},T;H_{\frac{1}{2}}\right)$.

First, by $\eqref{jiedebiaoshi0111}_{1}$, we conclude that as \(\mathbf{u}_{0}\in \mathbf{H}_{\alpha}\)($\alpha\in \left[\frac{1}{2},1\right]$),
\begin{align}\label{xuyao10117}
\mathbf{u}\in C^{1}\left(t_0,T;\mathbf{H}_{\frac{1}{2}}\right)\cap C\left(t_0,T;\mathbf{H}_{\alpha}\right).
\end{align}
In fact, from \eqref{xingzhi0111},
we deduce 
\begin{align*}
&\begin{aligned}
\left\|\mathbf{u}\right\|_{\mathbf{H}_{\alpha}}&\leq e^{-V_{a}(t-t_{0})}\left\|\mathbf{u}_{0}\right\|_{\mathbf{H}_{\alpha}}
+V_{a}R\int_{t_0}^{t}e^{-V_{a}\left(t-\tau\right)}\left\|\mathbf{P}\left(\theta\left(\tau\right)\mathbf{e}_{2}\right)\right\|_{\mathbf{H}_{\alpha}} d\tau
\\
&\leq e^{-V_{a}(t-t_{0})}\left\|\mathbf{u}_{0}\right\|_{\mathbf{H}_{\alpha}}+CV_{a}R\left[\int_{t_0}^{t}
e^{-2V_{a}\left(t-\tau\right)}d\tau\right]^{\frac{1}{2}}
\left[\int_{t_0}^{t}\left\|\theta\left(\tau\right)\right\|_{H_{1}}^2d\tau\right]^{\frac{1}{2}}<+\infty, 
\end{aligned}
\\
&\begin{aligned}
\left\|\frac{d\mathbf{u}}{dt}\right\|_{\mathbf{H}_{\frac{1}{2}}}\leq V_{a}e^{-V_{a}(t-t_{0})}\left\|\mathbf{u}_{0}\right\|_{\mathbf{H}_{\frac{1}{2}}}+V_{a}R
\left\|\theta\right\|_{H_{\frac{1}{2}}}+V_{a}^2R\int_{t_0}^{t}e^{-V_{a}\left(t-\tau\right)}\left\|\mathbf{P}\left(\theta\left(\tau\right)\mathbf{e}_{2}\right)\right\|_{\mathbf{H}_{\frac{1}{2}}} d\tau<+\infty.
\end{aligned}
\end{align*}

We choose \(\frac{1}{2}<\alpha<1\) fixed and consider 
\begin{align}\label{dandukaolv20111}
\begin{cases}
\frac{\partial \theta}{\partial t}+\mathbf{u}\cdot\nabla\theta=Ru_{2}+\Delta \theta,
\\
\theta\left(t_0\right)=\theta_{0}.
\end{cases}
\end{align}
Then, $\eqref{jiedebiaoshi0111}_{2}$ implies that 
\begin{align}\label{xingzhi20111}
\theta\left(t\right)\in  C\left(t_0,T; H_{\beta}\right),~\text{where~}0<\beta<1~\text{is~arbitrary}.
\end{align}
In fact, we have
\begin{align*}
\begin{aligned}
\left\|\left(-\Delta\right)^{\beta}\theta\right\|_{H_{0}}&\leq 
\left\|e^{(t-t_{0})\Delta}\theta_{0}\right\|_{H_{\beta}}+\int_{t_0}^{t}
\left\|e^{\left(t-\tau\right)\Delta}f\left(\tau\right)\right\|_{H_{\beta}}d\tau
\\
&\leq \left\|e^{(t-t_{0})\Delta}\theta_{0}\right\|_{H_{\beta}} 
+C\int_{t_0}^{t}\left(t-\tau\right)^{-\frac{\beta}{2}}\left\|f\left(\tau\right)\right\|_{H_{\frac{\beta}{2}}}ds
\\
&\leq \left\|e^{(t-t_{0})\Delta}\theta_{0}\right\|_{H_{\beta}} 
+C\left[\int_{t_0}^{t}\left(t-\tau\right)^{-\beta}d\tau\right]^{\frac{1}{2}}\left[\int_{t_0}^{t}\left\|f\left(\tau\right)\right\|_{H_{\frac{\beta}{2}}}^2d\tau\right]^{\frac{1}{2}}<+\infty.
\end{aligned}
\end{align*}

Furthermore, from \eqref{xuyao10117} and \eqref{xingzhi20111}, one can conclude that there exists a \(\eta>0\) such that 
\begin{align*}
f\left(t\right)\in C\left(t_0,T; H_{\eta}\right),
\end{align*}
which together with $\eqref{jiedebiaoshi0111}_{2}$ yields that 
\begin{align*}
\begin{aligned}
\left\|\Delta\theta\right\|_{H_{0}}
\leq \left\|e^{(t-t_{0})\Delta}\theta_{0}\right\|_{H_{1}}+C\int_{t_0}^{t}
\left(t-\tau\right)^{\eta-1}d\tau 
\left\|f\left(t\right)\right\|_{L^{\infty}\left(t_0,T;H_{\eta}\right)}<+\infty.
\end{aligned}
\end{align*}
Thus, we have 
\begin{align}\label{jinyibutisheng0117}
\theta\left(t\right)\in  C\left(t_0,T; H_{1}\right).
\end{align}
As a result, we have 
\(f\left(t\right)\in C\left(t_0,T; H_{\frac{1}{2}}\right)\). Similarly, we obtain from $\eqref{jiedebiaoshi0111}_{2}$ 
 that 
\begin{align}\label{from0117}
\theta\left(t\right)\in C\left(t_0,T;H_{\sigma}\right),~\text{where~}\sigma\in \left[1,\frac{3}{2}
\right),
\end{align}
which yields that \(f\left(t\right)\in C\left(t_0,T;H_{\alpha}\right)\), where \(\alpha\in \left(\frac{1}{2},1\right)\). Thus, we can obtain that 
\begin{align}\label{yibuyibu0117}
\theta\left(t\right)\in C\left(t_0,T;H_{\frac{3}{2}}\right).
\end{align}

And 
\[
\frac{d\theta}{dt}
=\Delta \left(e^{(t-t_{0})\Delta}\theta_{0}\right)
+f\left(t\right)
+\int_{t_0}^{t}\Delta \left(e^{\left(t-\tau\right)\Delta}f\left(\tau\right)\right)ds,
\]
then, we can obtain that 
\begin{align}\label{kendingshi0117}
\left\|\frac{d\theta}{dt}\right\|_{H_{\widetilde{\alpha}}}
\leq \left\|\Delta \left(e^{(t-t_{0})\Delta}\theta_{0}\right)\right\|_{H_{\widetilde{\alpha}}}
+\left\|f\left(t\right)\right\|_{H_{\widetilde{\alpha}}}
+\int_{t_0}^{t}\left\|\Delta \left(e^{\left(t-\tau\right)\Delta}f\left(\tau\right)\right)\right\|_{H_{\widetilde{\alpha}}}d\tau,
\end{align}
where \(\frac{1}{2}<\widetilde{\alpha}<\alpha\) is arbitrary. Thus, 
\(\theta\left(t\right)\in C^{1}\left(t_0,T;H_{\widetilde{\alpha}}\right)\cap C\left(t_0,T;H_{\frac{3}{2}}\right)\). From \eqref{jiedebiaoshi0111}, we can deduce that 
\begin{align}\label{buqueding0117}
    \mathbf{u}\in C^{1}\left(t_0,T;\mathbf{H}_{\alpha}\right).
\end{align}
\end{proof}

\section{Locally invariant manifold}\label{buxian0205}
This section is devoted to proving the existence, attractivity and approximation formula of locally invariant manifold for the system \eqref{model}-\eqref{bianjian0814}.
Thus, we introduce an auxiliary system with a smooth cut-off function. Specifically, we consider the following modified system, which coincides with \eqref{model} in a sufficiently small neighborhood of the trivial solution and is globally well-posed:
\begin{align}\label{jiale0727}
\begin{cases}
\frac{\partial \mathbf{u}}{\partial t}=-V_{a}\mathbf{u}+
V_{a}R\mathbf{P}\left(\theta \mathbf{e}_{2}\right),
\\
\frac{\partial \theta}{\partial t}
=\Delta\theta-\rho\left(\frac{\left\|\mathbf{\Phi}\right\|_{\mathbf{X}_{\alpha}}}{r}\right)\mathbf{u}\cdot\nabla\theta+Ru_{2},
\\
\Phi(t_{0})=\Phi_{0}=\left(\mathbf{u}_{0},\theta_{0}\right),
\end{cases}
\end{align}
where 
\begin{align*}
\Phi\left(t\right)=\left(\mathbf{u},\theta\right),~
\rho\left(s\right)=
\begin{cases}
1,~0\leq s\leq 1,
\\
0,~~~s\geq 2,
\end{cases}
\end{align*}
 \(\rho\left(s\right)\in C^{\infty}\left(R^{+}\right)\), \(\left|\rho\left(s\right)\right|\leq 1,~\left|\rho'\left(s\right)\right|\leq 2\), \(r>0\) is a sufficiently small constant. Since the invariant manifold is a local object, the modification outside a small neighborhood of the origin does not affect the local dynamics near the trivial solution. Consequently, any global invariant manifold of \eqref{jiale0727} yields a local invariant manifold of \eqref{model}. We obtain the following conclusions.

\begin{theorem}\label{yashuo0727}\textup{(}the existence and structure of invariant manifold\textup{)}
    For $r>0$ sufficiently small constant and $R$ in the neighborhood of the critical value $R_{c}$ \textup{(}defined in \eqref{linjie0926}\textup{)}, the system \eqref{jiale0727} admits a global invariant manifold $M\left(R\right)$ defined in \eqref{xuzai1215}. 
\end{theorem}

\begin{theorem}\label{bijin0727} \textup{(}the attractivity\textup{)}Under the same condition of \autoref{yashuo0727},
there exists a constant $\alpha^{*}>0$, such that the solution $\Phi\left(t\right)=\Phi_{c}\left(t\right)+\Phi_{s}\left(t\right)$, where $\Phi_{c}\left(t\right)\in \mathbf{X}_{c}$ and $\Phi_{s}\left(t\right)\in\mathbf{X}_{s}$, to the system \eqref{jiale0727} satisfies the following approximation 
\begin{align*}
    \lim\limits_{t\rightarrow+\infty}
    e^{\alpha^{*}t}\left\|\Phi_{s}\left(t\right)-h\left(\Phi_{c}\left(t\right)\right)\right\|_{\mathbf{X}_{\alpha}}=0,
\end{align*}
where $\mathbf{X}_{c}$ and $\mathbf{X}_{s}$ are the critical and stable subspaces of $\mathbf{L}$, respectively.  $h:\mathbf{X}_{c}\rightarrow\mathbf{X}_{s}$ is the invariant manifold function defined in \eqref{hdedansheng1215}.
\end{theorem}
\subsection{The existence and structure of locally invariant manifold}\label{lianxu0727}
In this subsection, we will prove the existence of local invariant manifold for the system \eqref{model}-\eqref{bianjian0814}.
 Similarly, we can rewrite the system \eqref{jiale0727} into abstract form 
 \begin{align}\label{suanzixing1010}
\begin{cases}
\frac{d\Phi}{dt}=\mathbf{L}\Phi+
\widetilde{\mathbf{G}}\left(\Phi\right),
\\
\Phi\left(t_{0}\right)=\left(
\mathbf{u}_{0},\theta_{0}\right),
\end{cases}
\end{align}
where 
\(\widetilde{\mathbf{G}}\left(\Phi\right)
=\widetilde{G}\left(\Phi,\Phi\right)\), 
\(\Phi\left(t_{0}\right)\) is the initial value and 
$$
\widetilde{G}(\Phi,\widetilde{\Phi})=\begin{pmatrix}
        \mathbf{0}
        \\
        -\rho\left(\frac{\left\|\mathbf{\Phi}\right\|_{\mathbf{X}_{\alpha}}}{r}\right)\mathbf{u}\cdot\nabla\widetilde{\theta}
    \end{pmatrix}.
$$
Then, we can conclude that the system \eqref{suanzixing1010} has a unique solution $\Phi\left(t\right)\in C\left(t_{1},t_{2};\mathbf{X}_{\alpha}\right)$ if $\Phi(t_{1})\in \mathbf{X}_{\alpha}=\mathbf{H}_{\alpha}\times H_{\alpha}$ and $\frac{1}{2}<\alpha<1$. Furthermore, the solution $\Phi\left(t\right)$ can be expressed in form 
\begin{align}\label{guaqila1010}
 \Phi\left(t\right)=T\left(t-t_{1}\right)\Phi\left(t_{1}\right)+\int_{t_{1}}^{t}T\left(t-\tau\right)\widetilde{\mathbf{G}}\left(\Phi\left(\tau\right)\right)d\tau.
 \end{align}

From the discussion of spectrum of $\mathbf{L}$, we have the following decomposition, 
\begin{align}\label{kongjianfenjie0709}
\begin{aligned}
\mathbf{H}=\mathbf{X}_{c}+\overline{\mathbf{X}}_{s},
\end{aligned}
\end{align}
where \(\mathbf{X}_{c}\) and \(\mathbf{X}_{s}\) are spanned by 
\(\left\{\Phi_{k_{1},1}^{1},\cdots,\Phi_{k_{m},1}^{1}\right\}\) and the rest eigenvectors of \(\mathbf{L}\), respectively. And, we define two projection operators as follows 
\begin{align}\label{touyingsuanzi0709}
\mathbf{P}_{c}:\mathbf{H}\rightarrow \mathbf{X}_{c},~
\mathbf{P}_{s}:\mathbf{H}\rightarrow \overline{\mathbf{X}}_{s}.
\end{align}
One can easily find that 
\begin{align}\label{jiaohuan0709}
\mathbf{L}\circ \mathbf{P}_{c}=\mathbf{P}_{c}\circ 
\mathbf{L}=\mathbf{P}_{c},~
\mathbf{L}\circ \mathbf{P}_{s}=\mathbf{P}_{s}\circ 
\mathbf{L}=\mathbf{P}_{s}.
\end{align}


From \eqref{guaqila1010}, we have
\begin{align}\label{manmanlai0117}
\begin{aligned}
\mathbf{P}_{c}\Phi\left(t_{2}\right)=\mathbf{P}_{c}T\left(t_{2}-t_{1}\right)\Phi\left(t_{1}\right)
+\mathbf{P}_{c}\int_{t_{1}}^{t_{2}}T\left(t_{2}-\tau\right)\widetilde{\mathbf{G}}\left(\Phi\left(\tau\right)\right)d\tau,
\end{aligned}
\end{align}
which deduces that 
\begin{align}\label{buhuitou0117}
\begin{aligned}
\mathbf{P}_{c}T\left(t-t_{1}\right)\Phi\left(t_{1}\right)=T\left(t-t_{2}\right)\mathbf{P}_{c}
\Phi\left(t_{2}\right)
-\int_{t_{1}}^{t_{2}}T\left(t-\tau\right)\mathbf{P}_{c}\widetilde{\mathbf{G}}\left(\Phi\left(\tau\right)\right)d\tau.
\end{aligned}
\end{align}
And from \eqref{guaqila1010}, we have 
\begin{align*}
\begin{aligned}
\Phi\left(t\right)
=&T\left(t-t_{1}\right)\mathbf{P}_{c}\Phi\left(t_{1}\right)+T\left(t-t_{1}\right)\mathbf{P}_{s}\Phi\left(t_{1}\right)
\\
&+\int_{t_{1}}^{t}T\left(t-\tau\right)
\mathbf{P}_{c}\widetilde{\mathbf{G}}\left(\Phi\left(\tau\right)\right)d\tau 
+\int_{t_{1}}^{t}T\left(t-\tau\right)
\mathbf{P}_{s}\widetilde{\mathbf{G}}\left(\Phi\left(\tau\right)\right)d\tau,
\end{aligned}
\end{align*}
which together with \eqref{buhuitou0117} gives that 
\begin{align}\label{bashuohao0117}
\begin{aligned}
\Phi\left(t\right)
=&T\left(t-t_{2}\right)\mathbf{P}_{c}\Phi\left(t_{2}\right)+T\left(t-t_{1}\right)\mathbf{P}_{s}\Phi\left(t_{1}\right)
\\
&-\int_{t}^{t_{2}}T\left(t-\tau\right)
\mathbf{P}_{c}\widetilde{\mathbf{G}}\left(\Phi\left(\tau\right)\right)d\tau 
+\int_{t_{1}}^{t}T\left(t-\tau\right)
\mathbf{P}_{s}\widetilde{\mathbf{G}}\left(\Phi\left(\tau\right)\right)d\tau.
\end{aligned}
\end{align}

For some \(\gamma>0\) satisfying $\gamma-\gamma_{0}<0$ where $\gamma_{0}$ can be found in \eqref{zuida1212}, we denote 
\[
C_{\gamma}^{-}
=\left\{
\Phi\in C\left(-\infty,0;\mathbf{X}_{\alpha}\right)
|\sup\limits_{t\leq 0}e^{\gamma t}\left\|\Phi\left(t\right)\right\|_{\mathbf{X}_{\alpha}}<+\infty\right\}
\]
and 
\[
\left\|\Phi\left(t\right)\right\|_{C_{\gamma}^{-}}=\sup\limits_{t\leq 0}e^{\gamma t}\left\|\Phi\left(t\right)\right\|_{\mathbf{X}_{\alpha}}.
\]

\begin{proposition}\label{bushouqingxu0117}
    If the equation \eqref{suanzixing1010} has a mild solution \(\Phi\left(t\right)\in C_{\gamma}^{-}\) with \(\Phi\left(0\right)=\Phi_{0}\), then \(\Phi\left(t\right)\) solves the following integral equation 
    \begin{align}\label{shatuo0117}
    \Phi\left(t\right)
=T\left(t\right)\mathbf{P}_{c}\Phi_{0}
-\int_{t}^{0}T\left(t-\tau\right)
\mathbf{P}_{c}\widetilde{\mathbf{G}}\left(\Phi\left(\tau\right)\right)d\tau 
+\int_{-\infty}^{t}T\left(t-\tau\right)
\mathbf{P}_{s}\widetilde{\mathbf{G}}\left(\Phi\left(\tau\right)\right)d\tau.
    \end{align}
Conversely, if \(\Phi\left(t\right)\in C_{\gamma}^{-}\) solves \eqref{shatuo0117} with \(\Phi\left(0\right)=\Phi_{0}\), then 
\(\Phi\left(t\right)\) is a mild solution of the equation \eqref{suanzixing1010}.
\end{proposition}
\begin{proof}
    If \(\Phi\left(t\right)\in C_{\gamma}^{-}\) is a mild solution to \eqref{suanzixing1010}, then we have \eqref{bashuohao0117}. Let 
    \(t_{2}=0\) in \eqref{bashuohao0117}, then for \(t_{1}\leq t\leq 0\), we have
    \begin{align}\label{guaiguai0117}
    \begin{aligned}
    \Phi\left(t\right)=&T\left(t\right)\mathbf{P}_{c}\Phi_{0}+T\left(t-t_{1}\right)\mathbf{P}_{s}\Phi\left(t_{1}\right)
    \\
    &-\int_{t}^{0}T\left(t-\tau\right)
    \mathbf{P}_{c}\widetilde{\mathbf{G}}\left(\Phi\left(\tau\right)\right)d\tau 
    +\int_{t_{1}}^{t}T\left(t-\tau\right)\mathbf{P}_{s}\widetilde{\mathbf{G}}\left(\Phi\left(\tau\right)\right)d\tau,
    \end{aligned}
    \end{align}
    which together with \(\Phi\left(t\right)\in C_{\gamma}^{-}\) and \(t_{1}\rightarrow -\infty\) yields \eqref{shatuo0117} since $T\left(t-t_{1}\right)\mathbf{P}_{s}\Phi\left(t_{1}\right)\rightarrow0$ as $t_{1}\rightarrow-\infty$.

    If \(\Phi\left(t\right)\in C_{\gamma}^{-}\) satisfies 
    \eqref{shatuo0117}, then we have 
    \begin{align*}
    \mathbf{P}_{s}\Phi\left(t_{1}\right)=\int_{-\infty}^{t_{1}}T\left(t_{1}-\tau\right)\mathbf{P}_{s}\widetilde{\mathbf{G}}\left(\Phi\left(\tau\right)\right)d\tau,
    \end{align*}
    which deduces 
    \begin{align}\label{miss0117}
    T\left(t-t_{1}\right)\mathbf{P}_{s}\Phi\left(t_{1}\right)=\int_{-\infty}^{t_{1}}T\left(t-\tau\right)\mathbf{P}_{s}\widetilde{\mathbf{G}}\left(\Phi\left(\tau\right)\right)d\tau.
    \end{align}
    And 
    \[
    \int_{-\infty}^{t}T\left(t-\tau\right)\mathbf{P}_{s}\widetilde{\mathbf{G}}\left(\Phi\left(\tau\right)\right)d\tau=\int_{-\infty}^{t_{1}}T\left(t-\tau\right)\mathbf{P}_{s}\widetilde{\mathbf{G}}\left(\Phi\left(\tau\right)\right)d\tau+\int_{t_{1}}^{t}T\left(t-\tau\right)\mathbf{P}_{s}\widetilde{\mathbf{G}}\left(\Phi\left(\tau\right)\right)d\tau.
    \]
    Thus, we obtain 
    \begin{align}\label{try0117}
    \int_{-\infty}^{t}T\left(t-\tau\right)\mathbf{P}_{s}\widetilde{\mathbf{G}}\left(\Phi\left(\tau\right)\right)d\tau=T\left(t-t_{1}\right)\mathbf{P}_{s}\Phi\left(t_{1}\right)+\int_{t_{1}}^{t}T\left(t-\tau\right)\mathbf{P}_{s}\widetilde{\mathbf{G}}\left(\Phi\left(\tau\right)\right)d\tau.
    \end{align}
    In addition, from \eqref{shatuo0117}, we obtain 
    \[
    T\left(t_{1}\right)\mathbf{P}_{c}\Phi_{0}=\mathbf{P}_{c}\Phi\left(t_{1}\right)+\int_{t_{1}}^{0}T\left(t_{1}-\tau\right)\mathbf{P}_{c}\widetilde{\mathbf{G}}\left(\Phi\left(\tau\right)\right)d\tau,
    \]
    which indicates that 
    \begin{align}\label{never0117}
    T\left(t\right)\mathbf{P}_{c}\Phi_{0}=T\left(t-t_{1}\right)\mathbf{P}_{c}\Phi\left(t_{1}\right)+\int_{t_{1}}^{0}T\left(t-\tau\right)\mathbf{P}_{c}\widetilde{\mathbf{G}}\left(\Phi\left(\tau\right)\right)d\tau.
    \end{align}
    Substituting \eqref{try0117} and \eqref{never0117} into \eqref{shatuo0117}, we obtain 
   \eqref{guaqila1010}. Thus, 
   \(\Phi\left(t\right)\) is a mild solution to the equation \eqref{suanzixing1010}.
\end{proof}
\begin{remark}
    One can refer to \cite{Pazy1983} for the definition of mild solution. From the Proposition \ref{bushouqingxu0117}, we define a set as follows,
    \begin{align}\label{jihe1010}
      \begin{aligned}
          M(R)=\left\{\Phi_{0}\in\mathbf{X}_{\alpha}|\Phi\left(t\right)\in C^{-}_{\gamma}~\text{solves~}\eqref{suanzixing1010}~\text{and~}\Phi\left(0\right)=\Phi_{0}\right\}.
      \end{aligned}  
    \end{align}
    Subsequently, we will show that $M\left(R\right)$ is invariant.
\end{remark}

\begin{proof}\textup(\textbf{Proof for the existence conclusion in \autoref{yashuo0727}}\textup{)} For $\forall t_{0}>0$,
by the uniqueness as in Lemma \ref{shidingxing0111}, we obtain
  \begin{align*}
   \begin{aligned}    \Phi(t,t_{0})=\Phi&\left(t;R,\Phi\left(t_{0};R,\Phi_{0}\right)\right)=\Phi\left(t+t_{0};R,\Phi_{0}\right)
   \end{aligned} 
  \end{align*}
  is the solution to the system \eqref{suanzixing1010} with initial value $\Phi\left(0,t_{0}\right)=\Phi\left(t_{0};R,\Phi_{0}\right)$. 

According to the definition of $M\left(R\right)$,  $\Phi\left(t,t_{0}\right)\in C_{\gamma}^{-}$. In addition,  by Theorem \ref{dingxia1002}, $\forall~t_{0}>0$, $\Phi\left(t_{0};R,\Phi_{0}\right)\in\mathbf{X}_{\alpha}$. 

  After all, we conclude that $\Phi\left(t_{0};R,\Phi_{0}\right)\in M\left(R\right)$. Thus, $M\left(R\right)$ is invariant.
\end{proof}

The next step is to consider the structure of $M\left(R\right)$.
 From \autoref{dingxia1002}, as \(\mathbf{u}_{0}\in\mathbf{H}_{\alpha}\) and \(\theta_{0}\in H_{\alpha}\), 
 we have 
 \[
 \theta\left(t\right)\in C^{1}\left(0,T;H_{\widetilde{\alpha}}\right)\cap C\left(0,T;H_{\frac{3}{2}}\right),~\mathbf{u}\in C^{1}\left(0,T;\mathbf{H}_{\alpha}\right),~\frac{1}{2}<\widetilde{\alpha}<\alpha<1.
 \]
And as \(\Phi\left(t\right)\in C_{\gamma}^{-}\) and from $\eqref{jiedebiaoshi0111}_{2}$, 
\begin{align*}
\begin{aligned}
\theta\left(t\right)
=\int_{-\infty}^{t}e^{\left(t-\tau\right)\Delta}Ru_{2}\left(\tau\right)d\tau-\int_{-\infty}^{t}
e^{\left(t-\tau\right)\Delta}\rho\left(\frac{\left\|\mathbf{\Phi}\right\|_{\mathbf{X}_{\alpha}}}{r}\right)\mathbf{u}\left(\tau\right)\cdot\nabla\theta\left(\tau\right)d\tau,
\end{aligned}
\end{align*}
which together with $\left\|\Phi\right\|_{\mathbf{X}_{\alpha}}\leq 2r$ yields that 
\begin{align}\label{yexu0625}
    \begin{aligned}
    &\begin{aligned}
\left\|\theta\left(t\right)\right\|_{H_{1}}&\leq CR\int_{-\infty}^{t}\left(t-\tau\right)^{\alpha-1}e^{-\frac{\pi^2}{2}\left(t-\tau\right)}\left\|
\mathbf{u}\right\|_{\mathbf{H}_{\alpha}}d\tau
\\
&+C\int_{-\infty}^{t}\left(t-\tau\right)^{\alpha-\frac{3}{2}}e^{-\frac{\pi^2}{2}\left(t-\tau\right)}
\left\|\mathbf{u}\cdot\nabla\theta\right\|_{H_{\alpha-\frac{1}{2}}}d\tau
\leq Cr,
    \end{aligned}
    \\
    &\begin{aligned}
        \left\|\theta\left(t\right)\right\|_{H_{\alpha+\frac{1}{2}}}&\leq CR\int_{-\infty}^{t}\left(t-\tau\right)^{-\frac{1}{2}}e^{-\frac{\pi^2}{2}\left(t-\tau\right)}\left\|
\mathbf{u}\right\|_{\mathbf{H}_{\alpha}}d\tau
\\
&+C\int_{-\infty}^{t}\left(t-\tau\right)^{-\alpha}e^{-\frac{\pi^2}{2}\left(t-\tau\right)}
\left\|\mathbf{u}\cdot\nabla\theta\right\|_{H_{\frac{1}{2}}}d\tau
\leq Cr,
    \end{aligned}
    \\
    &\begin{aligned}
        \left\|\theta\left(t\right)\right\|_{H_{\frac{3}{2}}}&\leq CR\int_{-\infty}^{t}\left(t-\tau\right)^{\alpha-\frac{3}{2}}e^{-\frac{\pi^2}{2}\left(t-\tau\right)}\left\|
\mathbf{u}\right\|_{\mathbf{H}_{\alpha}}d\tau
\\
&+C\int_{-\infty}^{t}\left(t-\tau\right)^{\alpha-\frac{3}{2}}e^{-\frac{\pi^2}{2}\left(t-\tau\right)}
\left\|\mathbf{u}\cdot\nabla\theta\right\|_{H_{\alpha}}d\tau
\leq Cr,
    \end{aligned}
    \end{aligned}
\end{align}
and
 \begin{align}\label{xuyao0119}
 \begin{aligned}
 e^{\gamma t}&\left\|\theta\left(t\right)\right\|_{H_{\frac{3}{2}}}
\leq CR\int_{-\infty}^{t}
\left(t-\tau\right)^{\alpha-\frac{3}{2}}e^{-\left(\frac{\pi^2}{2}-\gamma\right)\left(t-\tau\right)}d\tau
\sup\limits_{t\leq 0}e^{\gamma t}\left\|\mathbf{u}\right\|_{\mathbf{H}_{\alpha}}
\\
&+C\int_{-\infty}^{t}
\left(t-\tau\right)^{\alpha-\frac{3}{2}}e^{-\left(\frac{\pi^2}{2}-\gamma\right)\left(t-\tau\right)}\rho\left(\frac{\left\|\mathbf{\Phi}\right\|_{\mathbf{X}_{\alpha}}}{r}\right)e^{\gamma \tau}\left\|\mathbf{u}\left(\tau\right)\cdot\nabla\theta\left(\tau\right)\right\|_{H_{\alpha}}
d\tau
\\
&\leq C\left(R\right)\sup\limits_{t\leq 0}e^{\gamma t}\left\|\mathbf{u}\right\|_{\mathbf{H}_{\alpha}}+C\left(r\right)
\sup\limits_{t\leq 0}e^{\gamma t}\left\|\theta\right\|_{H_{\frac{3}{2}}},
 \end{aligned}
 \end{align}
where \(C\left(r\right)>0\) increases with \(r\) and we can define \(C\left(r\right)|_{r=0}:=0\). Thus, if \(r>0\) is small enough, we can deduce that 
\begin{align}\label{houmianxuyao0119}
\sup\limits_{t\leq 0}e^{\gamma t}
\left\|\theta\right\|_{H_{\frac{3}{2}}}\leq C\left(R,r\right)\sup\limits_{t\leq 0}e^{\gamma t}\left\|\mathbf{u}\right\|_{\mathbf{H}_{\alpha}}.
\end{align}

Now, we define 
\begin{align}\label{yinshe0119}
\mathbf{F}\left(\xi,\cdot\right)
:C_{\gamma}^{-}\rightarrow C_{\gamma}^{-}
\end{align}
as follows 
\begin{align*}
\begin{aligned}
\mathbf{F}\left(\xi,\Phi\right)
=T\left(t\right)\mathbf{P}_{c}\xi 
-\int_{t}^{0}T\left(t-\tau\right)
\mathbf{P}_{c}\widetilde{\mathbf{G}}\left(r,\Phi\left(\tau\right)\right)
d\tau 
+\int_{-\infty}^{t}T\left(t-\tau\right)\mathbf{P}_{s}\widetilde{\mathbf{G}}\left(r,\Phi\left(\tau\right)\right)d\tau,
\end{aligned}
\end{align*}
 where \(\xi\in\mathbf{X}_{\alpha}\) is given and
 \[
\widetilde{\mathbf{G}}\left(r,\Phi\right)=\widetilde{\mathbf{G}}\left(\Phi\right).
 \]
\begin{proposition}\label{babay0120}For any \(\Phi_{1}=\left(\mathbf{u}_{1},\theta_{1}\right)\) and \(\Phi_{2}=\left(\mathbf{u}_{2},\theta_{2}\right)\in C_{\gamma}^{-}\), then, there exists a positive constant \(\widetilde{C}\left(r\right)>0\) such that 
\begin{align}\label{xuyao0120}
\begin{aligned}
\sup\limits_{t\leq 0}e^{\gamma t}\left\|\widetilde{\mathbf{G}}\left(r,\Phi_{1}\right)-\widetilde{\mathbf{G}}\left(r,\Phi_{2}\right)\right\|_{\mathbf{X}_{\alpha}}
\leq \widetilde{C}\left(r\right)\sup\limits_{t\leq 0}
e^{\gamma t}\left\|\Phi_{1}-\Phi_{2}\right\|_{\mathbf{X}_{\alpha}},
\end{aligned}
\end{align}
where \(\widetilde{C}\left(r\right)\) increases with \(r\) and \(\widetilde{C}\left(0\right)=0\).
\end{proposition}
\begin{proof}
We consider three cases: 1)
\(\left\|\mathbf{\Phi}_{1}\right\|_{\mathbf{X}_{\alpha}}\leq 2r\) and 
\(\left\|\mathbf{\Phi}_{2}\right\|_{\mathbf{X}_{\alpha}}\leq 2r\); 2)
\(\left\|\mathbf{\Phi}_{1}\right\|_{\mathbf{X}_{\alpha}}\leq 2r\) and 
\(\left\|\mathbf{\Phi}_{2}\right\|_{\mathbf{X}_{\alpha}}\geq 2r\); 
3) \(\left\|\mathbf{\Phi}_{1}\right\|_{\mathbf{X}_{\alpha}}\geq 2r\) and 
\(\left\|\mathbf{\Phi}_{2}\right\|_{\mathbf{X}_{\alpha}}\leq 2r\).

For the first case, we have 
\begin{align}\label{bubian0120}
\begin{aligned}
\rho\left(\frac{\left\|\mathbf{\Phi}_{2}\right\|_{\mathbf{X}_{\alpha}}}{r}\right)\mathbf{\Phi}_{2}\cdot\nabla\theta_{2}-\rho\left(\frac{\left\|\mathbf{\Phi}_{1}\right\|_{\mathbf{X}_{\alpha}}}{r}\right)\mathbf{u}_{1}\cdot\nabla\theta_{1}=\sum\limits_{i=1}^{3}I_{i},
\end{aligned}
\end{align}
where 
\begin{align*}
\begin{aligned}
I_{1}=\rho\left(\frac{\left\|\mathbf{\Phi}_{2}\right\|_{\mathbf{X}_{\alpha}}}{r}\right)&\left(\mathbf{u}_{2}-\mathbf{u}_{1}\right)\cdot\nabla\theta_{2},~I_{2}=\rho\left(\frac{\left\|\mathbf{\Phi}_{2}\right\|_{\mathbf{X}_{\alpha}}}{r}\right)\mathbf{u}_{1}\cdot\nabla\left(\theta_{2}-\theta_{1}\right),
\\
&I_{3}=\left[\rho\left(\frac{\left\|\mathbf{\Phi}_{2}\right\|_{\mathbf{X}_{\alpha}}}{r}\right)-\rho\left(\frac{\left\|\mathbf{\Phi}_{1}\right\|_{\mathbf{X}_{\alpha}}}{r}\right)
\right]\mathbf{u}_{1}\cdot\nabla\theta_{1}
\end{aligned}
\end{align*}
A direct computation gives that 
\begin{align*}
\begin{aligned}
e^{\gamma t}\left\|I_{1}\right\|_{H_{\alpha}}
\leq C\left(\Omega\right)e^{\gamma t}
\left\|\theta_{2}\right\|_{H_{\frac{3}{2}}}\left\|\Phi_{1}-\Phi_{2}\right\|_{\mathbf{X}_{\alpha}}\leq C_{1}\left(r,\Omega\right)
\sup\limits_{t\leq 0}e^{\gamma t}\left\|\Phi_{1}-\Phi_{2}\right\|_{\mathbf{X}_{\alpha}},
\end{aligned}
\end{align*}
where \(\left|\rho\left(s\right)\right|\leq 1\) and \eqref{yexu0625} are used. And one can easily verify that \(C_{1}\left(r,\Omega\right)|_{r=0}=0\).

For \(I_{3}\), we have 
\begin{align*}
\begin{aligned}
e^{\gamma t}\left\|I_{3}\right\|_{H_{\alpha}}
\leq \frac{2C}{r}e^{\gamma t}\left|
\left\|\Phi_{1}\right\|_{\mathbf{H}_{\alpha}}-\left\|\Phi_{2}\right\|_{\mathbf{H}_{\alpha}}\right|\left\|\mathbf{u}_{1}\right\|_{\mathbf{H}_{\alpha}}\left\|\theta_{1}\right\|_{H_{\frac{3}{2}}}
\leq C_{3}\left(r,\Omega\right)\sup\limits_{t\leq 0}e^{\gamma t}
\left\|\Phi_{1}-\Phi_{2}\right\|_{\mathbf{X}_{\alpha}},
\end{aligned}
\end{align*}
where \(\left|\rho'\left(s\right)\right|\leq 2\), \(\left\|u_{1}\right\|_{\mathbf{H}_{\alpha}}\leq 2r\) and \eqref{yexu0625} are used. One can obtain that \(C_{3}\left(r,\Omega\right)|_{r=0}=0\).

Since
\begin{align*}
\begin{aligned}
\theta_{1}\left(t\right)
=-\int_{-\infty}^{t}e^{\left(t-\tau\right)\Delta}Ru_{12}\left(\tau\right)d\tau-\int_{-\infty}^{t}
e^{\left(t-\tau\right)\Delta}\rho\left(\frac{\left\|\mathbf{\Phi}_{1}\right\|_{\mathbf{X}_{\alpha}}}{r}\right)\mathbf{u}_{1}\left(\tau\right)\cdot\nabla\theta_{1}\left(\tau\right)d\tau,
\\
\theta_{2}\left(t\right)
=-\int_{-\infty}^{t}e^{\left(t-\tau\right)\Delta}Ru_{22}\left(\tau\right)d\tau-\int_{-\infty}^{t}
e^{\left(t-\tau\right)\Delta}\rho\left(\frac{\left\|\mathbf{\Phi}_{2}\right\|_{\mathbf{X}_{\alpha}}}{r}\right)\mathbf{u}_{2}\left(\tau\right)\cdot\nabla\theta_{2}\left(\tau\right)d\tau,
\end{aligned}
\end{align*}
then, follow the steps to conclude \eqref{houmianxuyao0119} similarly, one can obtain  
\begin{align*}
\begin{aligned}
e^{\gamma t}
\left\|\theta_{1}-\theta_{2}\right\|_{H_{\frac{3}{2}}}
\leq C\left(r,\Omega\right)\sup\limits_{t\leq0}e^{\gamma t}\left\|\mathbf{u}_{1}-\mathbf{u}_{2}\right\|_{\mathbf{H}_{\alpha}}, 
\end{aligned}
\end{align*}
by which we have 
\begin{align*}
\begin{aligned}
e^{\gamma t}\left\|I_{2}\right\|_{H_{\alpha}}
\leq C_{2}\left(r,\Omega\right)
\sup\limits_{t\leq 0}e^{\gamma t}
\left\|\Phi_{1}-\Phi_{2}\right\|_{\mathbf{X}_{\alpha}},
\end{aligned}
\end{align*}
where \(C_{2}\left(r,\Omega\right)|_{r=0}=0\). Thus, \eqref{xuyao0120} holds.

For the second and third scenarios, we just consider \(\left\|\mathbf{\Phi}_{1}\right\|_{\mathbf{X}_{\alpha}}\geq 2r\) and 
\(\left\|\mathbf{\Phi}_{2}\right\|_{\mathbf{X}_{\alpha}}\leq 2r\). Since 
\begin{align*}
\begin{aligned}
\rho\left(\frac{\left\|\mathbf{\Phi}_{2}\right\|_{\mathbf{X}_{\alpha}}}{r}\right)u_{2}\cdot\nabla\theta_{2}-\rho\left(\frac{\left\|\mathbf{\Phi}_{1}\right\|_{\mathbf{X}_{\alpha}}}{r}\right)u_{1}\cdot\nabla\theta_{1}=\left(\rho\left(\frac{\left\|\mathbf{\Phi}_{2}\right\|_{\mathbf{X}_{\alpha}}}{r}\right)-\rho\left(\frac{\left\|\mathbf{\Phi}_{1}\right\|_{\mathbf{X}_{\alpha}}}{r}\right)\right)u_{2}\cdot\nabla\theta_{2},
\end{aligned}
\end{align*}
which together with the Lagrange's mean value theorem and \eqref{yexu0625} yields  \eqref{xuyao0120}.
\end{proof}
\begin{theorem}\label{yashuo1010}
    Assume that the first eigenvalue of $\mathbf{L}$ is $\lambda_{0}$, $R$ is in the neighbourhood of $R_{c}$ and $|\lambda_{0}|<\gamma<\gamma_{0}$, where $\gamma_{0}$ can be found in \eqref{zuida1212}, then
one can verify that \(F\left(\Phi\right)\in C_{\gamma}^{-}\) for each 
\(\Phi\in C_{\gamma}^{-}\). And 
$$
\mathbf{F}\left(\xi,\cdot\right)
:C_{\gamma}^{-}\rightarrow C_{\gamma}^{-}
$$
is a contraction map if $r$ is sufficiently small. That is, 
\begin{align}\label{yasuo1215}
\left\|\mathbf{F}\left(\xi,\Phi_{1}\left(t\right)\right)-\mathbf{F}\left(\xi,\Phi_{2}\left(t\right)\right)\right\|_{C_{\gamma}^{-}}
\leq \mathcal{K}\left\|\Phi_{1}\left(t\right)-\Phi_{2}\left(t\right)\right\|_{C_{\gamma}^{-}},
\end{align}
where $\mathcal{K}=\mathcal{K}\left(r,\lambda_{0},\gamma_{0}\right)>0$ and $\mathcal{K}\left(0\right)=0$.
\end{theorem}
\begin{proof}
Indeed, we have 
\begin{align}\label{buyaomimang0119}
\begin{aligned}
e^{\gamma t}\left\|\mathbf{F}\left(\xi,\Phi\left(t\right)\right)\right\|_{\mathbf{X}_{\alpha}}
\leq& 
e^{\gamma t}\bigg{[}e^{\lambda_{0}t}
\left\|\mathbf{P}_{c}\xi\right\|_{\mathbf{X}_{\alpha}}
+\int_{t}^{0}e^{\lambda_{0}\left(t-\tau\right)}e^{-\gamma \tau}d\tau ~C\left(r\right) \sup\limits_{t\leq 0}e^{\gamma t}\left\|\Phi\right\|_{\mathbf{X}_{\alpha}}
\\
&+
\int_{-\infty}^{t}
e^{-\gamma_{0}\left(t-\tau\right)}e^{-\gamma \tau} d\tau~C\left(r\right)\sup\limits_{t\leq 0}e^{\gamma t}\left\|\Phi\right\|_{\mathbf{X}_{\alpha}}
\bigg{]}<+\infty.
\end{aligned}
\end{align} 

Subsequently, we verify that \(\mathbf{F}\) is a contraction map. 

For any \(\Phi_{1}\) and \(\Phi_{2}\in C_{\gamma}^{-}\), a direct computation gives that 
\begin{align}\label{daizhe0120}
\begin{aligned}
&\left\|\mathbf{F}\left(\xi,\Phi_{1}\right)-\mathbf{F}\left(\xi,\Phi_{2}\right)\right\|_{\mathbf{X}_{\alpha}}
\leq 
\int_{t}^{0}e^{\lambda_{0}\left(t-\tau\right)}
\left\|\widetilde{\mathbf{G}}\left(r,\Phi_{1}\right)-\widetilde{\mathbf{G}}\left(r,\Phi_{2}\right)\right\|_{\mathbf{X}_{\alpha}}d\tau
\\
&+
\int_{-\infty}^{t}e^{-\gamma_{0}\left(t-\tau\right)}
\left\|\widetilde{\mathbf{G}}\left(r,\Phi_{1}\right)-\widetilde{\mathbf{G}}\left(r,\Phi_{2}\right)\right\|_{\mathbf{X}_{\alpha}}d\tau
\\
&
\leq\widetilde{C}\left(r\right)
\left[\int_{t}^{0}e^{\lambda_{0}\left(t-\tau\right)}e^{-r\tau}d\tau+\int_{-\infty}^{t}e^{-\gamma_{0}\left(t-\tau\right)}e^{-\gamma \tau}d\tau\right]\left\|\Phi_{1}-\Phi_{2}\right\|_{C_{\gamma}^{-}},
\end{aligned}
\end{align}
which yields that 
\begin{align}\label{xuyaoni1215}
\left\|\mathbf{F}\left(\xi,\Phi_{1}\right)-\mathbf{F}\left(\xi,\Phi_{2}\right)\right\|_{C_{\gamma}^{-}}\leq \mathcal{K}
\left\|\Phi_{1}-\Phi_{2}\right\|_{C_{\gamma}^{-}},
\end{align}
where $\mathcal{K}=\widetilde{C}\left(r\right)\left(\frac{1}{\lambda_{0}+\gamma}+\frac{1}{\gamma_{0}-\gamma}\right)$.
Thus, from the Proposition \ref{babay0120}, as \(r\) is small enough,  \(\mathbf{F}\) is a contraction map.
\end{proof}

Since $
\mathbf{F}\left(\xi,\cdot\right)
:C_{\gamma}^{-}\rightarrow C_{\gamma}^{-}
$ is a contraction map, there exists a fixed point $\Phi\left(t;R,\xi\right)$ for any $\xi\in \mathbf{X}_{\alpha}$. Let $\xi\in\mathbf{X}_{c}$, then 
there exists a $\Phi\left(t;R,\xi\right)$ such that 

\begin{align*}
\begin{aligned}
\Phi\left(t;R,\xi\right)
=&T\left(t\right)\xi 
-\int_{t}^{0}T\left(t-\tau\right)
\mathbf{P}_{c}\widetilde{\mathbf{G}}\left(r,\Phi\left(\tau;R,\xi\right)\right)
d\tau
\\
&+\int_{-\infty}^{t}T\left(t-\tau\right)\mathbf{P}_{s}\widetilde{\mathbf{G}}\left(r,\Phi\left(\tau;R,\xi\right)\right)d\tau.
\end{aligned}
\end{align*}
Subsequently, let $t=0$ in the above equality, then we can define a map $h\left(R,\cdot\right)=h\left(\cdot\right):\mathbf{X}_{c}\rightarrow\overline{\mathbf{X}}_{s}$ 
as follows 
\begin{align}\label{hdedansheng1215}
\begin{aligned}
\Phi_{0}=\Phi\left(0;R,\xi\right)&=\xi+\int_{-\infty}^{0}T\left(-\tau\right)\mathbf{P}_{s}\widetilde{\mathbf{G}}\left(r,\Phi\left(\tau;R,\xi\right)\right)d\tau
\\
&=\xi+h\left(R,\xi\right).
\end{aligned}
\end{align}

Thus, $M\left(R\right)$ can be expressed in form of 
\begin{align}\label{xuzai1215}
M\left(R\right)=\left\{\Phi_{0}\in\mathbf{X}_{\alpha}|\Phi_{0}=\xi+h\left(\xi\right),~\xi\in\mathbf{X}_{c}\right\}.
\end{align}
Regarding the property of $h\left(\xi\right)$, we have the following lemma.
\begin{lemma}\label{xiang1215}
$h\left(\xi\right)$ is Lipschitz continuous. That is, 
$$
\left\|h\left(\xi_{1}\right)-h\left(\xi_{2}\right)\right\|_{\mathbf{X}_{\alpha}}\leq C\left\|\xi_{1}-\xi_{2}\right\|_{\mathbf{X}_{\alpha}},
$$
where $C=\frac{\widetilde{C}\left(r\right)}{\gamma_{0}-\gamma}/\left(1-\frac{\widetilde{C}\left(r\right)}{\lambda_{0}+\gamma}-\frac{\widetilde{C}\left(r\right)}{\gamma_{0}-\gamma}\right)>0.$
\end{lemma}
\begin{proof}
For any $\xi_{1}$ and $\xi_{2}\in\mathbf{X}_{c}$, we 
\begin{align*}
\begin{aligned}
    &h\left(\xi_{1}\right)=\int_{-\infty}^{0}T\left(-\tau\right)
    \mathbf{P}_{s}\widetilde{\mathbf{G}}\left(r,\Phi_{1}\left(\tau,R,\xi_{1}\right)\right)d\tau,
    \\
    &h\left(\xi_{2}\right)=\int_{-\infty}^{0}T\left(-\tau\right)
    \mathbf{P}_{s}\widetilde{\mathbf{G}}\left(r,\Phi_{2}\left(\tau,R,\xi_{2}\right)\right)d\tau,
    \end{aligned}
\end{align*}
where $\Phi_{1}\left(t,R,\xi_{1}\right)=\mathbf{F}\left(\xi_{1},\Phi_{1}\left(t,R,\xi_{1}\right)\right)$ and 
$\Phi_{2}\left(t,R,\xi_{2}\right)=\mathbf{F}\left(\xi_{2},\Phi_{2}\left(t,R,\xi_{2}\right)\right)$. Thus, a direct computation with the Proposition \ref{babay0120} gives that 
\begin{align}\label{pa1215}
\left\|h\left(\xi_{1}\right)-h\left(\xi_{2}\right)\right\|_{\mathbf{X}_{\alpha}}\leq \frac{\widetilde{C}\left(r\right)}{\gamma_{0}-\gamma}\sup\limits_{t\leq 0}e^{\gamma t}\left\|\Phi_{1}\left(t;R,\xi_{1}\right)-\Phi_{2}\left(t;R,\xi_{2}\right)\right\|_{\mathbf{X}_{\alpha}}.
\end{align}
In addition, we have 
\begin{align*}
\begin{aligned}
e^{\gamma t}&\left\|\Phi_{1}\left(t;R,\xi_{1}\right)-\Phi_{2}\left(t;R,\xi_{2}\right)\right\|_{\mathbf{X}_{\alpha}}=e^{\gamma t}\left\|
\mathbf{F}\left(\xi_{1},\Phi_{1}\left(t;R,\xi_{1}\right)\right)-\mathbf{F}\left(\xi_{2},\Phi_{2}\left(t;R,\xi_{2}\right)\right)\right\|_{\mathbf{X}_{\alpha}}
\\
&
\leq e^{\left(\gamma+\lambda_{0}\right)t}\left\|\xi_{1}-\xi_{2}\right\|_{\mathbf{X}_{\alpha}}
+\frac{\left(\lambda_{0}+\gamma_{0}\right)\widetilde{C}\left(r\right)}{\left(\lambda_{0}+\gamma\right)\left(\gamma_{0}-\gamma\right)}
\sup\limits_{t\leq 0}e^{\gamma t}\left\|\Phi_{1}\left(t;R,\xi_{1}\right)-\Phi_{2}\left(t;R,\xi_{2}\right)\right\|_{\mathbf{X}_{\alpha}},
\end{aligned}
\end{align*}
where the expression of $\mathbf{F}\left(\xi,\Phi\right)$ and the Proposition \ref{babay0120} are used. And as $r$ is small sufficiently, we can obtain 
\begin{align*}
\sup\limits_{t\leq 0}e^{\gamma t}\left\|\Phi_{1}\left(t;R,\xi_{1}\right)-\Phi_{2}\left(t;R,\xi_{2}\right)\right\|_{\mathbf{X}_{\alpha}}\leq 
\left(1-\frac{\left(\lambda_{0}+\gamma_{0}\right)\widetilde{C}\left(r\right)}{\left(\lambda_{0}+\gamma\right)\left(\gamma_{0}-\gamma\right)}\right)^{-1}\left\|\xi_{1}-\xi_{2}\right\|_{\mathbf{X}_{\alpha}},
\end{align*}
which together with \eqref{pa1215} yields that 
$$
\left\|h\left(\xi_{1}\right)-h\left(\xi_{2}\right)\right\|_{\mathbf{X}_{\alpha}}\leq C\left\|\xi_{1}-\xi_{2}\right\|_{\mathbf{X}_{\alpha}}.
$$
\end{proof}

From $h\left(\xi\right)=\int_{-\infty}^{0}T\left(-\tau\right)\mathbf{P}_{s}\widetilde{\mathbf{G}}\left(r,\Phi\left(\tau;R,\xi\right)\right)d\tau$ and Lemma \ref{xiang1215}, one can deduce that 
\begin{align}\label{xuyao1222}
    h\left(0\right)=0,  D_\xi h\left(0\right)=\mathbf{0}.
\end{align}
\subsection{The attracting property of locally invariant manifold}\label{xinzeng0719}
Next, we show that the solution $\Phi\left(t\right)$ to the system \eqref{suanzixing1010} with initial value $\Phi\left(0\right)=\left(\mathbf{u}_{0},\theta_{0}\right)\in\mathbf{X}_{\alpha}$ will tend to the invariant manifold exponentially as $t\rightarrow +\infty$. To this end, we deduce the following two estimates.
\begin{lemma}\label{xuyaode1231}
For the solution $\Phi\left(t\right)=\left(\mathbf{u}\left(t\right),\theta\left(t\right)\right)$ to the system \eqref{suanzixing1010} with initial value $\Phi\left(0\right)=\left(\mathbf{u}_{0},\theta_{0}\right)\in\mathbf{X}_{\alpha}$,
as $\left\|\mathbf{\Phi}\left(t\right)\right\|_{\mathbf{X}_{\alpha}}\leq 2r$ for any $t\ge0$, we have
\begin{align}\label{shide0105}
\begin{aligned}
\left\|\theta\left(t\right)\right\|_{H_{1}}\leq Cr+Ct^{\alpha-1}e^{-\frac{\pi^2 t}{2}},~\left\|\theta\right\|_{H_{\alpha+\frac{1}{2}}}\leq Cr+Ct^{-\frac{1}{2}}e^{-\frac{\pi^2t}{2}},~\forall~t>0,
        \end{aligned}
\end{align}
where $C\left(R,\Omega,\left\|\theta_{0}\right\|_{H_{\alpha}}\right)>0$ is a constant depending on $R$, $\Omega$ and $\left\|\theta_{0}\right\|_{H_{\alpha}}$. Furthermore, for the solution $\widetilde{\Phi}\left(t\right)=\left(\widetilde{\mathbf{u}}\left(t\right),\widetilde{\theta}\left(t\right)\right)$ to the system \eqref{suanzixing1010} with initial value $\widetilde{\Phi}\left(0\right)=\left(\widetilde{\mathbf{u}}_{0},\widetilde{\theta}_{0}\right)\in\mathbf{X}_{\alpha}$, as $\left\|\widetilde{\mathbf{\Phi}}\left(t\right)\right\|_{\mathbf{X}_{\alpha}}\leq 2 r$ for any $t> 0$, we have 
\begin{align}\label{xuehui0104}
\begin{aligned}
    \left\|\theta\left(t\right)-\widetilde{\theta}\left(t\right)\right\|_{H_{\alpha+\frac{1}{2}}}
    \leq &Cr+Ct^{-\frac{1}{2}}e^{-\frac{\pi^2 t}{2}}
    ,
\end{aligned}
\end{align}
where $C=C\left(\Omega,\left\|\theta_{0}\right\|_{H_{\alpha}},\left\|\widetilde{\theta}_{0}\right\|_{H_{\alpha}},R\right)>0$.
\end{lemma}
\begin{proof}
We have obtained that 
\begin{align}\label{mohuang1231}
\begin{aligned}
\theta\left(t\right)=e^{t\Delta}\theta_{0}+R\int_{0}^{t}e^{\left(t-\tau\right)\Delta}u_{2}\left(\tau\right)d\tau-\int_{0}^{t}e^{\left(t-\tau\right)\Delta}\mathbf{u}\cdot\nabla\theta d\tau,
\end{aligned}
\end{align}
which together with $\left\|\Phi\left(t\right)\right\|_{\mathbf{X}_{\alpha}}\leq 2r$, we have 
\begin{align}\label{buyong0105}
\begin{aligned}
\left\|\theta\left(t\right)\right\|_{H_{1}}&\leq C t^{\alpha-1}e^{-\frac{\pi^2 t}{2}}+Cr\int_{0}^{t}\left(t-\tau\right)^{\alpha-1}e^{-\frac{\pi^2}{2}\left(t-\tau\right)}d\tau 
\\
&~~~~+Cr\int_{0}^{t}\left(t-\tau\right)^{\alpha-\frac{3}{2}}e^{-\frac{\pi^2}{2}\left(t-\tau\right)}\left\|\theta\left(\tau\right)\right\|_{H_{\alpha}}d\tau
\\
&\leq Ct^{\alpha-1}e^{-\frac{\pi^2 t}{2}}+Cr,~\forall~t>0.
\end{aligned}
\end{align}


Finally, we have
\begin{align*}
    \begin{aligned}
\left\|\theta\left(t\right)\right\|_{H_{\alpha+\frac{1}{2}}}&\leq Ct^{-\frac{1}{2}}e^{-\frac{\pi^2 t}{2}}
+Cr\int_{0}^{t}\left(t-\tau\right)^{-\frac{1}{2}}e^{-\frac{\pi^2\left(t-\tau\right)}{2}}d\tau +Cr\int_{0}^{t}\left(t-\tau\right)^{-\alpha}e^{-\frac{\pi^2\left(t-\tau\right)}{2}}\left\|\theta\left(\tau\right)\right\|_{H_{1}}d\tau
\\
&\leq Cr+Ct^{-\frac{1}{2}}e^{-\frac{\pi^2 t}{2}},
    \end{aligned}
\end{align*}
where \eqref{buyong0105} is used.
 Thus, \eqref{shide0105} is proved.

Subsequently, we consider $\left\|\theta\left(t\right)-\widetilde{\theta}\left(t\right)\right\|_{H_{\alpha+\frac{1}{2}}}$ when $\left\|\mathbf{\Phi}\right\|_{\mathbf{X}_{\alpha}}\leq 2r$ and $\left\|\mathbf{\widetilde{\Phi}}\right\|_{\mathbf{X}_{\alpha}}\leq 2r$. We have 
\begin{align*}
    \begin{aligned}
        \theta\left(t\right)-\widetilde{\theta}\left(t\right)=&e^{t\Delta}\left(\theta_{0}-\widetilde{\theta}_{0}\right)
        +R\int_{0}^{t}e^{\left(t-\tau\right)\Delta}\left(u_{2}\left(\tau\right)-\widetilde{u}_{2}\left(\tau\right)\right)d\tau
\\
&-\int_{0}^{t}e^{\left(t-\tau\right)\Delta}\left(\mathbf{u}\left(\tau\right)-\widetilde{\mathbf{u}}\left(\tau\right)\right)\cdot\nabla\theta\left(\tau\right) d\tau
\\
&-\int_{0}^{t}e^{\left(t-\tau\right)\Delta}\widetilde{\mathbf{u}}\left(\tau\right)\cdot\nabla\left(\theta\left(\tau\right)-\widetilde{\theta}\left(\tau\right)\right)d\tau,
    \end{aligned}
\end{align*}

 which together with $\left\|\mathbf{\Phi}\right\|_{\mathbf{X}_{\alpha}}\leq 2r$ and $\left\|\mathbf{\widetilde{\Phi}}\right\|_{\mathbf{X}_{\alpha}}\leq 2r$ yields  \eqref{xuehui0104} in a similar way.

\end{proof}

\begin{lemma}\label{xuyaode0116}
For the solution $\Phi\left(t\right)=\left(\mathbf{u}\left(t\right),\theta\left(t\right)\right)$ and $\widetilde{\Phi}\left(t\right)=\left(\mathbf{\widetilde{u}}\left(t\right),\widetilde{\theta}\left(t\right)\right)$ to the system \eqref{suanzixing1010} with initial value $\Phi\left(0\right)=\left(\mathbf{u}_{0},\theta_{0}\right)\in\mathbf{X}_{\alpha}$ and $\widetilde{\Phi}\left(0\right)=\left(\mathbf{\widetilde{u}}_{0},\widetilde{\theta}_{0}\right)\in\mathbf{X}_{\alpha}$ ,
as $\left\|\mathbf{\Phi}\left(t\right)\right\|_{\mathbf{X}_{\alpha}}\leq 2r$ and $\left\|\mathbf{\widetilde{\Phi}}\left(t\right)\right\|_{\mathbf{X}_{\alpha}}\leq 2r$ for any $t\ge0$, we have
\begin{align}\label{dierge01166}
    \begin{aligned}
    &\begin{aligned}
        \left\|\theta\left(t\right)-\widetilde{\theta}\left(t\right)\right\|_{H_{1}}
        \leq& Ct^{\alpha-1}e^{-\frac{\pi^2 t}{2}}
        +R\int_{0}^{t}\left(t-\tau\right)^{\alpha-1}e^{-\frac{\pi^2\left(t-\tau\right)}{2}}\left\|\mathbf{u}-\widetilde{\mathbf{u}}\right\|_{\mathbf{H}_{\alpha}}d\tau
        \\
        &+Cr\int_{0}^{t}\left(t-\tau\right)^{\alpha-\frac{3}{2}}e^{-\frac{\pi^2\left(t-\tau\right)}{2}}\left\|\mathbf{u}-\widetilde{\mathbf{u}}\right\|_{\mathbf{H}_{\alpha}}d\tau
        \\
        &+Cr\int_{0}^{t}\left(t-\tau\right)^{\alpha-\frac{3}{2}}e^{-\frac{\pi^2\left(t-\tau\right)}{2}}\left\|\theta-\widetilde{\theta}\right\|_{H_{\alpha}}d\tau,
    \end{aligned}
    \\
    & \begin{aligned}
        ||\theta\left(t\right)-\widetilde{\theta}&\left(t\right)||_{H_{\alpha+\frac{1}{2}}}
        \leq  C \left(t^{-\frac{1}{2}}+1\right)e^{-\frac{\pi^2 t
        }{2}}+C\int_{0}^{t}\left(t-\tau\right)^{-\frac{1}{2}}e^{-\frac{\pi^2\left(t-\tau\right)}{2}}\left\|\mathbf{u}-\widetilde{\mathbf{u}}\right\|_{\mathbf{H}_{\alpha}}d\tau 
        \\
        &
        +Cr\int_{0}^{t}e^{-\frac{\pi^2\left(t-\tau\right)}{2}}\left\|\mathbf{u}-\widetilde{\mathbf{u}}\right\|_{\mathbf{H}_{\alpha}}d\tau 
        +Cr
        \int_{0}^{t}\left(t-\tau\right)^{-\frac{1}{2}}e^{-\frac{\pi^2 \left(t-\tau\right)}{2}}
        \left\|\mathbf{u}-\widetilde{\mathbf{u}}\right\|_{\mathbf{H}_{\alpha}}d\tau
        \\
        &+Cr\int_{0}^{t}\left(t-\tau\right)^{-\frac{1}{2}}
        e^{-\frac{\pi^2\left(t-\tau\right)}{2}}\left\|\theta-\widetilde{\theta}\right\|_{H_{\alpha}}d\tau,
    \end{aligned}
    \end{aligned}
\end{align}
where $C=C\left(\Omega,\left\|\theta_{0}\right\|_{H_{\alpha}},\left\|\widetilde{\theta}_{0}\right\|_{H_{\alpha}},R\right)>0$.
\end{lemma}
\begin{proof}

    Since 
    \begin{align*}
    \begin{aligned}
        \theta\left(t\right)-\widetilde{\theta}\left(t\right)=&e^{t\Delta}\left(\theta_{0}-\widetilde{\theta}_{0}\right)
        +R\int_{0}^{t}e^{\left(t-\tau\right)\Delta}\left(u_{2}\left(\tau\right)-\widetilde{u}_{2}\left(\tau\right)\right)d\tau
\\
&-\int_{0}^{t}e^{\left(t-\tau\right)\Delta}\left(\mathbf{u}\left(\tau\right)-\widetilde{\mathbf{u}}\left(\tau\right)\right)\cdot\nabla\theta\left(\tau\right) d\tau
\\
&-\int_{0}^{t}e^{\left(t-\tau\right)\Delta}\widetilde{\mathbf{u}}\left(\tau\right)\cdot\nabla\left(\theta\left(\tau\right)-\widetilde{\theta}\left(\tau\right)\right)d\tau,
    \end{aligned}
\end{align*}

then, using $\left\|\mathbf{\Phi}\left(t\right)\right\|_{\mathbf{X}_{\alpha}}\leq 2r$ and $\left\|\mathbf{\widetilde{\Phi}}\left(t\right)\right\|_{\mathbf{X}_{\alpha}}\leq 2r$, we have 
\begin{align}\label{disange0116}
    \begin{aligned}
        \left\|\theta\left(t\right)-\widetilde{\theta}\left(t\right)\right\|_{H_{1}}
        \leq &C t^{\alpha-1}e^{-\frac{\pi^2 t}{2}}+R\int_{0}^{t}\left(t-\tau\right)^{\alpha-1}e^{-\frac{\pi^2\left(t-\tau\right)}{2}}\left\|\mathbf{u}-\widetilde{\mathbf{u}}\right\|_{\mathbf{H}_{\alpha}}d\tau
        \\
        &+C\int_{0}^{t}\left(t-\tau\right)^{\alpha-\frac{3}{2}}e^{-\frac{\pi^2\left(t-\tau\right)}{2}}\left\|\left(\mathbf{u}-\mathbf{\widetilde{u}}\right)\cdot\nabla\theta\right\|_{H_{\alpha-\frac{1}{2}}}d\tau
        \\
        &+Cr\int_{0}^{t}\left(t-\tau\right)^{\alpha-\frac{3}{2}}e^{-\frac{\pi^2\left(t-\tau\right)}{2}}\left\|\theta-\widetilde{\theta}\right\|_{H_{\alpha}}d\tau
        \\
        \leq& Ct^{\alpha-1}e^{-\frac{\pi^2 t}{2}}
        +R\int_{0}^{t}\left(t-\tau\right)^{\alpha-1}e^{-\frac{\pi^2\left(t-\tau\right)}{2}}\left\|\mathbf{u}-\widetilde{\mathbf{u}}\right\|_{\mathbf{H}_{\alpha}}d\tau
        \\
        &+Cr\int_{0}^{t}\left(t-\tau\right)^{\alpha-\frac{3}{2}}e^{-\frac{\pi^2\left(t-\tau\right)}{2}}\left\|\mathbf{u}-\widetilde{\mathbf{u}}\right\|_{\mathbf{H}_{\alpha}}d\tau
        \\
        &+Cr\int_{0}^{t}\left(t-\tau\right)^{\alpha-\frac{3}{2}}e^{-\frac{\pi^2\left(t-\tau\right)}{2}}\left\|\theta-\widetilde{\theta}\right\|_{H_{\alpha}}d\tau.
    \end{aligned}
\end{align}

Furthermore, we have 
\begin{align}\label{jiushizheyang0116}
    \begin{aligned}
        \left\|\theta\left(t\right)-\widetilde{\theta}\left(t\right)\right\|_{H_{\alpha+\frac{1}{2}}}\leq &C t^{-\frac{1}{2}}e^{-\frac{\pi^2 t
        }{2}}+R\int_{0}^{t}\left(t-\tau\right)^{-\frac{1}{2}}e^{-\frac{\pi^2\left(t-\tau\right)}{2}}\left\|\mathbf{u}-\widetilde{\mathbf{u}}\right\|_{\mathbf{H}_{\alpha}}d\tau 
        \\
        &+C\int_{0}^{t}\left(t-\tau\right)^{-\frac{1}{2}}e^{-\frac{\pi^2\left(t-\tau\right)}{2}}\left\|\left(\mathbf{u}-\widetilde{\mathbf{u}}\right)\cdot\nabla\theta\right\|_{H_{\alpha}}d\tau 
        \\
        &+Cr\int_{0}^{t}\left(t-\tau\right)^{-\alpha}e^{-\frac{\pi^2\left(t-\tau\right)}{2}}\left\|\theta-\widetilde{\theta}\right\|_{H_{1}}d\tau
        \\
        \leq & C \left(t^{-\frac{1}{2}}+1\right)e^{-\frac{\pi^2 t
        }{2}}+C\int_{0}^{t}\left(t-\tau\right)^{-\frac{1}{2}}e^{-\frac{\pi^2\left(t-\tau\right)}{2}}\left\|\mathbf{u}-\widetilde{\mathbf{u}}\right\|_{\mathbf{H}_{\alpha}}d\tau 
        \\
        &
        +Cr\int_{0}^{t}e^{-\frac{\pi^2\left(t-\tau\right)}{2}}\left\|\mathbf{u}-\widetilde{\mathbf{u}}\right\|_{\mathbf{H}_{\alpha}}d\tau 
        \\
        &+Cr
        \int_{0}^{t}\left(t-\tau\right)^{-\frac{1}{2}}e^{-\frac{\pi^2 \left(t-\tau\right)}{2}}
        \left\|\mathbf{u}-\widetilde{\mathbf{u}}\right\|_{\mathbf{H}_{\alpha}}d\tau
        \\
        &+Cr\int_{0}^{t}\left(t-\tau\right)^{-\frac{1}{2}}
        e^{-\frac{\pi^2\left(t-\tau\right)}{2}}\left\|\theta-\widetilde{\theta}\right\|_{H_{\alpha}}d\tau,
    \end{aligned}
\end{align}
where \eqref{disange0116} are used.
\end{proof}

Next, we provide the proof for \autoref{bijin0727}.
\begin{proof} \textup{(}proof for \autoref{bijin0727}\textup{)}
    Without loss generality, let $t_{0}=0$ in \eqref{suanzixing1010} and $\Phi_{0}=\xi+\eta,$ where $\xi\in\mathbf{X}_{c}$ and $\eta\in \mathbf{X}_{s}$. Then, we can obtain 
    \begin{align}\label{yixiazi1221}
    \begin{aligned}
\Phi_{c}\left(t\right)=T\left(t\right)\xi+\int_{0}^{t}T\left(t-\tau\right)\mathbf{P}_{c}\widetilde{\mathbf{G}}\left(\Phi_{c}\left(\tau\right)+\Phi_{s}\left(\tau\right)\right)d\tau,
\\
\Phi_{s}\left(t\right)=T\left(t\right)\eta+\int_{0}^{t}T\left(t-\tau\right)\mathbf{P}_{s}\widetilde{\mathbf{G}}\left(\Phi_{c}\left(\tau\right)+\Phi_{s}\left(\tau\right)\right)d\tau.
    \end{aligned}
    \end{align}
In addition, for any $\xi\in\mathbf{X}_{c}$, there exists a $\Phi\left(t;\xi\right)\in C_{\gamma}^{-}$ such that 
\begin{align}\label{jiushi1221}
    \Phi\left(t;\xi\right)=T\left(t\right)\xi-\int_{t}^{0}T\left(t-\tau\right)\mathbf{P}_{c}\widetilde{\mathbf{G}}\left(\Phi\left(\tau;\xi\right)\right)d\tau+
    \int_{-\infty}^{t}T\left(t-\tau\right)\mathbf{P}_{s}\widetilde{\mathbf{G}}\left(\Phi\left(\tau;\xi\right)\right)d\tau.
\end{align}
And $\Phi\left(t;\xi\right)\in M\left(R\right)$ for any $t>0$. Then, according to \eqref{xuzai1215}, we have 
\begin{align*}
\Phi\left(t;\xi\right)=\mathbf{P}_{c}\Phi\left(t;\xi\right)+h\left(\mathbf{P}_{c}\Phi\left(t;\xi\right)\right),
\end{align*}
which together with \eqref{jiushi1221} yields that 
\begin{align}\label{mei1221}
\begin{aligned}
h\left(\mathbf{P}_{c}\Phi\left(t;\xi\right)\right)&=\int_{-\infty}^{t}T\left(t-\tau\right)\mathbf{P}_{s}\widetilde{\mathbf{G}}\left(\Phi\left(\tau;\xi\right)\right)d\tau
\\
&=
\int_{-\infty}^{t}T\left(t-\tau\right)\mathbf{P}_{s}\widetilde{\mathbf{G}}\left(\mathbf{P}_{c}\Phi\left(\tau;\xi\right)+h\left(\mathbf{P}_{c}\Phi\left(\tau;\xi\right)\right)\right)d\tau
\end{aligned}
\end{align}
Thus, 
\begin{align}\label{meizi1221}
\begin{aligned}
h\left(\Phi_{c}\left(t\right)\right)=
\int_{-\infty}^{t}T\left(t-\tau\right)\mathbf{P}_{s}\widetilde{\mathbf{G}}\left(\Phi_{c}\left(\tau\right)+h\left(\Phi_{c}\left(\tau\right)\right)\right)d\tau.
\end{aligned}
\end{align}
Let $v\left(t\right)=\Phi_{s}\left(t\right)-h\left(\Phi_{c}\left(t\right)\right)$, then a direct computation gives that 
\begin{align}\label{beiai1222}
\begin{aligned}
    v\left(t\right)
    =I_{1}+I_{2}+I_{3},
\end{aligned}
\end{align}
where 
\begin{align*}
    \begin{aligned}
        &I_{1}=T\left(t\right)\eta,~I_{2}=-\int_{-\infty}^{0}T\left(t-\tau\right)\mathbf{P}_{s}\widetilde{\mathbf{G}}\left(\Phi_{c}\left(\tau\right)+h\left(\Phi_{c}\left(\tau\right)\right)\right)d\tau,
        \\
        &I_{3}=\int_{0}^{t}T\left(t-\tau\right)\left[\mathbf{P}_{s}\widetilde{\mathbf{G}}\left(\Phi_{c}\left(\tau\right)+\Phi_{s}\left(\tau\right)\right)-\mathbf{P}_{s}\widetilde{\mathbf{G}}\left(\Phi_{c}\left(\tau\right)+h\left(\Phi_{c}\left(\tau\right)\right)\right)\right]d\tau.
    \end{aligned}
\end{align*}
First, we prove that $\left\|v\left(t\right)\right\|_{\mathbf{X}_{\alpha}}$ is bounded for any $t\geq 0$.  Assume that $\Phi_{c}\left(t\right)+\Phi_{s}\left(t\right)=\left(\mathbf{u}\left(t\right),\theta\left(t\right)\right)^{T}$ and $\Phi_{c}\left(t\right)+h\left(\Phi_{c}\left(t\right)\right)=\left(\widetilde{\mathbf{u}}\left(t\right),\widetilde{\theta}\left(t\right)\right)^{T}$. In addition, $\left\|\mathbf{\Phi}\left(t\right)\right\|_{\mathbf{X}_{\alpha}}\leq 2r$ and $\left\|\mathbf{\widetilde{\Phi}}\left(t\right)\right\|_{\mathbf{X}_{\alpha}}\leq 2r$ for any $t\geq 0$.

For $I_{1}$ and $I_{2}$, we have 
\begin{align}\label{yige1222}
\begin{aligned}
    &\left\|I_{1}\right\|_{\mathbf{X}_{\alpha}}\leq e^{-\gamma_{0}t}\left\|\Phi_{0}\right\|_{\mathbf{X}_{\alpha}},
    \\
    &
    \begin{aligned}\left\|I_{2}\right\|_{\mathbf{X}_{\alpha}}&\leq \int_{-\infty}^{0}\left\|T\left(t-\tau\right)\mathbf{P}_{s}\widetilde{\mathbf{G}}\left(\Phi_{c}\left(\tau\right)+h\left(\Phi_{c}\left(\tau\right)\right)\right)\right\|_{\mathbf{X}_{\alpha}}
    d\tau
    \\
    &\leq \frac{\widetilde{C}\left(r\right)e^{-\gamma_{0}t}\left\|\Phi_{c}\left(t\right)+h\left(\Phi_{c}\left(t\right)\right)\right\|_{C_{\gamma}^{-}}}{\gamma_{0}-\gamma}.
    \end{aligned}
\end{aligned}
\end{align}
For $I_{3}$, a direct computation gives that 
\begin{align}\label{liaoliao0116}
    \begin{aligned}
        \left\|\mathbf{u}\cdot\nabla\theta-\mathbf{\widetilde{u}}\cdot\nabla\widetilde{\theta}\right\|_{H_{\alpha}}&\leq \left\|\left(\mathbf{u}-\mathbf{\widetilde{u}}\right)\cdot\nabla\theta\right\|_{H_{\alpha}}+\left\|\widetilde{\mathbf{u}}\cdot\nabla\left(\theta-\widetilde{\theta}\right)\right\|_{H_{\alpha}}
        \\
        &\leq Cr\left(\left\|\theta\right\|_{H_{\alpha+\frac{1}{2}}}+\left\|\theta-\widetilde{\theta}\right\|_{H_{\alpha+\frac{1}{2}}}
    \right)
    \\
    &\leq Ct^{-\frac{1}{2}}e^{-\frac{\pi^2 t}{2}}+Cr,
    \end{aligned}
\end{align}
where \eqref{shide0105} and \eqref{xuehui0104} are used. Then,
\begin{align}\label{yong1222}
    \begin{aligned}
        \left\|I_{3}\right\|_{\mathbf{X}_{\alpha}}&\leq\int_{0}^{t}\left\|
        T\left(t-\tau\right)\left[\widetilde{\mathbf{G}}\left(\Phi_{c}\left(\tau\right)+\Phi_{s}\left(\tau\right)\right)-\widetilde{\mathbf{G}}\left(\Phi_{c}\left(\tau\right)+h\left(\Phi_{c}\left(\tau\right)\right)\right)\right]\right\|_{\mathbf{X}_{\alpha}}d\tau
        \\
        &\leq\int_{0}^{t}\left\|
        e^{-\pi^2\left(t-\tau\right)}\left[\mathbf{u}\cdot\nabla\theta-\widetilde{\mathbf{u}}\cdot\nabla\widetilde{\theta}\right]\right\|_{H_{\alpha}}d\tau
        \\
        &\leq C
        \int_{0}^{t}e^{-\pi^2\left(t-\tau\right)}\left(\tau^{-\frac{1}{2}}e^{-\frac{\pi^2\tau}{2}}+r\right)d\tau\leq C.
    \end{aligned}
\end{align}

Then from \eqref{beiai1222}-\eqref{yong1222}, we have 
\begin{align}\label{faxian1222}
\begin{aligned}
\left\|v\left(t\right)\right\|_{\mathbf{X}_{\alpha}}\leq C,~\forall ~t\geq 0.
\end{aligned}
\end{align}
Notice that when $t\rightarrow+\infty$, $\left\|v\left(t\right)\right\|_{\mathbf{X}_{\alpha}}$ can be very small since $r$ is small sufficiently.

Subsequently, it is our turn to consider the tendency of $\left\|v\left(t\right)\right\|_{\mathbf{X}_{\alpha}}$ as $t\rightarrow+\infty$ for the case when $\left\|\mathbf{\Phi}\right\|_{\mathbf{X}_{\alpha}}\leq 2r$ and $\left\|\widetilde{\mathbf{\Phi}}\right\|_{\mathbf{X}_{\alpha}}\leq 2r$. 

A direct computation gives that 
\begin{align}\label{jiushi0165}
    \begin{aligned}
        \left\|\mathbf{u}\cdot\nabla\theta-\widetilde{\mathbf{u}}\cdot\nabla\widetilde{\theta}\right\|_{H_{\alpha}}\leq 
        C\left[\left\|\theta\right\|_{H_{\alpha+\frac{1}{2}}}\left\|\mathbf{u}-\widetilde{\mathbf{u}}\right\|_{\mathbf{H}_{\alpha}}+r\left\|\theta-\widetilde{\theta}\right\|_{H_{\alpha+\frac{1}{2}}}\right].
    \end{aligned}
\end{align}
Thus, 
\begin{align}\label{xunzhao0116}
    \begin{aligned}
        \left\|I_{3}\right\|_{\mathbf{X}_{\alpha}}
        \leq C\int_{0}^{t}
        e^{-\pi^2\left(t-\tau\right)}\left[\left\|\theta\right\|_{H_{\alpha+\frac{1}{2}}}\left\|\mathbf{u}-\widetilde{\mathbf{u}}\right\|_{\mathbf{H}_{\alpha}}+r\left\|\theta-\widetilde{\theta}\right\|_{H_{\alpha+\frac{1}{2}}}\right]d\tau.
    \end{aligned}
\end{align}
Then, according to \eqref{shide0105} and $\eqref{dierge01166}_{2}$, we have 
\begin{align*}
\begin{aligned}
     &\left\|I_{3}\right\|_{\mathbf{X}_{\alpha}}\leq 
     Cr\int_{0}^{t}e^{-\pi^2\left(t-\tau\right)}\left\|v\left(\tau\right)\right\|_{\mathbf{X}_{\alpha}}d\tau+Ce^{-\pi^2 t}\int_{0}^{t}\tau^{-\frac{1}{2}}\left\|v\left(\tau\right)\right\|_{\mathbf{X}_{\alpha}}d\tau+Cr e^{-\frac{\pi^2 t}{2}}
\\
&+Cr\int_{0}^{t}\left(t-\tau\right)^{-\frac{1}{2}}e^{-\frac{\pi^2\left(t-\tau\right)}{2}}\left\|v\left(\tau\right)\right\|_{\mathbf{X}_{\alpha}}d\tau
+Cr\int_{0}^{t}\left(t-\tau\right)e^{-\frac{\pi^2\left(t-\tau\right)}{2}}\left\|v\left(\tau\right)\right\|_{\mathbf{X}_{\alpha}}d\tau
\end{aligned}
\end{align*}
which together with \eqref{beiai1222} and \eqref{yige1222} yields that 
\begin{align*}
    \begin{aligned}
        &\left\|v\left(t\right)\right\|_{\mathbf{X}_{\alpha}}
        \leq 
     Cr\int_{0}^{t}e^{-\pi^2\left(t-\tau\right)}\left\|v\left(\tau\right)\right\|_{\mathbf{X}_{\alpha}}d\tau+Ce^{-\pi^2 t}\int_{0}^{t}\tau^{-\frac{1}{2}}\left\|v\left(\tau\right)\right\|_{\mathbf{X}_{\alpha}}d\tau+C e^{-\gamma_{0}t}
\\
&+Cr\int_{0}^{t}\left(t-\tau\right)^{\frac{1}{2}}e^{-\frac{\pi^2\left(t-\tau\right)}{2}}\left\|v\left(\tau\right)\right\|_{\mathbf{X}_{\alpha}}d\tau
+Cr\int_{0}^{t}\left(t-\tau\right)e^{-\frac{\pi^2\left(t-\tau\right)}{2}}\left\|v\left(\tau\right)\right\|_{\mathbf{X}_{\alpha}}d\tau.
    \end{aligned}
\end{align*}
Since $\left\|v\right\|_{\mathbf{X}_{\alpha}}$ is bounded and $t^{q}e^{-bt}$ is bounded on $\left[0,+\infty\right)$ for any $q>0$ and $b>0$, we deduce 
\begin{align*}
    \begin{aligned}
        \left\|v\left(t\right)\right\|_{\mathbf{X}_{\alpha}}
        \leq Ce^{-\gamma_{0} t}+
        Cr\int_{0}^t e^{-\gamma_{0}\left(t-\tau\right)}\left\|v\left(\tau\right)\right\|_{\mathbf{X}_{\alpha}}d\tau\leq Ce^{-\gamma_{0}t}+Cr.
\end{aligned}
\end{align*}
Putting $\left\|v\left(t\right)\right\|_{\mathbf{X}_{\alpha}}\leq Ce^{-\gamma_{0}t}+Cr$ into $\left\|v\left(t\right)\right\|_{\mathbf{X}_{\alpha}}
        \leq Ce^{-\gamma_{0} t}+
        Cr\int_{0}^t e^{-\gamma_{0}\left(t-\tau\right)}\left\|v\left(\tau\right)\right\|_{\mathbf{X}_{\alpha}}d\tau$, we have 
\begin{align}\label{fanfu0627}
\begin{aligned}
    \left\|v\left(t\right)\right\|_{\mathbf{X}_{\alpha}}
        \leq C e^{-\gamma_{0}t}+C^2 r e^{-\frac{\gamma_{0}t}{2}}+C^2 r^2,
\end{aligned}
\end{align}
which together with $\left\|v\left(t\right)\right\|_{\mathbf{X}_{\alpha}}
        \leq Ce^{-\gamma_{0} t}+
        Cr\int_{0}^t e^{-\gamma_{0}\left(t-\tau\right)}\left\|v\left(\tau\right)\right\|_{\mathbf{X}_{\alpha}}d\tau$ yields that 
\begin{align}\label{fanfanfufu0627}
    \begin{aligned}
        \left\|v\right\|_{\mathbf{X}_{\alpha}}
        \leq C e^{-\gamma_{0}t}+C^2 r e^{-\frac{\gamma_{0}t}{2}}+C^3 r^3.
    \end{aligned}
\end{align}
Since $r$ can be sufficiently small, it follows from the iterative method applied to $\left\|v\left(t\right)\right\|_{\mathbf{X}_{\alpha}}
        \leq Ce^{-\gamma_{0} t}+
        Cr\int_{0}^t e^{-\gamma_{0}\left(t-\tau\right)}\left\|v\left(\tau\right)\right\|_{\mathbf{X}_{\alpha}}d\tau$ that 
$$\left\|v\right\|_{\mathbf{X}_{\alpha}}\leq C e^{-\frac{\gamma_{0}t}{2}}.$$
Thus, one can conclude that        
there exists a positive constant $\alpha^{*}$ such that 
$$
\lim\limits_{t\rightarrow+\infty}
e^{\alpha^{*}t}\left\|\Phi_{s}\left(t\right)-h\left(\Phi_{c}\left(t\right)\right)\right\|_{\mathbf{X}_{\alpha}}=0.
$$
\end{proof}


\subsection{The approximation formula for local invariant manifold}\label{jinsi0205}
From the preceding investigation, we have established the existence of a local invariant manifold for system \eqref{model}–\eqref{bianjian0814} when the velocity field $\left\|\mathbf{\Phi}\right\|_{\mathbf{X}_{\alpha}}$ is sufficiently small. Furthermore, by \autoref{zhongxinliuxing07276}, for sufficiently large time, the dynamical behavior of small solutions to system \eqref{model}–\eqref{bianjian0814} can be viewed as evolving on an invariant manifold. Consequently, deriving an approximation formula for the local invariant manifold is of great importance. Subsequently, we will give the approximation formula for $h\left(\xi\right)$ in two scenarios: (1) $\text{Card}(B)=1$ and (2) $\text{Card}(B)=2$. That is, $\xi=z\Phi_{k_{0},1}^{1}$ or $\xi=z_{1}\Phi_{k_{0},1}^{1}+z_{2}\Phi_{k_{0}+1,1}^{1}$.

First, for the solution $\Phi\left(t\right)$ to system \eqref{model}–\eqref{bianjian0814} on invariant manifold, it can be rewritten in forms 
\begin{align*}    \Phi\left(t\right)=\sum\limits_{i=1}^{m}q_{i}\left(t\right)\Phi_{k_{i},1}^{i}+\Psi_{2}\left(t\right)=\Psi_{1}+h\left(\Psi_{1}\right),~\text{where}~\Psi_{1}\in\mathbf{X}_{c}~\text{and}~\Psi_{2}\in\mathbf{X}_{s}.
\end{align*}
By \eqref{xuyao1222}, $h\left(\xi\right)$ is a higher-term of $\xi$, as $\xi\rightarrow 0$. Thus, for the following system with small solution
\begin{align}\label{lailai0209}
    \begin{cases}
        \frac{d\Psi_{1}}{dt}=\mathbf{P}_{c}\mathbf{L}\Psi_{1}+\mathbf{P}_{c}\widetilde{\mathbf{G}}\left(\Psi_{1}+h\left(\Psi_{1}\right)\right),
        \\
        \frac{d\Psi_{2}}{dt}=\mathbf{P}_{s}\mathbf{L}\Psi_{2}+\mathbf{P}_{s}\widetilde{\mathbf{G}}\left(\Psi_{1}+h\left(\Psi_{1}\right)\right),
        \\
        \Psi_{1}\left(0\right)=\sum\limits_{i=1}^{m}z_{i}\Phi_{k_{i},1}^{1},~\Psi_{2}\left(0\right)=h\left(\Psi_{1}\left(0\right)\right),
    \end{cases}
\end{align}
we obtain
\begin{align*}
    \Phi\left(t\right)\approx\Psi_{1}\left(t\right)=\sum\limits_{i=1}^{m}z_{i}e^{\lambda_{k_{i},1}^{1}t}\Phi_{k_{i},1}^{1}+o.
\end{align*}
Thus, from the expression of $h\left(\xi\right)$ as in \eqref{hdedansheng1215}, for system \eqref{lailai0209}, the approximate formula   
\begin{align}\label{bijin0209}
    \begin{aligned}
        h\left(\sum\limits_{i=1}^{m}z_{i}\Phi_{k_{i},1}^{1}\right)&=\int_{-\infty}^{0}T\left(-\tau\right)\mathbf{P}_{s}\widetilde{\mathbf{G}}\left(\Phi\left(\tau\right)\right)d\tau
        \\
        &\approx \int_{-\infty}^{0}T\left(-\tau\right)\mathbf{P}_{s}\widetilde{\mathbf{G}}\left(\sum\limits_{i=1}^{m}z_{i}e^{\lambda_{k_{i},1}^{1}\tau}\Phi_{k_{i},1}^{1}\right)d\tau+o(2).
        \end{aligned}
\end{align}

We now derive the approximate formulas for invariant manifolds corresponding to two cases: $m=1$ and $m=2$. 

(1) When $m=1$, $\left(\mathbf{u},\theta\right)^{T}=ze^{\lambda_{k_{0},1}^{1}t}\Phi_{k_{0},{1}}^{1}=\frac{2ze^{\lambda_{k_{0},1}^{1}t}}{\sqrt{\left[1+\left(u_{k_{0},1}^1\right)^2\gamma_{k_{0}1}^2\right]L}}
\begin{pmatrix}
    u_{k_{0},1}^{1}\pi \sin{\frac{k_{0}\pi x_{1}}{L}}\cos{\pi x_{2}}
    \\
    -u_{k_{0},1}^{1}\frac{k_{0}\pi}{L} \cos{\frac{k_{0}\pi x_{1}}{L}}\sin{\pi x_{2}}
    \\
    \cos{\frac{k_{0}\pi x_{1}}{L}}\sin{\pi x_{2}}
\end{pmatrix}$,
$$
\widetilde{\mathbf{G}}\left(ze^{\lambda_{k_{0},1}^{1}t}\Phi_{k_{0},{1}}^{1}\right)=-\mathbf{u}\cdot\nabla\theta=G_{1}e^{2\lambda_{k_{0},1}^{1}t}z^2\begin{pmatrix}
  0
  \\
  0
  \\
\sin{2\pi x_{2}}\end{pmatrix},
$$
where $G_{1}=\frac{2k_{0}\pi^2 u_{k_{0},1}^1}{L^2\left(1+\left(u_{k_{0},1}^{1}\right)^2\gamma_{k_{0},1}^2\right)}$. Then, from \eqref{bijin0209}, it follows that 
\begin{align}\label{gongshi0209}
    \begin{aligned}
        h\left(z \Phi_{k_{0},1}^{1}\right)\approx\frac{G_{1}z^2\sqrt{L}}{\sqrt{2}}\Phi_{0,2}\int_{-\infty}^{0}e^{\left(2\lambda_{k_{0},1}^{1}+4\pi^2\right)\tau}d\tau=\frac{G_{1}z^2\sqrt{L}}{\sqrt{2}\left(2\lambda_{k_{0},1}^{1}+4\pi^2\right)}\Phi_{0,2}.
    \end{aligned}
\end{align}

(2)When $m=2$, 
\begin{align*}
    \begin{aligned}
        \left(\mathbf{u},\theta\right)^{T}&=z_{1}e^{\lambda_{k_{0},1}^{1}t}\Phi_{k_{0},{1}}^{1}+z_{2}e^{\lambda_{k_{0}+1,1}^{1}t}\Phi_{k_{0}+1,{1}}^{1}
        \\
        &=\frac{2z_{1}e^{\lambda_{k_{0},1}^{1}t}}{\sqrt{\left[1+\left(u_{k_{0},1}^1\right)^2\gamma_{k_{0}1}^2\right]L}}
\begin{pmatrix}
    u_{k_{0},1}^{1}\pi \sin{\frac{k_{0}\pi x_{1}}{L}}\cos{\pi x_{2}}
    \\
    -u_{k_{0},1}^{1}\frac{k_{0}\pi}{L} \cos{\frac{k_{0}\pi x_{1}}{L}}\sin{\pi x_{2}}
    \\
    \cos{\frac{k_{0}\pi x_{1}}{L}}\sin{\pi x_{2}}
\end{pmatrix}
\\
&~~~~
+\frac{2z_{2}e^{\lambda_{k_{0}+1,1}^{1}t}}{\sqrt{\left[1+\left(u_{k_{0}+1,1}^1\right)^2\gamma_{(k_{0}+1)1}^2\right]L}}
\begin{pmatrix}
    u_{k_{0}+1,1}^{1}\pi \sin{\frac{(k_{0}+1)\pi x_{1}}{L}}\cos{\pi x_{2}}
    \\
    -u_{k_{0}+1,1}^{1}\frac{(k_{0}+1)\pi}{L} \cos{\frac{(k_{0}+1)\pi x_{1}}{L}}\sin{\pi x_{2}}
    \\
    \cos{\frac{(k_{0}+1)\pi x_{1}}{L}}\sin{\pi x_{2}}
\end{pmatrix},
    \end{aligned}
\end{align*}
\begin{align}\label{suanwan0210}
    \begin{aligned}
\widetilde{\mathbf{G}}&\left(z_{1}e^{\lambda_{k_{0},1}^{1}t}\Phi_{k_{0},{1}}^{1}+z_{2}e^{\lambda_{k_{0}+1,1}^{1}t}\Phi_{k_{0}+1,{1}}^{1}\right)=\left(G_{11}e^{2\lambda_{k_{0},1}^{1}t}z_{1}^2+G_{22}e^{2\lambda_{k_{0}+1,1}^{1}t}z_{2}^2\right)\Phi_{0,2}
\\
&+z_{1}z_{2}e^{\left(\lambda_{k_{0},1}^{1}+\lambda_{k_{0}+1,1}^{1}\right)t}\left(G_{12}^{1}\Phi_{1,2}^{1}+G_{12}^{2}\Phi_{1,2}^{2}+G_{21}^{1}\Phi_{2k_{0}+1,1}^{1}+G_{21}^{2}\Phi_{2k_{0}+1,1}^{2}\right),
    \end{aligned}
\end{align}
where 
\begin{align*}
    \begin{aligned}
        &G_{11}=\frac{\sqrt{2}k_{0}\pi^2 u_{k_{0},1}^{1}}{\sqrt{L}A^2},~
        G_{22}=\frac{\sqrt{2}(k_{0}+1)\pi^2 u_{k_{0}+1,1}^{1}}{\sqrt{L}B^2},
        \\
        & A=\sqrt{\left[1+\left(u_{k_{0},1}^{1}\right)^2\gamma_{k_{0}1}^2\right]L},~B=\sqrt{\left[1+\left(u_{k_{0}+1,1}^{1}\right)^2\gamma_{(k_{0}+1)1}^2\right]L},
        \\
        &G_{12}^{1}=\frac{\left(2k_{0}+1\right)\pi^2\left(u_{k_{0},1}^{1}+u_{k_{0}+1,1}^{1}\right)}{L AB}*\frac{u_{1,2}^{2}\sqrt{\left[1+\left(u_{1,2}^{1}\right)^2\gamma_{12}^{2}\right]L}}{2\left(u_{1,2}^{2}-u_{1,2}^{1}\right)},
        \\
        &G_{12}^{2}=\frac{\left(2k_{0}+1\right)\pi^2\left(u_{k_{0},1}^{1}+u_{k_{0}+1,1}^{1}\right)}{L AB}*\frac{u_{1,2}^{1}\sqrt{\left[1+\left(u_{1,2}^{2}\right)^2\gamma_{12}^{2}\right]L}}{2\left(u_{1,2}^{1}-u_{1,2}^{2}\right)},
        \\
        \end{aligned}
        \end{align*}
        \begin{align*}
        \begin{aligned}
        &G_{21}^{1}=\frac{\pi^2\left(u_{k_{0}+1,1}^{1}-u_{k_{0},1}^{1}\right)}{LAB}*\frac{u_{2k_{0}+1,2}^{2}\sqrt{\left[1+\left(u_{2k_{0}+1,2}^{1}\right)^2\gamma_{\left(2k_{0}+1\right)2}^{2}\right]L}}{2\left(u_{2k_{0}+1,2}^{2}-u_{2k_{0}+1,2}^{1}\right)},
        \\
        &G_{21}^{2}=\frac{\pi^2\left(u_{k_{0}+1,1}^{1}-u_{k_{0},1}^{1}\right)}{L AB}*\frac{u_{2k_{0}+1,2}^{1}\sqrt{\left[1+\left(u_{2k_{0}+1,2}^{2}\right)^2\gamma_{\left(2k_{0}+1\right)2}^{2}\right]L}}{2\left(u_{2k_{0}+1,2}^{1}-u_{2k_{0}+1,2}^{2}\right)}.
    \end{aligned}
\end{align*}
Then, 
\begin{align}\label{gongshi0210}
    \begin{aligned}
        h&\left(z_{1}\Phi_{k_{0},{1}}^{1}+z_{2}\Phi_{k_{0}+1,{1}}^{1}\right)\approx \left(\frac{G_{11}z_{1}^2}{2\lambda_{k_{0},1}^{1}+4\pi^2}+\frac{G_{22}z_{2}^2}{2\lambda_{k_{0}+1,1}^{1}+4\pi^2}\right)\Phi_{0,2}
        \\
        &+z_{1}z_{2}\left(
        \frac{G_{12}^{1}\Phi_{1,2}^{1}}{\lambda_{k_{0},1}^{1}+\lambda_{k_{0}+1,1}^{1}+\lambda_{1,2}^{1}}+\frac{G_{12}^{2}\Phi_{1,2}^{2}}{\lambda_{k_{0},1}^{1}+\lambda_{k_{0}+1,1}^{1}+\lambda_{1,2}^{2}}
        \right)
        \\
        &+z_{1}z_{2}\left(
        \frac{G_{21}^{1}\Phi_{2k_{0}+1,2}^{1}}{\lambda_{k_{0},1}^{1}+\lambda_{k_{0}+1,1}^{1}+\lambda_{2k_{0}+1,2}^{1}}+\frac{G_{21}^{2}\Phi_{2k_{0}+1,2}^{2}}{\lambda_{k_{0},1}^{1}+\lambda_{k_{0}+1,1}^{1}+\lambda_{2k_{0}+1,2}^{2}}
        \right).
    \end{aligned}
\end{align}
\section{Reduction and bifurcation}\label{yuehua0205}
In this section, the center manifold reduction method is employed to perform the reduction of the system \eqref{model}-\eqref{bianjian0814}, yielding a finite-dimensional ordinary differential equation (ODE). The dynamical behavior of this ODE in the neighborhood of zero is equivalent to that of system \eqref{model}-\eqref{bianjian0814} near the origin. Through the analysis of the resulting ODE system, conclusions regarding the dynamical behavior of system \eqref{model}-\eqref{bianjian0814} are drawn.
\subsection{Reduction}\label{yuehua0211}
\begin{lemma}\label{yuehua0210}\rm{(}\textbf{Reduction}\rm{)} For $R$ in the vicinity of $R_{c}$, the stability and transition of the system \eqref{model}-\eqref{bianjian0814} with small solution is equivalent to these of the following ODEs:

\begin{enumerate}[label=(\arabic*)]
    \item  when $m=1$, the ODE is given by 
    \begin{align}\label{diyige0210}
        \begin{aligned}
            \frac{dq}{dt}=\lambda_{k_{0},1}^{1}q+aq^{3}+o,
        \end{aligned}
    \end{align}
    where $a=-\frac{u_{k_{0},1}^{1}u_{k_{0},1}^{2}k_{0}\pi^2 G_{1}}{2L\left(\lambda_{k_{0},1}^{1}+2\pi^2\right)\left(u_{k_{0},1}^{2}-u_{k_{0},1}^{1}\right)}$.
    \item when $m=2$, the ODEs are given by 
    \begin{align}\label{dierge0210}
        \begin{cases}
            \frac{dq_{1}}{dt}=\lambda_{k_{0},1}^{1}q_{1}+q_{1}\left(a_{11}q_{1}^{2}+a_{22}q_{2}^2\right)+o,
            \\
            \frac{dq_{2}}{dt}=\lambda_{k_{0}+1,1}^{1}q_{2}+q_{2}\left(b_{11}q_{1}^2+b_{22}q_{2}^2\right)+o,
        \end{cases}
    \end{align}
    where $a_{11}$, $a_{22}$, $b_{11}$, and $b_{22}$ can be found in \eqref{xishu0210}. 
\end{enumerate}
\end{lemma}
\begin{proof}
We have shown that the linear operator $\mathbf{L}$ has a complete eigenvector system, see Lemma \ref{wanbei0205}. Specifically, any solution $\Phi=\left(\mathbf{u},\theta\right)\in \mathbf{E}_{1}\times\widetilde{\mathbf{E}_{2}}$ to the system \eqref{model}-\eqref{bianjian0814} can be decomposed into 
$$
\Phi=\Psi_{1}+\Psi_{2},
$$
where $\Psi_{1}=\sum\limits_{i=1}^{m}q_{i}\left(t\right)\Phi_{k_{i},1}^{1}\in \mathbf{X}_{c}$, $\Psi_{2}\in\overline{\mathbf{X}}_{s} $ and $\overline{\mathbf{X}}_{s}$ is spanned by the rest of eigenvectors. In addition, under the small solution  assumption, $\Psi_{2}$ can be approximated by $h\left(\Psi_{1}\right)$. 
Thus, the dynamical behavior of the system \eqref{model}-\eqref{bianjian0814} is equivalent to these of the following ODEs:
\begin{align}\label{odes0210}
    \frac{dq_{i}}{dt}=\langle L\Phi,\widetilde{\Phi}_{k_{i},1}^{1}\rangle +\langle\mathbf{G}\left(\Psi_{1}+h\left(\Psi_{1}\right)\right),\widetilde{\Phi}_{k_{i},1}^{1} \rangle =\lambda_{k_{i},1}^{1}q_{i}+\langle\mathbf{G}\left(\Psi_{1}+h\left(\Psi_{1}\right)\right),\widetilde{\Phi}_{k_{i},1}^{1} \rangle,
\end{align}
where $m=1$ or $2$, and $\widetilde{\Phi}_{k_{i},1}^{1}$ can be found in Remark \ref{fangbian0205}.

    (1)When $m=1$, $\Psi_{1}=q\Phi_{k_{0},1}^{1}$, according to \eqref{gongshi0209}, 
    \begin{align}\label{zhongxin10210}
        \begin{aligned}
            h\left(q\Phi_{k_{0},1}^{1}\right)\approx\frac{q^2G_{1}\sqrt{L}\Phi_{0,2}}{\sqrt{2}\left(2\lambda_{k_{0},1}^{1}+4\pi^2\right)}=q^2\widetilde{G}_{1}\Phi_{0,2},~\widetilde{G}_{1}=\frac{G_{1}\sqrt{L}}{\sqrt{2}\left(2\lambda_{k_{0},1}^{1}+4\pi^2\right)}.
        \end{aligned}
    \end{align}
Substituting \eqref{zhongxin10210} into \eqref{odes0210}, we obtain 
\begin{align}\label{yijie0210}
    \frac{dq}{dt}=\lambda_{k_{0},1}^{1}q+a q^{3}+o,
\end{align}
where $a=-\frac{\sqrt{2}u_{k_{0},1}^{1}u_{k_{0},1}^{2}k_{0}\pi^2 \widetilde{G}_{1}}{L^{\frac{3}{2}}\left(u_{k_{0},1}^{2}-u_{k_{0},1}^{1}\right)}$.

(2)When $m=2$, $\Psi_{1}=q_{1}\Phi_{k_{0},1}^{1}+q_{2}\Phi_{k_{0}+1,1}^{1}$, according to \eqref{gongshi0210}, 
    \begin{align}\label{zhongxin20210}
        \begin{aligned}
            h&\left(q_{1}\Phi_{k_{0},{1}}^{1}+q_{2}\Phi_{k_{0}+1,{1}}^{1}\right)\approx \left(\frac{G_{11}q_{1}^2}{2\lambda_{k_{0},1}^{1}+4\pi^2}+\frac{G_{22}q_{2}^2}{2\lambda_{k_{0}+1,1}^{1}+4\pi^2}\right)\Phi_{0,2}
        \\
        &+q_{1}q_{2}\left(
        \frac{G_{12}^{1}\Phi_{1,2}^{1}}{\lambda_{k_{0},1}^{1}+\lambda_{k_{0}+1,1}^{1}+\lambda_{1,2}^{1}}+\frac{G_{12}^{2}\Phi_{1,2}^{2}}{\lambda_{k_{0},1}^{1}+\lambda_{k_{0}+1,1}^{1}+\lambda_{1,2}^{2}}
        \right)
        \\
        &+q_{1}q_{2}\left(
        \frac{G_{21}^{1}\Phi_{2k_{0}+1,2}^{1}}{\lambda_{k_{0},1}^{1}+\lambda_{k_{0}+1,1}^{1}+\lambda_{2k_{0}+1,2}^{1}}+\frac{G_{21}^{2}\Phi_{2k_{0}+1,2}^{2}}{\lambda_{k_{0},1}^{1}+\lambda_{k_{0}+1,1}^{1}+\lambda_{2k_{0}+1,2}^{2}}
        \right).
        \end{aligned}
    \end{align}
    Substituting \eqref{zhongxin20210} into \eqref{odes0210}, we obtain 
    \begin{align}\label{ODEs202010}
        \begin{cases}
            \frac{dq_{1}}{dt}=\lambda_{k_{0},1}^{1}q_{1}+q_{1}\left(a_{11}q_{1}^2+a_{22}q_{2}^{2}\right)+o,
            \\
            \frac{dq_{2}}{dt}=\lambda_{k_{0}+1,1}^{1}q_{2}+q_{2}\left(b_{11}q_{1}^2+b_{22}q_{2}^{2}\right)+o,
        \end{cases}
    \end{align}
    where 
    \begin{align}\label{xishu0210}
        \begin{aligned}
            &a_{11}=\frac{\widetilde{G}_{11}k_{0}\pi^2 u_{k_{0},1}^{1}u_{k_{0},1}^{2}}{\sqrt{2}L^{\frac{3}{2}}\left(u_{k_{0},1}^{1}-u_{k_{0},1}^2\right)},~b_{22}=\frac{\widetilde{G}_{22}(k_{0}+1)\pi^2 u_{k_{0}+1,1}^{1}u_{k_{0}+1,1}^{2}}{\sqrt{2}L^{\frac{3}{2}}\left(u_{k_{0}+1,1}^{1}-u_{k_{0}+1,1}^2\right)},
            \\
           & \begin{aligned}
                a_{22}&=\frac{\pi^2 u_{k_{0},1}^{2}}{2L^3\left(u_{k_{0},1}^{1}-u_{k_{0},1}^{2}\right)}\bigg{[}
                \sqrt{2}\widetilde{G}_{22}k_{0}L^{\frac{3}{2}}u_{k_{0},1}^{1}
                \\
                &+\frac{L^2A}{B}
                \sum\limits_{i=1}^{2}\bigg{(}
                \widetilde{G}_{12}^{i}+\widetilde{G}_{21}^{i}u_{2k_{0}+1,2}^{i}+\widetilde{G}_{21}^{i}u_{k_{0}+1,1}^{1}
                \bigg{)}
                \bigg{]},
            \end{aligned}
            \\
            &\begin{aligned}
                b_{11}&=\frac{\pi^2u_{k_{0}+1,1}^{2}}{2L^{3}\left(u_{k_{0}+1,1}^{1}-u_{k_{0}+1,1}^{2}\right)}\bigg{[}
                \sqrt{2}\widetilde{G}_{11}\left(k_{0}+1\right)L^{\frac{3}{2}} u_{k_{0}+1,1}^{1}
                \\
                &-\frac{L^2 B}{A}\sum\limits_{i=1}^{2}
                \bigg{(}
                \widetilde{G}_{121}^{i}+\widetilde{G}_{211}^{i}u_{2k_{0}+1,2}^{i}+\widetilde{G}_{211}^{i}u_{k_{0},1}^{1}
                \bigg{)}
                \bigg{]},
            \end{aligned}
        \end{aligned}
        \end{align}
        \begin{align*}
        \begin{aligned}            &\widetilde{G}_{11}=\frac{G_{11}}{\lambda_{k_{0},1}^{1}+2\pi^2},~\widetilde{G}_{22}=\frac{G_{22}}{\lambda_{k_{0}+1,1}^{1}+2\pi^2},
            \\  &\widetilde{G}_{12}^{i}=\frac{(2k_{0}+1)\left(u_{1,2}^{i}+u_{k_{0}+1,1}^{1}\right)G_{12}^{i}}{\left(\lambda_{k_{0},1}^{1}+\lambda_{k_{0}+1,1}^{1}+\lambda_{1,2}^{i}\right)\sqrt{L\left[1+\left(u_{1,2}^{i}\right)^2\gamma_{12}^2\right]}},
\\
&\widetilde{G}_{21}^{i}=\frac{G_{21}^{i}}{\left(\lambda_{k_{0},1}^{1}+\lambda_{k_{0}+1,1}^{1}+\lambda_{2k_{0}+1,2}^{i}\right)\sqrt{L\left[1+\left(u_{2k_{0}+1,2}^{i}\right)^2\gamma_{12}^2\right]}},
\\
&\widetilde{G}_{121}^{i}=\frac{(2k_{0}+1)G_{12}^{i}\left(u_{1,2}^{i}-u_{k_{0},1}^{1}\right)}{\sqrt{L\left(1+\left(u_{1,2}^{i}\right)^2\gamma_{12}^2\right)}\left(\lambda_{1,2}^{i}+\lambda_{k_{0},1}^{1}+\lambda_{k_{0}+1,1}^{1}\right)},
\\
&\widetilde{G}_{211}^{i}=\frac{G_{21}^{i}}{\sqrt{L\left[1+\left(u_{2k_{0}+1,2}^{i}\right)^2 \gamma_{12}^2\right]}\left(\lambda_{2k_{0}+1,2}^{i}+\lambda_{k_{0},1}^{1}+\lambda_{k_{0}+1,1}^{1}\right)}.
        \end{aligned}
    \end{align*}
\end{proof}

\subsection{bifurcation from real simple eigenvalue}\label{diyi0211}
By analyzing the system \eqref{diyige0210}, we obtain
\begin{theorem}\label{diyidingli0211}\rm{[}\textbf{bifurcation from real simple eigenvalue}\rm{]} For the system \eqref{model}-\eqref{bianjian0814}, we have the following conclusions:

\begin{enumerate}[label=(\arabic*)]
    \item  If $a(R_{c})<0$, the system undergoes a supercritical pitchfork bifurcation at $R=R_{c}$ from the trivial branch $\Phi=0$, and bifurcates for $R>R_{c}$
to a local attractor
$\sum$ consisting of two stable steady-state solutions $\Phi_{1}=\left(\mathbf{u}_{1},\theta_{1}\right)$ and  $\Phi_{2}=\left(\mathbf{u}_{2},\theta_{2}\right)$
, as shown in \autoref{diyigedinglitu1}.
The two steady-state solutions are approximately given by
\begin{align}\label{liangge0211}
    \begin{aligned}
        \Phi_{i}=\left(-1\right)^{i}\sqrt{\frac{\lambda_{k_{0},1}^{1}}{-a}}\Phi_{k_{0},1}^{1}+o\left(\left|\lambda_{k_{0},1}^{1}\right|^{\frac{1}{2}}\right),~\text{where~}i=1,2.
    \end{aligned}
\end{align}
    \item If $a(R_{c})>0$, the system undergoes a supercritical pitchfork bifurcation at $R=R_{c}$ from the trivial branch $\Phi=0$, and bifurcates for $R<R_{c}$
to two unstable points $\Phi_{1}=\left(\mathbf{u}_{1},\theta_{1}\right)$ and  $\Phi_{2}=\left(\mathbf{u}_{2},\theta_{2}\right)$
, as shown in \autoref{diyigedinglitu2}.
The two steady-state solutions are approximately given as \eqref{liangge0211}.
\end{enumerate}
\end{theorem}

\begin{figure}[htp]
		\begin{minipage}[t]{0.45\linewidth}
			\centering
			{\begin{tikzpicture}[>=stealth, scale=0.8]
    \draw[black, thick] (-1.5, 0) -- (1.5, 0);
    \draw[black, thick] (0, -1.5) -- (0, 1.5);
    \fill[red] (0, 0) circle (2pt);
    \node at (0, -0.5) {$\Phi=0$};
    \draw[green, thick, ->] (1, 0) -- (0.2, 0);
    \draw[green, thick, ->] (-1, 0) -- (-0.2, 0);
    \draw[green, thick, ->] (0, 1) -- (0, 0.2);
    \draw[green, thick, ->] (0, -1) -- (0, -0.2);
\end{tikzpicture}
}
			
		\end{minipage}
		\hfill
		\begin{minipage}[t]{0.5\linewidth}
			\centering
			{\begin{tikzpicture}[>=stealth,scale=1]
    \tikzset{
        axis/.style={black, thick},
        arrow/.style={green, thick, ->},
        point/.style={circle, fill, minimum size=5pt, inner sep=0pt}
    }

    \draw[axis] (-0.5, 0) -- (1, 0);
    \draw[axis] (1, -1) -- (1, 1);
    \node[point, red] at (1, 0) {};
    \node at (1, -0.6) {$\Phi_1$};
    \draw[arrow] (1.8, 0) -- (1.2, 0);
    \draw[arrow] (0.2, 0) -- (0.8, 0);
    \draw[arrow] (1, 0.8) -- (1, 0.2);
    \draw[arrow] (1, -0.8) -- (1, -0.2);

    \draw[axis] (3, -1) -- (3, 1);
    \node[point, black] at (3, 0) {};
    \node at (3, -0.6) {$\Phi=0$};
    \draw[arrow] (3, 0.8) -- (3, 0.2);
    \draw[arrow] (3, -0.8) -- (3, -0.2);

    \draw[axis] (5, -1) -- (5, 1);
    \node[point, red] at (5, 0) {};
    \node at (5, -0.6) {$\Phi_2$};
    \draw[arrow] (6, 0) -- (5.2, 0);
    \draw[arrow] (4.2, 0) -- (4.8, 0);
    \draw[arrow] (5, 0.8) -- (5, 0.2);
    \draw[arrow] (5, -0.8) -- (5, -0.2);

    \draw[axis] (1, 0) -- (3, 0);
    \draw[axis] (3, 0) -- (6, 0);
\end{tikzpicture}}
		\end{minipage}
        \caption{Topological structure of supercritical pitchfork bifurcation of the system when $a(R_{c})<0$: $R<R_{c}$(left) and $R>R_{c}$(right).}
        \label{diyigedinglitu1}
	\end{figure}
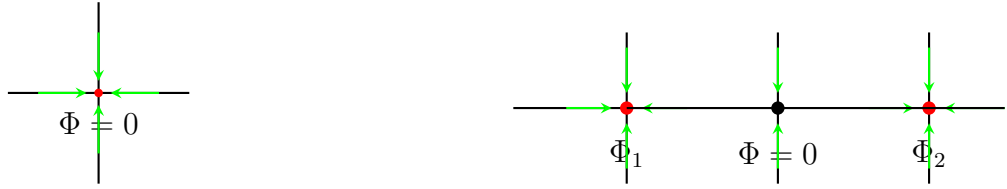

\begin{figure}[htp]
		\begin{minipage}[t]{0.45\linewidth}
			\centering
			{\begin{tikzpicture}[>=stealth]
   \tikzset{
        axis/.style={black, thick},
        arrow/.style={green, thick, ->},
        point/.style={circle, fill, minimum size=4pt, inner sep=0pt}
    }
    \draw[black, thick] (-1.5, 0) -- (1.5, 0);
    \draw[black, thick] (0, -1.5) -- (0, 1.5);
    \fill[black] (0, 0) circle (2pt);
    \node at (0, -0.5) {$\Phi=0$};
    \draw[arrow] (0.2, 0) -- (1, 0);
    \draw[arrow] (-0.2, 0) -- (-1, 0);
     \draw[arrow] (0, 0.2) -- (0, 1);
     \draw[arrow] (0, -0.2) -- (0, -1);
\end{tikzpicture}
}
		\end{minipage}
		\hfill
		\begin{minipage}[t]{0.5\linewidth}
			\centering
			{\begin{tikzpicture}[>=stealth, scale=1]
    \tikzset{
        axis/.style={black, thick},
        arrow/.style={green, thick, ->},
        point/.style={circle, fill, minimum size=4pt, inner sep=0pt}
    }

    \draw[axis] (-0.5, 0) -- (1, 0);
    \draw[axis] (1, -1) -- (1, 1);
    \node[point, black] at (1, 0) {};
    \node at (1, -0.6) {$\Phi_1$};
    \draw[arrow] (2.2, 0) -- (2.8, 0);
    \draw[arrow] (0.8, 0) -- (0.2, 0);
    \draw[arrow] (1, 0.8) -- (1, 0.2);
    \draw[arrow] (1, -0.8) -- (1, -0.2);

    \draw[axis] (3, -1) -- (3, 1);
    \node[point, red] at (3, 0) {};
    \node at (3, -0.6) {$\Phi=0$};
    \draw[arrow] (3, 0.8) -- (3, 0.2);
    \draw[arrow] (3, -0.8) -- (3, -0.2);

    \draw[axis] (5, -1) -- (5, 1);
    \node[point, black] at (5, 0) {};
    \node at (5, -0.6) {$\Phi_2$};
    \draw[arrow] (5.2, 0) -- (5.4, 0);
    \draw[arrow] (3.8, 0) -- (3.2, 0);
    \draw[arrow] (5, 0.8) -- (5, 0.2);
    \draw[arrow] (5, -0.8) -- (5, -0.2);

    \draw[axis] (1, 0) -- (3, 0);
    \draw[axis] (3, 0) -- (6, 0);
\end{tikzpicture}}
		\end{minipage}
        \caption{Topological structure of subcritical pitchfork bifurcation of the system when $a(R_{c})>0$: $R>R_{c}$(left) and $R<R_{c}$(right).}
        \label{diyigedinglitu2}
	\end{figure}
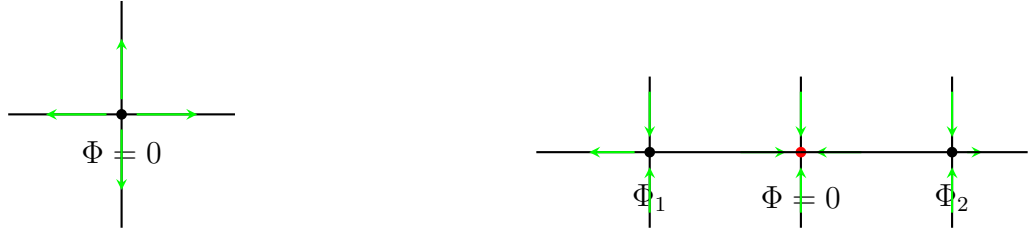
\begin{proof}
   
(1) When $a(R_{c})<0$, continuity implies $a(R)<0$ for all $R$ in a neighborhood of $R_{c}$. Furthermore, for $R<R_{c}$, all eigenvalues of $L$ are negative, so $\Phi=\mathbf{0}$ is linearly stable.
For $R>R_{c}$, $\lambda_{k_{0},1}^{1}>0$, and System \eqref{diyige0210} admits two new steady states. This implies that Systems \eqref{model}–\eqref{bianjian0814} undergoes a bifurcation at $R=R_{c}$, giving rise to two new steady states $\Phi_{1}$ and $\Phi_{2}$ given as \eqref{liangge0211} when R crosses $R_{c}$ from left to right.
Finally, a straightforward stability analysis of these two new steady states for System \eqref{diyige0210} shows that $\Phi_{1}$ and $\Phi_{2}$ are stable.

 (2) A similar argument applies to yield the corresponding conclusions as $a(R_{c})>0$.
\end{proof}
    
\subsection{bifurcation from two real eigenvalues }\label{dier0211}
We now consider the case $m=2$. For the system \eqref{dierge0210}, it has at most eight non-zero equilibrium points as follows:
\begin{align}\label{cibang0211}
    \begin{aligned}
        \left(q_{1},q_{2}\right)=\left(\pm\sqrt{\eta_{1}},0\right),~(0,\pm\sqrt{\eta_{2}}),~\left(\pm\sqrt{\xi_{1}},\pm\sqrt{\xi_{2}}\right),
    \end{aligned}
\end{align}
where 
$$
\eta_{1}=-\frac{\lambda_{k_{0},1}^{1}}{a_{11}},~\eta_{2}=-\frac{\lambda_{k_{0}+1,1}^{1}}{b_{22}},~
\xi_{1}=\frac{\lambda_{k_{0}+1,1}^{1}a_{22}-\lambda_{k_{0},1}^{1}b_{22}}{a_{11}b_{22}-a_{22}b_{11}},~
\xi_{2}=\frac{\lambda_{k_{0}+1,1}^{1}a_{11}-\lambda_{k_{0},1}^{1}b_{11}}{a_{22}b_{11}-a_{11}b_{22}}.
$$
In order to narrate easily, we make the following notations:
\begin{align}\label{wan20211}
\begin{aligned}
&\mathbf{Y_{1}}=-\mathbf{Y}_{2}=\left(\sqrt{\eta_{1}},0\right),~\mathbf{Y}_{3}=-\mathbf{Y}_{4}=\left(0,\sqrt{\eta_{2}}\right),
\\
&\mathbf{Y}_{5}=-\mathbf{Y}_{6}=\left(\sqrt{\xi_{1}},\sqrt{\xi_{2}}\right),~\mathbf{Y}_{7}=-\mathbf{Y}_{8}=\left(-\sqrt{\xi_{1}},\sqrt{\xi_{2}}\right).
\end{aligned}
\end{align}
And 
\begin{align}\label{wan10211}
    \begin{aligned}
        &\Phi_{1}=-\Phi_{2}=\sqrt{\eta_{1}}\Phi_{k_{0},1}^{1},~\Phi_{3}=-\Phi_{4}=\sqrt{\eta_{2}}\Phi_{k_{0}+1,1}^{1},
        \\
        &\Phi_{5}=-\Phi_{6}=\sqrt{\xi_{1}}\Phi_{k_{0},1}^{1}+\sqrt{\xi_{2}}\Phi_{k_{0}+1,1}^{1},
        \\
        &\Phi_{7}=-\Phi_{8}=
        -\sqrt{\xi_{1}}\Phi_{k_{0},1}^{1}+\sqrt{\xi_{2}}\Phi_{k_{0}+1,1}^{1}.
    \end{aligned}
\end{align}

A straightforward stability analysis to each steady state implies the following theorem.
\begin{theorem}\label{dierdingli0211}\rm{[}\textbf{bifurcation from two real  eigenvalues}\rm{]} For the system \eqref{model}-\eqref{bianjian0814}, we have the following conclusions:
\begin{enumerate}[label=(\arabic*)]
    \item If $a_{11}<0$, $b_{22}<0$, $b_{11}>a_{11}$, $a_{22}>b_{22}$, $a_{11}b_{22}-a_{22}b_{11}>0$, and $a_{22}a_{11}+b_{11}b_{22}-2a_{11}b_{22}<0$, it has a  supercritical pitchfork bifurcation at $R=R_{c}$  from the trivial branch $\Phi=0$, and bifurcates
for $R>R_{c}$
to an attractor $\mathcal{A}$ which exactly contains eight non-degenerate equilibrium points $\Phi_{i}$ $(i=1,2,\cdots,8)$ and
is homeomorphic to the one-dimensional sphere $S^{1}$, as shown in \autoref{diyigedinglitu3}. Among them, $\Phi_{i}$ $(i=1,2,3,4)$ are unstable
while $\Phi_{5}-\Phi_{8}$ are stable;
   \item If $b_{11}<a_{11}<0$, $a_{22}<b_{22}<0$, and $a_{11}b_{22}-a_{22}b_{11}<0$, it has a supercritical pitchfork bifurcation at $R=R_{c}$  from the trivial branch $\Phi=0$, and bifurcates
for $R>R_{c}$
to an attractor $\mathcal{A}$ which exactly contains eight non-degenerate equilibrium points $\Phi_{i}$ $(i=1,2,\cdots,8)$ and
is homeomorphic to the one-dimensional sphere $S^{1}$, as shown in \autoref{diyigedinglitu11}. Among them, $\Phi_{i}$ $(i=1,2,3,4)$ are stable
while $\Phi_{5}-\Phi_{8}$ are unstable;

\item If $b_{11}>a_{11}>0$, $a_{22}>b_{22}>0$, $a_{11}b_{22}-a_{22}b_{11}>0$, and $b_{11}b_{22}+a_{11}a_{22}-2 a_{11}b_{22}<0$, it has a  supercritical pitchfork bifurcation at $R=R_{c}$ from the trivial branch $\Phi=0$. Specifically, as $R<R_{c}$, $\Phi=0$ is stable and $\Phi_{i}$ $\left(i=1,2,3,4\right)$ are unstable; as $R>R_{c}$,  $\Phi_{i}$ $\left(i=5,6,7,8\right)$ are stable, see \autoref{diyigedinglitu14}.

\item If $b_{11}<a_{11}<0$, $a_{22}<b_{22}<0$, $a_{11}b_{22}-a_{22}b_{11}>0$, and $b_{11}b_{22}+a_{11}a_{22}-2a_{11}b_{22}<0$, it has a supercritical pitchfork bifurcation at $R=R_{c}$  from the trivial branch $\Phi=0$. Specifically, as $R<R_{c}$, $\Phi=0$ is stable and $\Phi_{i}$ $\left(i=5,6,7,8\right)$ are unstable; as $R>R_{c}$,  $\Phi_{i}$ $\left(i=1,2,3,4\right)$ are stable, see \autoref{diyigedinglitu24}.

\item If $a_{11}<0$, $b_{22}>0$, $a_{22}<b_{22}$, and $b_{11}>a_{11}$, it has a subcritical pitchfork bifurcation at $R=R_{c}$  from the trivial branch $\Phi=0$. Specifically, as $R<R_{c}$, $\Phi=0$ is stable and $\Phi_{i}$ $\left(i=3,4\right)$ are unstable; as $R>R_{c}$,  $\Phi_{i}$ $\left(i=1,2\right)$ are unstable, see \autoref{diyigedinglitu4}.
\item If $a_{11}>0$, $b_{22}<0$, $a_{22}>b_{22}$, and $b_{11}<a_{11}$, it has a subcritical pitchfork bifurcation at $R=R_{c}$  from the trivial branch $\Phi=0$. Specifically, as $R<R_{c}$, $\Phi=0$ is stable and $\Phi_{i}$ $\left(i=1,2\right)$ are unstable; as $R>R_{c}$,  $\Phi_{i}$ $\left(i=3,4\right)$ are unstable, see \autoref{diyigedinglitu5}.
\item If $a_{11}<0$, $b_{22}<0$, $b_{11}>a_{11}$, $a_{22}>b_{22}$, and $a_{11}b_{22}-a_{22}b_{11}<0$, it has a subcritical pitchfork bifurcation at $R=R_{c}$  from the trivial branch $\Phi=0$. Specifically, as $R<R_{c}$, $\Phi=0$ is stable and $\Phi_{i}$ $\left(i=5,6,7,8\right)$ are unstable; as $R>R_{c}$,  $\Phi_{i}$ $\left(i=1,2,3,4\right)$ are unstable, see \autoref{diyigedinglitu15}.
\item If $a_{11}>0$, $b_{22}>0$, $b_{11}<a_{11}$, $a_{22}<b_{22}$, and $a_{11}b_{22}-a_{22}b_{11}<0$, it has a subcritical pitchfork bifurcation at $R=R_{c}$  from the trivial branch $\Phi=0$. Specifically, as $R<R_{c}$, $\Phi=0$ is stable and $\Phi_{i}$ $\left(i=1,2,3,4\right)$ are unstable; as $R>R_{c}$,  $\Phi_{i}$ $\left(i=5,6,7,8\right)$ are unstable, see \autoref{diyigedinglitu16}.
\end{enumerate}
\end{theorem}
\begin{remark}
    It should be remarked that, while the theorem above enumerates eight possible scenarios, only the first conclusion is attained in the concrete examples considered herein, as listed in \autoref{tab:coeff_values}.
\end{remark}
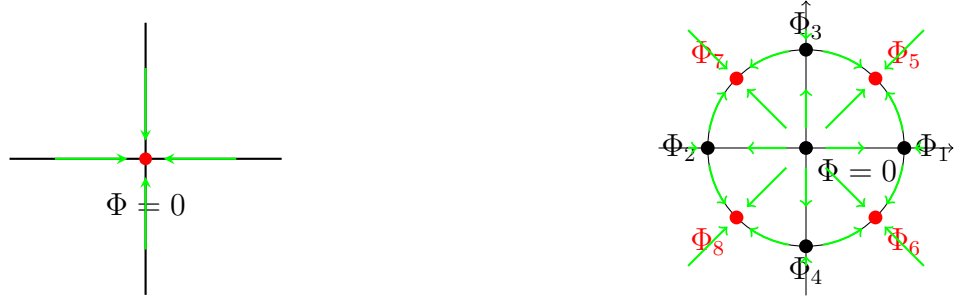
\begin{figure}[htp]
		\begin{minipage}[t]{0.45\linewidth}
			\centering
			{\begin{tikzpicture}[>=stealth, scale=1.2]
    \draw[black, thick] (-1.5, 0) -- (1.5, 0);
    \draw[black, thick] (0, -1.5) -- (0, 1.5);
    \fill[red] (0, 0) circle (2pt);
    \node at (0, -0.5) {$\Phi=0$};
    \draw[green, thick, ->] (1, 0) -- (0.2, 0);
    \draw[green, thick, ->] (-1, 0) -- (-0.2, 0);
    \draw[green, thick, ->] (0, 1) -- (0, 0.2);
    \draw[green, thick, ->] (0, -1) -- (0, -0.2);
\end{tikzpicture}
}
			
		\end{minipage}
		\hfill
		\begin{minipage}[t]{0.5\linewidth}
			\centering
			{\begin{tikzpicture}[scale=1.3]
    \draw[->] (-1.5, 0) -- (1.5, 0) ;
    \draw[->] (0, -1.5) -- (0, 1.5) ;

    \draw (0,0) circle (1);

    \fill (0,0) circle (2pt) node[below right] {$\Phi=0$};

    \fill[black] (1,0) circle (2pt) node[right] {$\Phi_1$};
    \fill[black] (-1,0) circle (2pt) node[left] {$\Phi_2$};
    \fill[black] (0,1) circle (2pt) node[above] {$\Phi_3$};
    \fill[black] (0,-1) circle (2pt) node[below] {$\Phi_4$};

    \fill[red] (-0.707, 0.707) circle (2pt) node[above left] {$\Phi_7$};
    \fill[red] (0.707, 0.707) circle (2pt) node[above right] {$\Phi_5$};
    \fill[red] (0.707, -0.707) circle (2pt) node[below right] {$\Phi_6$};
    \fill[red] (-0.707, -0.707) circle (2pt) node[below left] {$\Phi_8$};
    \draw[green, thick, ->] (0.2, 0) -- (0.6, 0);
    \draw[green, thick, ->] (-0.2, 0) -- (-0.6, 0);
    \draw[green, thick, ->] (0, 0.2) -- (0, 0.6);
    \draw[green, thick, ->] (0, -0.2) -- (0, -0.6);

    \draw[green, thick, ->] (1.2, 0) -- (1.1, 0);
    \draw[green, thick, ->] (-1.2, 0) -- (-1.1, 0);
    \draw[green, thick, ->] (0, 1.2) -- (0, 1.1);
    \draw[green, thick, ->] (0, -1.2) -- (0, -1.1);

\draw[green, thick, ->] (0.2, 0.2) -- (0.6, 0.6);
    \draw[green, thick, ->] (1.2, 1.2) -- (0.8, 0.8);  

    \draw[green, thick, ->] (-0.2, 0.2) -- (-0.6, 0.6);
    \draw[green, thick, ->] (-1.2, 1.2) -- (-0.8, 0.8);

\draw[green, thick, ->] (0.2, -0.2) -- (0.6, -0.6);
    \draw[green, thick, ->] (1.2, -1.2) -- (0.8, -0.8);

\draw[green, thick, ->] (-0.2, -0.2) -- (-0.6, -0.6);
    \draw[green, thick, ->] (-1.2, -1.2) -- (-0.8, -0.8);

\draw[green, thick, ->] (10:1) arc(10:35:1); 

\draw[green, thick, ->] (80:1) arc(80:55:1); 

\draw[green, thick, ->] (100:1) arc(100:125:1); 

\draw[green, thick, ->] (170:1) arc(170:145:1);

\draw[green, thick, ->] (190:1) arc(190:215:1); 

\draw[green, thick, ->] (260:1) arc(260:235:1);

\draw[green, thick, ->] (280:1) arc(280:305:1); 

\draw[green, thick, ->] (350:1) arc(350:325:1);
    
\end{tikzpicture}}
		\end{minipage}
        \caption{Topological structure of supercritical pitchfork bifurcation of the system when $a_{11}<0$, $b_{22}<0$, $b_{11}>a_{11}$, $a_{22}>b_{22}$, $a_{11}b_{22}-a_{22}b_{11}>0$, and $a_{22}a_{11}+b_{11}b_{22}-2a_{11}b_{22}<0$: $R<R_{c}$(left) and $R>R_{c}$(right).}
        \label{diyigedinglitu3}
	\end{figure}

\begin{figure}[htp]
		\begin{minipage}[t]{0.45\linewidth}
			\centering
			{\begin{tikzpicture}[>=stealth, scale=1.2]
    \draw[black, thick] (-1.5, 0) -- (1.5, 0);
    \draw[black, thick] (0, -1.5) -- (0, 1.5);
    \fill[red] (0, 0) circle (2pt);
    \node at (0, -0.5) {$\Phi=0$};
    \draw[green, thick, ->] (1, 0) -- (0.2, 0);
    \draw[green, thick, ->] (-1, 0) -- (-0.2, 0);
    \draw[green, thick, ->] (0, 1) -- (0, 0.2);
    \draw[green, thick, ->] (0, -1) -- (0, -0.2);
\end{tikzpicture}
}
			
		\end{minipage}
		\hfill
		\begin{minipage}[t]{0.5\linewidth}
			\centering
			{\begin{tikzpicture}[scale=1.3]
    \draw[->] (-1.5, 0) -- (1.5, 0) ;
    \draw[->] (0, -1.5) -- (0, 1.5) ;

    \draw (0,0) circle (1);

    \fill (0,0) circle (2pt) node[below right] {$\Phi=0$};

    \fill[red] (1,0) circle (2pt) node[right] {$\Phi_1$};
    \fill[red] (-1,0) circle (2pt) node[left] {$\Phi_2$};
    \fill[red] (0,1) circle (2pt) node[above] {$\Phi_3$};
    \fill[red] (0,-1) circle (2pt) node[below] {$\Phi_4$};

    \fill[black] (-0.707, 0.707) circle (2pt) node[above left] {$\Phi_7$};
    \fill[black] (0.707, 0.707) circle (2pt) node[above right] {$\Phi_5$};
    \fill[black] (0.707, -0.707) circle (2pt) node[below right] {$\Phi_6$};
    \fill[black] (-0.707, -0.707) circle (2pt) node[below left] {$\Phi_8$};
    \draw[green, thick, ->] (0.2, 0) -- (0.6, 0);
    \draw[green, thick, ->] (-0.2, 0) -- (-0.6, 0);
    \draw[green, thick, ->] (0, 0.2) -- (0, 0.6);
    \draw[green, thick, ->] (0, -0.2) -- (0, -0.6);

    \draw[green, thick, ->] (1.2, 0) -- (1.1, 0);
    \draw[green, thick, ->] (-1.2, 0) -- (-1.1, 0);
    \draw[green, thick, ->] (0, 1.2) -- (0, 1.1);
    \draw[green, thick, ->] (0, -1.2) -- (0, -1.1);

\draw[green, thick, ->] (0.2, 0.2) -- (0.6, 0.6);
    \draw[green, thick, ->] (1.2, 1.2) -- (0.8, 0.8);  

    \draw[green, thick, ->] (-0.2, 0.2) -- (-0.6, 0.6);
    \draw[green, thick, ->] (-1.2, 1.2) -- (-0.8, 0.8);

\draw[green, thick, ->] (0.2, -0.2) -- (0.6, -0.6);
    \draw[green, thick, ->] (1.2, -1.2) -- (0.8, -0.8);

\draw[green, thick, ->] (-0.2, -0.2) -- (-0.6, -0.6);
    \draw[green, thick, ->] (-1.2, -1.2) -- (-0.8, -0.8);

\draw[green, thick, ->] (35:1) arc(35:10:1); 

\draw[green, thick, ->] (55:1) arc(55:80:1); 

\draw[green, thick, ->] (125:1) arc(125:100:1); 

\draw[green, thick, ->] (145:1) arc(145:170:1);

\draw[green, thick, ->] (215:1) arc(215:190:1); 

\draw[green, thick, ->] (235:1) arc(235:260:1);

\draw[green, thick, ->] (305:1) arc(305:280:1); 

\draw[green, thick, ->] (325:1) arc(325:350:1);
    
\end{tikzpicture}}
		\end{minipage}
        \caption{Topological structure of supercritical pitchfork bifurcation of the system when $b_{11}<a_{11}<0$, $a_{22}<b_{22}<0$, and $a_{11}b_{22}-a_{22}b_{11}<0$: $R<R_{c}$(left) and $R>R_{c}$(right).}
        \label{diyigedinglitu11}
	\end{figure}

\begin{figure}[htp]
		\begin{minipage}[t]{0.5\linewidth}
			\centering
			{\begin{tikzpicture}[scale=1.2]
    \draw[black, thick] (-1.5, 0) -- (1.5, 0);
    \draw[black, thick] (0, -1.5) -- (0, 1.5);
    \fill[red] (0, 0) circle (2pt);
    \node at (0, -0.5) {$\Phi=0$};

     \fill[black] (1,0) circle (2pt) node[above] {$\Phi_1$};
    \fill[black] (-1,0) circle (2pt) node[below] {$\Phi_2$};
    
    \fill[black] (0,1) circle (2pt) node[above] {$\Phi_3$};
    \fill[black] (0,-1) circle (2pt) node[below] {$\Phi_4$};
     
    \draw[green, thick, ->] (0.8, 0) -- (0.2, 0);
    \draw[green, thick, ->] (-0.8, 0) -- (-0.2, 0);
    \draw[green, thick, ->] (0, 0.8) -- (0, 0.2);
    \draw[green, thick, ->] (0, -0.8) -- (0, -0.2);

\end{tikzpicture}}
		\end{minipage}
        \hfill
		\begin{minipage}[t]{0.45\linewidth}
			\centering
			{\begin{tikzpicture}[>=stealth, scale=1.2]
    \draw[black, thick] (-1.5, 0) -- (1.5, 0);
    \draw[black, thick] (0, -1.5) -- (0, 1.5);
   
    \fill[black] (0, 0) circle (2pt);
    \node at (0, -0.5) {$\Phi=0$};

    \fill[red] (0.707, 0.707) circle (2pt) node[above right] {$\Phi_5$};
    \fill[red] (0.707, -0.707) circle (2pt) node[below right] {$\Phi_6$};
     \fill[red] (-0.707, 0.707) circle (2pt) node[above left] {$\Phi_7$};
    \fill[red] (-0.707, -0.707) circle (2pt) node[below left] {$\Phi_8$};
    \draw[green, thick, ->] (0.2, 0) -- (0.8, 0);
    \draw[green, thick, ->] (-0.2, 0) -- (-0.8, 0);
    \draw[green, thick, ->] (0, 0.2) -- (0, 0.8);
    \draw[green, thick, ->] (0, -0.2) -- (0, -0.8);
\end{tikzpicture}
}
			
		\end{minipage}
		
        \caption{Topological structure of supercritical pitchfork bifurcation of the system when $b_{11}>a_{11}>0$, $a_{22}>b_{22}>0$, $a_{11}b_{22}-a_{22}b_{11}>0$, and $b_{11}b_{22}+a_{11}a_{22}-2a_{11}b_{22}<0$: $R<R_{c}$(left) and $R>R_{c}$(right).}
        \label{diyigedinglitu14}
\end{figure}
\begin{figure}[htp]
		\begin{minipage}[t]{0.5\linewidth}
			\centering
			{\begin{tikzpicture}[scale=1.2]
    \draw[black, thick] (-1.5, 0) -- (1.5, 0);
    \draw[black, thick] (0, -1.5) -- (0, 1.5);
    \fill[red] (0, 0) circle (2pt);
    \node at (0, -0.5) {$\Phi=0$};

 \fill[black] (-0.707, 0.707) circle (2pt) node[above left] {$\Phi_7$};
    \fill[black] (0.707, 0.707) circle (2pt) node[above right] {$\Phi_5$};
    \fill[black] (0.707, -0.707) circle (2pt) node[below right] {$\Phi_6$};
    \fill[black] (-0.707, -0.707) circle (2pt) node[below left] {$\Phi_8$};
     
    \draw[green, thick, ->] (0.8, 0) -- (0.2, 0);
    \draw[green, thick, ->] (-0.8, 0) -- (-0.2, 0);
    \draw[green, thick, ->] (0, 0.8) -- (0, 0.2);
    \draw[green, thick, ->] (0, -0.8) -- (0, -0.2);

\end{tikzpicture}}
		\end{minipage}
        \hfill
		\begin{minipage}[t]{0.45\linewidth}
			\centering
			{\begin{tikzpicture}[>=stealth, scale=1.2]
    \draw[black, thick] (-1.5, 0) -- (1.5, 0);
    \draw[black, thick] (0, -1.5) -- (0, 1.5);
    \fill[red] (1,0) circle (2pt) node[above] {$\Phi_1$};
    \fill[red] (-1,0) circle (2pt) node[below] {$\Phi_2$};
    \fill[black] (0, 0) circle (2pt);
    \node at (0, -0.5) {$\Phi=0$};

     \fill[red] (0,1) circle (2pt) node[above] {$\Phi_3$};
    \fill[red] (0,-1) circle (2pt) node[below] {$\Phi_4$};
    \draw[green, thick, ->] (0.2, 0) -- (0.8, 0);
    \draw[green, thick, ->] (-0.2, 0) -- (-0.8, 0);
    \draw[green, thick, ->] (0, 0.2) -- (0, 0.8);
    \draw[green, thick, ->] (0, -0.2) -- (0, -0.8);
      \draw[green, thick, ->] (0, 1.4) -- (0, 1.2);
    \draw[green, thick, ->] (0, -1.4) -- (0, -1.2);
  \draw[green, thick, ->] (1.4, 0) -- (1.2, 0);
    \draw[green, thick, ->] (-1.4, 0) -- (-1.2, 0);
\end{tikzpicture}
}
			
		\end{minipage}
		
        \caption{Topological structure of supercritical pitchfork bifurcation of the system when $b_{11}<a_{11}<0$, $a_{22}<b_{22}<0$, $a_{11}b_{22}-a_{22}b_{11}>0$, and $b_{11}b_{22}+a_{11}a_{22}-2a_{11}b_{22}<0$: $R<R_{c}$(left) and $R>R_{c}$(right).}
        \label{diyigedinglitu24}
\end{figure}

    \begin{figure}[htp]
		\begin{minipage}[t]{0.45\linewidth}
			\centering
			{\begin{tikzpicture}[>=stealth, scale=1.2]
    \draw[black, thick] (-1.5, 0) -- (1.5, 0);
    \draw[black, thick] (0, -1.5) -- (0, 1.5);
    \fill[red] (0, 0) circle (2pt);
    \node at (0, -0.5) {$\Phi=0$};

     \fill[black] (0,1) circle (2pt) node[above] {$\Phi_3$};
    \fill[black] (0,-1) circle (2pt) node[below] {$\Phi_4$};
    \draw[green, thick, ->] (0.8, 0) -- (0.2, 0);
    \draw[green, thick, ->] (-0.8, 0) -- (-0.2, 0);
    \draw[green, thick, ->] (0, 0.8) -- (0, 0.2);
    \draw[green, thick, ->] (0, -0.8) -- (0, -0.2);
      \draw[green, thick, ->] (0, 1.2) -- (0, 1.4);
    \draw[green, thick, ->] (0, -1.2) -- (0, -1.4);

    \draw[green, thick, ->] (-0.8, 1) -- (-0.2, 1);
    \draw[green, thick, ->] (0.8, 1) -- (0.2, 1);
    \draw[green, thick, ->] (-0.8, -1) -- (-0.2, -1);
    \draw[green, thick, ->] (0.8, -1) -- (0.2, -1);
\end{tikzpicture}
}
			
		\end{minipage}
		\hfill
		\begin{minipage}[t]{0.5\linewidth}
			\centering
			{\begin{tikzpicture}[scale=1.2]
    \draw[black, thick] (-1.5, 0) -- (1.5, 0);
    \draw[black, thick] (0, -1.5) -- (0, 1.5);
    \fill[black] (0, 0) circle (2pt);
    \node at (0, -0.5) {$\Phi=0$};

     \fill[black] (1,0) circle (2pt) node[above] {$\Phi_1$};
    \fill[black] (-1,0) circle (2pt) node[below] {$\Phi_2$};
    \draw[green, thick, ->] (0.2, 0) -- (0.8, 0);
    \draw[green, thick, ->] (-0.2, 0) -- (-0.8, 0);
    \draw[green, thick, ->] (0, 0.2) -- (0, 0.8);
    \draw[green, thick, ->] (0, -0.2) -- (0, -0.8);
      \draw[green, thick, ->] (1.4, 0) -- (1.2, 0);
    \draw[green, thick, ->] (-1.4, 0) -- (-1.2, 0);

    \draw[green, thick, ->] (1, 0.2) -- (1, 0.8);
    \draw[green, thick, ->] (1, -0.2) -- (1, -0.8);
    \draw[green, thick, ->] (-1, 0.2) -- (-1, 0.8);
    \draw[green, thick, ->] (-1, -0.2) -- (-1, -0.8);
    
\end{tikzpicture}}
		\end{minipage}
        \caption{Topological structure of subcritical pitchfork bifurcation of the system when $a_{11}<0$, $b_{22}>0$, $a_{22}<b_{22}$, and $b_{11}>a_{11}$: $R<R_{c}$(left) and $R>R_{c}$(right).}
        \label{diyigedinglitu4}
\end{figure}

 \begin{figure}[htp]
		\begin{minipage}[t]{0.5\linewidth}
			\centering
			{\begin{tikzpicture}[scale=1.2]
    \draw[black, thick] (-1.5, 0) -- (1.5, 0);
    \draw[black, thick] (0, -1.5) -- (0, 1.5);
    \fill[red] (0, 0) circle (2pt);
    \node at (0, -0.5) {$\Phi=0$};

     \fill[black] (1,0) circle (2pt) node[above] {$\Phi_1$};
    \fill[black] (-1,0) circle (2pt) node[below] {$\Phi_2$};
    \draw[green, thick, ->] (0.8, 0) -- (0.2, 0);
    \draw[green, thick, ->] (-0.8, 0) -- (-0.2, 0);
    \draw[green, thick, ->] (0, 0.8) -- (0, 0.2);
    \draw[green, thick, ->] (0, -0.8) -- (0, -0.2);
      \draw[green, thick, ->] (1.2, 0) -- (1.4, 0);
    \draw[green, thick, ->] (-1.2, 0) -- (-1.4, 0);

    \draw[green, thick, ->] (1, 0.8) -- (1, 0.2);
    \draw[green, thick, ->] (1, -0.8) -- (1, -0.2);
    \draw[green, thick, ->] (-1, 0.8) -- (-1, 0.2);
    \draw[green, thick, ->] (-1, -0.8) -- (-1, -0.2);
    
\end{tikzpicture}}
		\end{minipage}
        \hfill
		\begin{minipage}[t]{0.45\linewidth}
			\centering
			{\begin{tikzpicture}[>=stealth, scale=1.2]
    \draw[black, thick] (-1.5, 0) -- (1.5, 0);
    \draw[black, thick] (0, -1.5) -- (0, 1.5);
    \fill[black] (0, 0) circle (2pt);
    \node at (0, -0.5) {$\Phi=0$};

     \fill[black] (0,1) circle (2pt) node[above] {$\Phi_3$};
    \fill[black] (0,-1) circle (2pt) node[below] {$\Phi_4$};
    \draw[green, thick, ->] (0.2, 0) -- (0.8, 0);
    \draw[green, thick, ->] (-0.2, 0) -- (-0.8, 0);
    \draw[green, thick, ->] (0, 0.2) -- (0, 0.8);
    \draw[green, thick, ->] (0, -0.2) -- (0, -0.8);
      \draw[green, thick, ->] (0, 1.4) -- (0, 1.2);
    \draw[green, thick, ->] (0, -1.4) -- (0, -1.2);

    \draw[green, thick, ->] (-0.2, 1) -- (-0.8, 1);
    \draw[green, thick, ->] (0.2, 1) -- (0.8, 1);
    \draw[green, thick, ->] (-0.2, -1) -- (-0.8, -1);
    \draw[green, thick, ->] (0.2, -1) -- (0.8, -1);
\end{tikzpicture}
}
			
		\end{minipage}
		
        \caption{Topological structure of subcritical pitchfork bifurcation of the system when $a_{11}>0$, $b_{22}<0$, $b_{11}
        <a_{11}$, and $a_{22}>b_{22}$: $R<R_{c}$(left) and $R>R_{c}$(right).}
        \label{diyigedinglitu5}
\end{figure}

\begin{figure}[htp]
		\begin{minipage}[t]{0.5\linewidth}
			\centering
			{\begin{tikzpicture}[scale=1.2]
    \draw[black, thick] (-1.5, 0) -- (1.5, 0);
    \draw[black, thick] (0, -1.5) -- (0, 1.5);
    \fill[red] (0, 0) circle (2pt);
    \node at (0, -0.5) {$\Phi=0$};

 \fill[black] (-0.707, 0.707) circle (2pt) node[above left] {$\Phi_7$};
    \fill[black] (0.707, 0.707) circle (2pt) node[above right] {$\Phi_5$};
    \fill[black] (0.707, -0.707) circle (2pt) node[below right] {$\Phi_6$};
    \fill[black] (-0.707, -0.707) circle (2pt) node[below left] {$\Phi_8$};
     
    \draw[green, thick, ->] (0.8, 0) -- (0.2, 0);
    \draw[green, thick, ->] (-0.8, 0) -- (-0.2, 0);
    \draw[green, thick, ->] (0, 0.8) -- (0, 0.2);
    \draw[green, thick, ->] (0, -0.8) -- (0, -0.2);

\end{tikzpicture}}
		\end{minipage}
        \hfill
		\begin{minipage}[t]{0.45\linewidth}
			\centering
			{\begin{tikzpicture}[>=stealth, scale=1.2]
    \draw[black, thick] (-1.5, 0) -- (1.5, 0);
    \draw[black, thick] (0, -1.5) -- (0, 1.5);
    \fill[black] (1,0) circle (2pt) node[above] {$\Phi_1$};
    \fill[black] (-1,0) circle (2pt) node[below] {$\Phi_2$};
    \fill[black] (0, 0) circle (2pt);
    \node at (0, -0.5) {$\Phi=0$};

     \fill[black] (0,1) circle (2pt) node[above] {$\Phi_3$};
    \fill[black] (0,-1) circle (2pt) node[below] {$\Phi_4$};
    \draw[green, thick, ->] (0.2, 0) -- (0.8, 0);
    \draw[green, thick, ->] (-0.2, 0) -- (-0.8, 0);
    \draw[green, thick, ->] (0, 0.2) -- (0, 0.8);
    \draw[green, thick, ->] (0, -0.2) -- (0, -0.8);
      \draw[green, thick, ->] (0, 1.4) -- (0, 1.2);
    \draw[green, thick, ->] (0, -1.4) -- (0, -1.2);
  \draw[green, thick, ->] (1.4, 0) -- (1.2, 0);
    \draw[green, thick, ->] (-1.4, 0) -- (-1.2, 0);
\end{tikzpicture}
}
			
		\end{minipage}
		
        \caption{Topological structure of subcritical pitchfork bifurcation of the system when $a_{11}<0$, $b_{22}<0$, $a_{22}>b_{22}$, $b_{11}>a_{11}$, and $a_{11}b_{22}-a_{22}b_{11}<0$: $R<R_{c}$(left) and $R>R_{c}$(right).}
        \label{diyigedinglitu15}
\end{figure}

\begin{figure}[htp]
		\begin{minipage}[t]{0.5\linewidth}
			\centering
			{\begin{tikzpicture}[scale=1.2]
    \draw[black, thick] (-1.5, 0) -- (1.5, 0);
    \draw[black, thick] (0, -1.5) -- (0, 1.5);
    \fill[red] (0, 0) circle (2pt);
    \node at (0, -0.5) {$\Phi=0$};
       \fill[black] (1,0) circle (2pt) node[above] {$\Phi_1$};
    \fill[black] (-1,0) circle (2pt) node[below] {$\Phi_2$};
    \fill[black] (0,1) circle (2pt) node[above] {$\Phi_3$};
    \fill[black] (0,-1) circle (2pt) node[below] {$\Phi_4$};
    \draw[green, thick, ->] (0.8, 0) -- (0.2, 0);
    \draw[green, thick, ->] (-0.8, 0) -- (-0.2, 0);
    \draw[green, thick, ->] (0, 0.8) -- (0, 0.2);
    \draw[green, thick, ->] (0, -0.8) -- (0, -0.2);

\end{tikzpicture}}
		\end{minipage}
        \hfill
		\begin{minipage}[t]{0.45\linewidth}
			\centering
			{\begin{tikzpicture}[>=stealth, scale=1.2]
    \draw[black, thick] (-1.5, 0) -- (1.5, 0);
    \draw[black, thick] (0, -1.5) -- (0, 1.5);
  
    \fill[black] (0, 0) circle (2pt);
    \node at (0, -0.5) {$\Phi=0$};
\fill[black] (-0.707, 0.707) circle (2pt) node[above left] {$\Phi_7$};
    \fill[black] (0.707, 0.707) circle (2pt) node[above right] {$\Phi_5$};
    \fill[black] (0.707, -0.707) circle (2pt) node[below right] {$\Phi_6$};
    \fill[black] (-0.707, -0.707) circle (2pt) node[below left] {$\Phi_8$};
     
    \draw[green, thick, ->] (0.2, 0) -- (0.8, 0);
    \draw[green, thick, ->] (-0.2, 0) -- (-0.8, 0);
    \draw[green, thick, ->] (0, 0.2) -- (0, 0.8);
    \draw[green, thick, ->] (0, -0.2) -- (0, -0.8);
\end{tikzpicture}
}
			
		\end{minipage}
		
        \caption{Topological structure of subcritical pitchfork bifurcation of the system when $a_{11}>0$, $b_{22}>0$, $a_{22}<b_{22}$, $b_{11}<a_{11}$, and $a_{11}b_{22}-a_{22}b_{11}<0$: $R<R_{c}$(left) and $R>R_{c}$(right).}
        \label{diyigedinglitu16}
\end{figure}

\begin{proof}
    Since $\lambda_{k_{0},1}^{1}=\lambda_{k_{0}+1,1}^{1}=0$ as $R=R_{c}$, we let $\lambda_{k_{0},1}^{1}=\lambda_{k_{0}+1}^{1}=\lambda$ when $R$ is in the vicinity of $R_{c}$. In addition, let 
    \begin{align*}
        \begin{aligned}
            \mathbf{f}\left(q_{1},q_{2}\right)=\left(\lambda q_{1}+q_{1}\left(a_{11}q_{1}^{2}+a_{22}q_{2}^2\right),\lambda q_{2}+q_{2}\left(b_{11}q_{1}^{2}+b_{22}q_{2}^2\right)\right).
        \end{aligned}
    \end{align*}
    Thus, 
    \begin{align}\label{xuanzhe0212}
        \begin{aligned}
            D\mathbf{f}\left(q_{1},q_{2}\right)=
            \begin{pmatrix}
                \lambda +3 q_{11}q_{1}^2+a_{22} q_{2}^2&2a_{22}q_{1}q_{2}
                \\
                2b_{11}q_{1}q_{2}&\lambda+b_{11}q_{1}^2+3b_{22}q_{2}^2
            \end{pmatrix}.
        \end{aligned}
    \end{align}
    Furthermore, according to \eqref{cibang0211}, we deduce the conditions for the existence of the steady state $\mathbf{Y}_{1}-\mathbf{Y}_{8}$ in two scenarios $R>R_{c}$ and $R<R_{c}$ as follows.
\begin{center}
\begin{tabular}{|c|c|c|c|}
\hline
 &$\mathbf{Y}_1,\mathbf{Y}_2$  & $\mathbf{Y}_3,\mathbf{Y}_4$ & $\mathbf{Y}_5-\mathbf{Y}_8$ \\
\hline
 $R>R_{c}$& $a_{11}<0$ & $b_{22}<0$  & 
$\begin{aligned}

(a_{22}-b_{22})(a_{11}b_{22}-a_{22}b_{11}) &> 0 \\
(a_{11}-b_{11})(a_{11}b_{22}-a_{22}b_{11}) &< 0
\end{aligned}$ \\
\hline
$R<R_{c}$& $a_{11}>0$ & $b_{22}>0$  & 
$\begin{aligned}

(a_{22}-b_{22})(a_{11}b_{22}-a_{22}b_{11}) &< 0 \\
(a_{11}-b_{11})(a_{11}b_{22}-a_{22}b_{11}) &> 0
\end{aligned}$
\\
\hline
\end{tabular}
\end{center}
By a straightforward eigenvalue computation of $D\mathbf{f}\left(\mathbf{Y}_{i}\right)$($i=1,2,\cdots,8$) as $R>R_{c}$, we also obtain the conditions which implies the linear stability of $\mathbf{Y}_{i}$ as follows. 
\begin{center}
\begin{tabular}{|c|c|c|}
\hline
 $\mathbf{Y}_1,\mathbf{Y}_2$  & $\mathbf{Y}_3,\mathbf{Y}_4$ & $\mathbf{Y}_5-\mathbf{Y}_8$ \\
\hline
 
 $\begin{aligned}
     &a_{11}\left(a_{11}+b_{11}\right)>0
     \\
     &a_{11}\left(b_{11}-a_{11}\right)>0
 \end{aligned}$ & $\begin{aligned}
     &b_{22}\left(a_{22}+b_{22}\right)>0
     \\
     &b_{22}\left(a_{22}-b_{22}\right)>0
 \end{aligned}$  & 
$\begin{aligned}

(a_{22}a_{11}+b_{11}b_{22}-2a_{11}b_{22})(a_{11}b_{22}-a_{22}b_{11}) &< 0 \\
(a_{11}-b_{11})\left(b_{22}-a_{22}\right)(a_{11}b_{22}-a_{22}b_{11}) &>0
\end{aligned}$ \\
\hline
\end{tabular}
\end{center}
From the above two tables and Topology degree theory\cite{Chang2005}, all the conclusions are obtained. 


\end{proof}

\section{Numerical stimulation}\label{shuzhi0205}
In this section we present numerical investigations of the model
\eqref{model}-\eqref{bianjian0814}, organized into four parts.
First we describe the parameter regime and the bifurcated steady
solutions predicted by the center manifold reduction in
Subsection~\ref{diyi0211} (see \autoref{yueqianleixing0702}).
Second we introduce the stream-function--temperature formulation of
the governing equations and develop a finite-difference
scheme for their numerical integration. Third we report simulation
results that validate the theoretical predictions. Finally, we investigate the effect of $V_{a}$ on convection.


\subsection{Parameter regime and bifurcation analysis}

\autoref{yueqianleixing0702} displays the parameter regime diagram in the $(L,V_a)$-plane
for which all observed bifurcations are supercritical when the first
critical eigenvalue is simple. We omit the scenario in which the number of the first
critical eigenvalue is two since the parameter set yielding two primary characteristic roots is a zero measure set, see remark \ref{remark0723}. 
\begin{figure}[H]
    \centering
    {\includegraphics[width=2in]{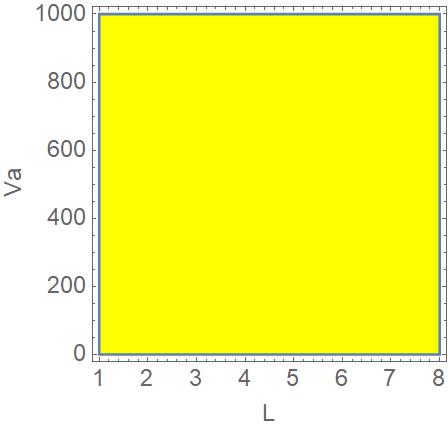}}
\caption{$L\in [1,8]$, $V_{a}\in [0.1,1000]$}.
\label{yueqianleixing0702}
\end{figure}
 We now present two examples: one in which the first eigenvalue is simple, and one in which it is double. For the former, we conduct numerical simulations via a finite difference scheme in the subsequent section.

\begin{example}[Simple eigenvalue]
Take $L=3.2$ and $V_a=4$.  From formula \eqref{linjie0926} we obtain
$R_c = 6.296275276569543$, with critical index $(k_0,\ell_0)=(3,1)$.
We choose $R = 6.297 > R_c$.  Evaluating \eqref{tezhenzhi0814}-\eqref{tehzenxiangliang0814} and \eqref{diyige0210}
yields
\[
\lambda_{k_0,\!1}^1 = 2.58397\times 10^{-5},\quad
\lambda_{k_0,\!1}^2 = -22.5441,
\]
\[
u_{k_0,\!1}^{1} = -0.99997,\quad u_{k_0,\!1}^{2} = 0.215703,
\quad a = -952.3990839194586 < 0.
\]
Since $a < 0$, Theorem \ref{diyidingli0211} guarantees a supercritical pitchfork
bifurcation, and the stable approximated nontrivial solutions are
\begin{equation}\label{nonzero0703}
\Phi_i \approx (-1)^i\,
\begin{pmatrix}
-0.000130868\sin\dfrac{3\pi x_1}{3.2}\cos(\pi x_2) \\[6pt]
0.000122688\cos\dfrac{3\pi x_1}{3.2}\sin(\pi x_2) \\[6pt]
0.0000416565\cos\dfrac{3\pi x_1}{3.2}\sin(\pi x_2)
\end{pmatrix},
\qquad i = 1,2.
\end{equation}
\end{example}
\begin{example}[Two eigenvalues]
In view of Remark \ref{remark0723}, the first eigenvalue of the linear operator 
$\mathbf{L}$ is double precisely when $L=\sqrt{n\left(n+1\right)}$ for some $n\in\mathbf{N}^{*}$.
 A number of representative cases are gathered in Table  below; in each instance, the coefficients satisfy the first conclusion of Theorem \ref{dierdingli0211}. This provides strong evidence that, in the double-eigenvalue regime, the bifurcation is necessarily supercritical pitchfork bifurcation as shown in \autoref{diyigedinglitu3}.

\begin{table}[htbp]
\centering
\begin{tabular}{|c|c|c|c|c|c|}
\hline
\(n\) & \(V_a\) & \(a_{11}\) & \(a_{22}\) & \(b_{11}\) & \(b_{22}\) \\
\hline
1 & 4 & \(-2.34826 \times 10^{-2}\) & \(1.21249 \times 10^{-2}\) & \(-6.56936 \times 10^{-3}\) & \(-2.71360 \times 10^{-2}\) \\
2 & 4 & \(-1.50558 \times 10^{-2}\) & \(3.83726 \times 10^{-3}\) & \(-2.68433 \times 10^{-3}\) & \(-1.64196 \times 10^{-2}\) \\
3 & 4 & \(-1.09955 \times 10^{-2}\) & \(1.85800 \times 10^{-3}\) & \(-1.44224 \times 10^{-3}\) & \(-1.16979 \times 10^{-2}\) \\
4 & 4 & \(-8.64540 \times 10^{-3}\) & \(1.09232 \times 10^{-3}\) & \(-8.97529 \times 10^{-4}\) & \(-9.07198 \times 10^{-3}\) \\
5 & 4 & \(-7.11877 \times 10^{-3}\) & \(7.18145 \times 10^{-4}\) & \(-6.11688 \times 10^{-4}\) & \(-7.40489 \times 10^{-3}\) \\
6 & 4 & \(-6.04879 \times 10^{-3}\) & \(5.07819 \times 10^{-4}\) & \(-4.43404 \times 10^{-4}\) & \(-6.25387 \times 10^{-3}\) \\
1 & 10 & \(-4.45059 \times 10^{-2}\) & \(2.29800 \times 10^{-2}\) & \(-1.39356 \times 10^{-2}\) & \(-5.75634 \times 10^{-2}\) \\
2 & 10 & \(-2.91011 \times 10^{-2}\) & \(7.41695 \times 10^{-3}\) & \(-5.54958 \times 10^{-3}\) & \(-3.39459 \times 10^{-2}\) \\
3 & 10 & \(-2.14410 \times 10^{-2}\) & \(3.62307 \times 10^{-3}\) & \(-2.95061 \times 10^{-3}\) & \(-2.39323 \times 10^{-2}\) \\
4 & 10 & \(-1.69429 \times 10^{-2}\) & \(2.14068 \times 10^{-3}\) & \(-1.82581 \times 10^{-3}\) & \(-1.84547 \times 10^{-2}\) \\
5 & 10 & \(-1.39961 \times 10^{-2}\) & \(1.41193 \times 10^{-3}\) & \(-1.23990 \times 10^{-3}\) & \(-1.50098 \times 10^{-2}\) \\
6 & 10 & \(-1.19192 \times 10^{-2}\) & \(1.00066 \times 10^{-3}\) & \(-8.96585 \times 10^{-4}\) & \(-1.26456 \times 10^{-2}\) \\
\hline
\end{tabular}
\caption{Numerical values of the coefficients for various \(n\) and \(V_a\).}
\label{tab:coeff_values}
\end{table}
\end{example}

\subsection{Stream-function formulation and discretization}

To solve \eqref{model}-\eqref{bianjian0814}  numerically we introduce the
stream function $\Psi$ satisfying
\[
u_1=u = \frac{\partial\Psi}{\partial x_2},\qquad
u_2 =v= -\frac{\partial\Psi}{\partial x_1}.
\]
 Then, the pressure term is eliminated from the momentum
equation, and the system becomes
\begin{align}\label{bianleyixia0703}
\begin{aligned}
&\frac{\partial\Delta\Psi}{\partial t}
= -V_a\Delta\Psi - V_a R\,\frac{\partial\theta}{\partial x_1},
\\
&\frac{\partial\theta}{\partial t}
+ \frac{\partial\Psi}{\partial x_2}\frac{\partial\theta}{\partial x_1}
- \frac{\partial\Psi}{\partial x_1}\frac{\partial\theta}{\partial x_2}
= -R\,\frac{\partial\Psi}{\partial x_1} + \Delta\theta,
\end{aligned}
\end{align}
with boundary conditions
\[
\Psi\big|_{\partial\Omega}=0,\qquad
\theta\big|_{x_2=0,1}=0,\qquad
\frac{\partial\theta}{\partial x_1}\Big|_{x_1=0,L}=0,
\]
and initial conditions chosen to match the theoretical prediction
\eqref{nonzero0703}:
\begin{equation}\label{eq:init}
\Psi_0 = -0.0000416564\sin\frac{3\pi x_1}{3.2}\sin(\pi x_2),\quad
\theta_0 = 0.0000416565\cos\frac{3\pi x_1}{3.2}\sin(\pi x_2).
\end{equation}

\subsubsection{Spatial discretization}

We cover the domain by a uniform grid
\[
x_{1,i}=i\Delta x,\;\;i=0,\ldots,N_x,\qquad
x_{2,j}=j\Delta y,\;\;j=0,\ldots,N_y,
\]
where $\Delta x=L/N_x$, $\Delta y=1/N_y$.  All spatial derivatives
are approximated by second-order central finite differences.
Let $\Psi_{j,i}^n$ and $\theta_{j,i}^n$ denote the discrete
approximations at time level $t_n=n\Delta t$.


\subsubsection{Time-stepping algorithm}

At each time step the algorithm proceeds through five stages:

\medskip\noindent
\textbf{Step 1. Compute the discrete Laplacian of $\psi$.}
\[
(\Delta_h\psi)_{j,i}^n =
\frac{\psi_{j,i+1}^n-2\psi_{j,i}^n+\psi_{j,i-1}^n}{\Delta x^2}
+ \frac{\psi_{j+1,i}^n-2\psi_{j,i}^n+\psi_{j-1,i}^n}{\Delta y^2},
\qquad 2\le j\le N_y\!-\!1,\; 2\le i\le N_x\!-\!1.
\]

\medskip\noindent
\textbf{Step 2. Compute $\partial\theta/\partial x_1$ by central differences.}
\[
(\partial_{x_1}\theta)_{j,i}^n =
\frac{\theta_{j,i+1}^n-\theta_{j,i-1}^n}{2\Delta x}.
\]

\medskip\noindent
\textbf{Step 3. Update the discrete Laplacian via the stream function equation.}
From $\eqref{bianleyixia0703}_{1}$ we advance $\Delta_h\psi$ with an explicit Euler step:
\begin{equation}
(\Delta_h\psi)_{j,i}^{n+1} =
(\Delta_h\psi)_{j,i}^n
+ \Delta t\Bigl[ -V_a(\Delta_h\psi)_{j,i}^n
- V_a R\,(\partial_{x_1}\theta)_{j,i}^n \Bigr].
\label{eq:update_lap}
\end{equation}

\medskip\noindent
\textbf{Step 4. Recover $\psi^{n+1}$ from its Laplacian by solving a Poisson equation.}
Given the updated $(\Delta_h\psi)^{n+1}$, we solve the discrete
Poisson problem
\[
\Delta_h \psi^{n+1} = (\Delta_h\psi)^{n+1},
\qquad \psi^{n+1}\big|_{\partial\Omega} = 0,
\]
using the Jacobi iteration.  Denoting the $k$-th Jacobi iterate by
$\psi_{j,i}^{(k)}$ and using $\psi^n$ as the initial guess, the
iteration reads
\begin{equation}
\psi_{j,i}^{(k+1)} =
\frac{
\displaystyle\frac{\psi_{j,i+1}^{(k)}+\psi_{j,i-1}^{(k)}}{\Delta x^2}
+ \frac{\psi_{j+1,i}^{(k)}+\psi_{j-1,i}^{(k)}}{\Delta y^2}
- (\Delta_h\psi)_{j,i}^{n+1}
}{
\displaystyle\frac{2}{\Delta x^2} + \frac{2}{\Delta y^2}
}.
\label{eq:jacobi}
\end{equation}
The iteration terminates when
$\max_{j,i}\bigl|\psi_{j,i}^{(k+1)}-\psi_{j,i}^{(k)}\bigr| < 10^{-12}$
or a prescribed maximum number of iterations (3000) is reached.

\begin{remark}
The Poisson matrix arising from the five-point stencil with
Dirichlet boundary conditions is symmetric positive definite and
strictly diagonally dominant; therefore the Jacobi iteration is
guaranteed to converge.
\end{remark}

\medskip\noindent
\textbf{Step 5. Update the temperature field.}
First compute the discrete Laplacian of $\theta$,
\[
(\Delta_h\theta)_{j,i}^n =
\frac{\theta_{j,i+1}^n-2\theta_{j,i}^n+\theta_{j,i-1}^n}{\Delta x^2}
+ \frac{\theta_{j+1,i}^n-2\theta_{j,i}^n+\theta_{j-1,i}^n}{\Delta y^2},
\]
and the velocity components via central differences,
\[
u_{j,i}^n = \frac{\psi_{j+1,i}^n-\psi_{j-1,i}^n}{2\Delta y},\qquad
v_{j,i}^n = -\frac{\psi_{j,i+1}^n-\psi_{j,i-1}^n}{2\Delta x}.
\]
Then the temperature is advanced with an explicit Euler step
applied to $\eqref{bianleyixia0703}_{2}$:
\begin{equation}
\theta_{j,i}^{n+1} = \theta_{j,i}^n
+ \Delta t\Bigl[
- \underbrace{\bigl(u_{j,i}^{n+1}(\partial_{x_1}\theta)_{j,i}^n
+ v_{j,i}^{n+1}(\partial_{x_2}\theta)_{j,i}^n\bigr)}_{\text{advection}}
- \underbrace{R\,u_{j,i}^{n+1}}_{\text{buoyancy}}
+ \underbrace{(\Delta_h\theta)_{j,i}^n}_{\text{diffusion}}
\Bigr].
\label{eq:theta_update}
\end{equation}
Here $(\partial_{x_2}\theta)_{j,i}^n$ is again computed by central
differences.  The Neumann boundary conditions for $\theta$ on the
lateral walls are enforced by setting
$\theta_{j,1}^{n+1}=\theta_{j,2}^{n+1}$ and
$\theta_{j,N_x+1}^{n+1}=\theta_{j,N_x}^{n+1}$.
\subsection{Simulation results and validation}

We now present numerical results obtained with the scheme described
above, taking the parameters of Example~1 and 
$(N_x,N_y)=(128,80)$, $\Delta t=10^{-5}$, $T=500$ in two cases: 1) $R=6.297>R_{c}$ and 2) $R=6.2<R_{c}$.
\begin{figure}[H]
    \centering
    {\includegraphics[width=4.8in]{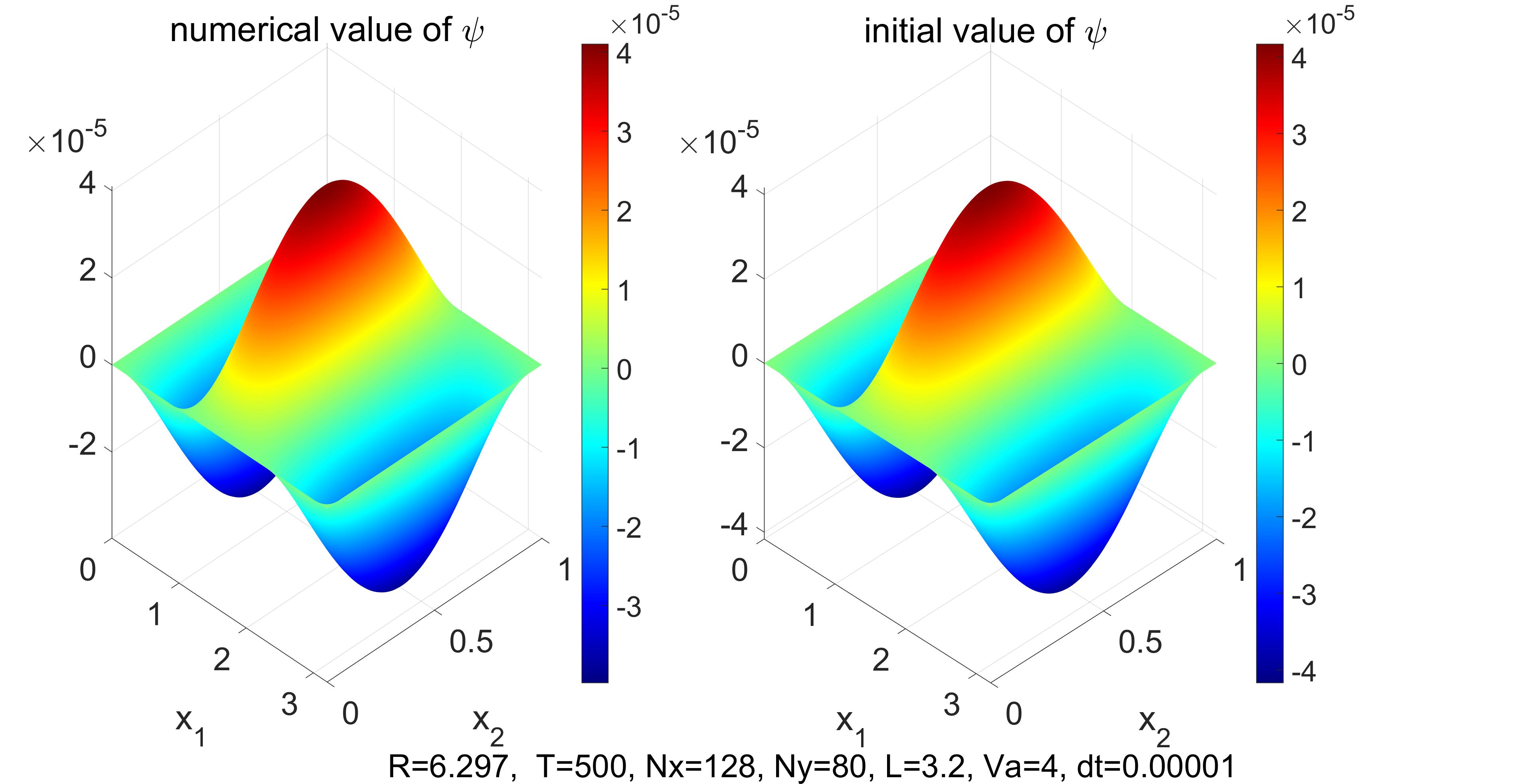}}
\caption{Stream function $\Psi$: numerical solution
(left) vs.\ initial value (right).
}.
\label{psi0702}
\end{figure}

\begin{figure}[H]
    \centering
    {\includegraphics[width=4.8in]{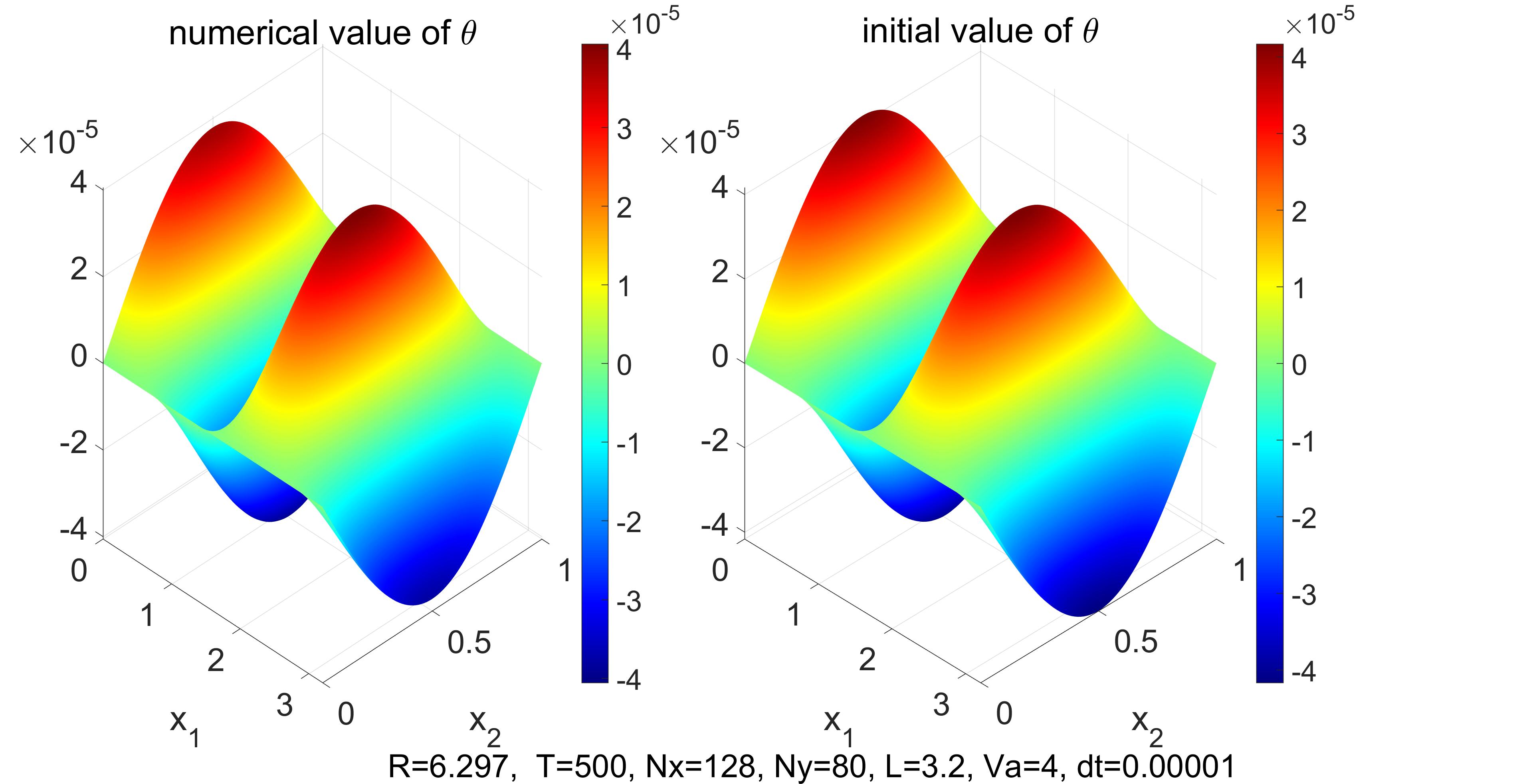}}
\caption{Temperature field $\theta$: numerical solution
(left) vs.\ initial value (right).}.
\label{theta0702}
\end{figure}

\begin{table}[htbp]
\centering
\begin{tabular}{ccccc}
\hline
$t$ & $\max|\psi|$ & $\max|\theta|$ & $\psi$ decrease (\%) & $\theta$ decrease (\%) \\
\hline
$0$   & $4.1656\times10^{-5}$ & $4.1657\times10^{-5}$ & --- & --- \\
$100$ & $4.161\times10^{-5}$  & $4.151\times10^{-5}$  & 0.11 & 0.35 \\
$200$ & $4.152\times10^{-5}$  & $4.131\times10^{-5}$  & 0.33 & 0.83 \\
$300$ & $4.138\times10^{-5}$  & $4.108\times10^{-5}$  & 0.66 & 1.39 \\
$400$ & $4.120\times10^{-5}$  & $4.084\times10^{-5}$  & 1.09 & 1.96 \\
$500$ & $4.098\times10^{-5}$  & $4.057\times10^{-5}$  & 1.62 & 2.61 \\
\hline
\end{tabular}
\caption{Evolution of the maximum amplitudes of the stream function and temperature field at $R=6.297$. Parameters: $L=3.2$, $V_a=4$, mesh $N_x=128$, $N_y=80$, time step $\Delta t=10^{-5}$, total time $T=500$.}
\label{tab:R6297_amp}
\end{table}
\begin{figure}[H]
    \centering
    {\includegraphics[width=4.8in]{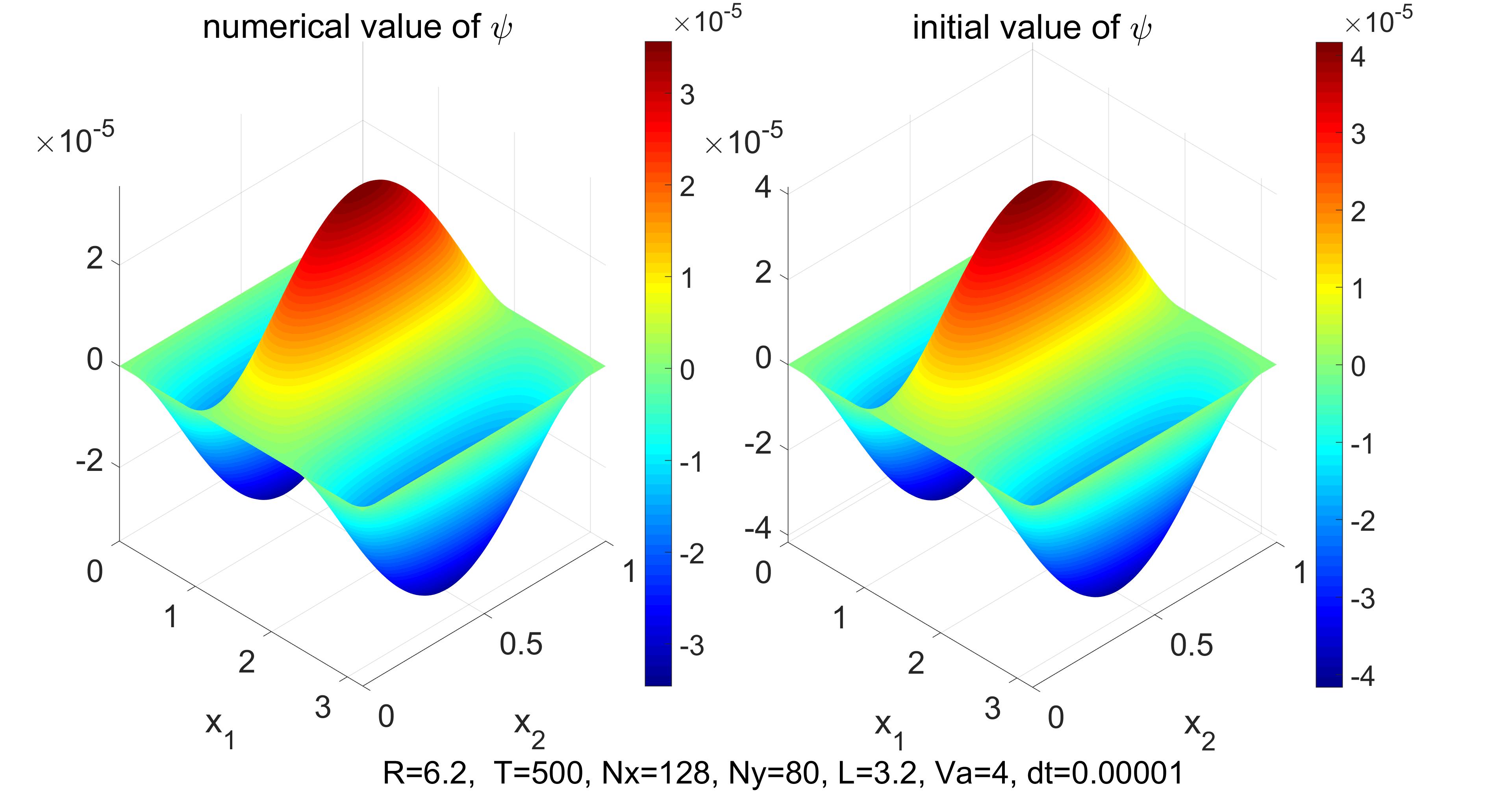}}
\caption{Stream function $\Psi$: numerical solution
(left) vs.\ initial value (right).
}.
\label{psi0714}
\end{figure}

\begin{figure}[H]
    \centering
    {\includegraphics[width=4.8in]{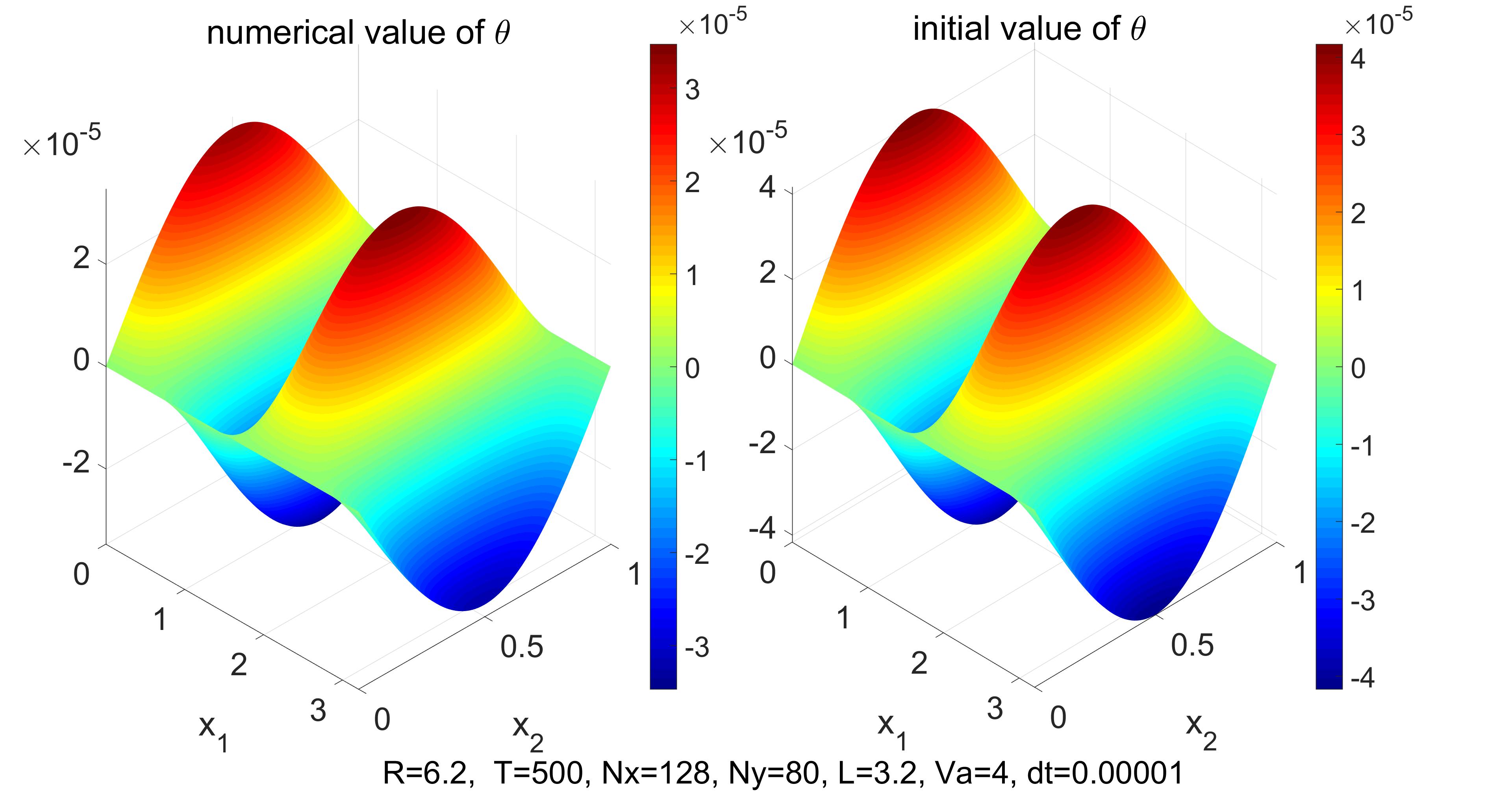}}
\caption{Temperature field $\theta$: numerical solution
(left) vs.\ initial value (right).
}.
\label{theta0714}
\end{figure}

\begin{table}[htbp]
\centering
\begin{tabular}{ccccc}
\hline
$t$ & $\max|\psi|$ & $\max|\theta|$ & $\psi$ decrease (\%) & $\theta$ decrease (\%) \\
\hline
$0$   & $4.1656\times10^{-5}$ & $4.1657\times10^{-5}$ & --- & --- \\
$100$ & $4.045\times10^{-5}$  & $3.974\times10^{-5}$  & 2.90 & 4.60 \\
$200$ & $3.924\times10^{-5}$  & $3.845\times10^{-5}$  & 5.80 & 7.70 \\
$300$ & $3.802\times10^{-5}$  & $3.717\times10^{-5}$  & 8.73 & 10.77 \\
$400$ & $3.680\times10^{-5}$  & $3.592\times10^{-5}$  & 11.66 & 13.77 \\
$500$ & $3.560\times10^{-5}$  & $3.470\times10^{-5}$  & 14.54 & 16.70 \\
\hline
\end{tabular}
\caption{Evolution of the maximum amplitudes of the stream function and temperature field at $R=6.2$. Parameters: $L=3.2$, $V_a=4$, mesh $N_x=128$, $N_y=80$, time step $\Delta t=10^{-5}$, total time $T=500$.}
\label{tab:R62_amp}
\end{table}

From \autoref{psi0702}, \autoref{theta0702} and \autoref{tab:R6297_amp}, one can see that the
numerical solution has evolved toward a steady state whose spatial
structure closely matches the theoretical bifurcation pattern
predicted by the center manifold reduction. And from \autoref{psi0714}, \autoref{theta0714} and \autoref{tab:R62_amp}, we see that the amplitudes of the stream function and temperature field decrease obviously.
These findings provide strong computational evidence that the
center manifold reduction developed in subsection~\ref{diyi0211}
accurately captures the nonlinear dynamics of the Darcy--B\'{e}nard
system \eqref{model}-\eqref{bianjian0814} near the primary bifurcation point, despite the non-analytic
of the underlying linear operator.

\subsection{The effect of $V_{a}$ on convection}\label{zhenfu0714}
From formula \eqref{linjie0926}, the critical value of $R$ is determined by the horizontal length $L$. We next investigate the effect of $V_{a}$ on convection. \autoref{vadeyingxiang0714} shows that $-a$ increases monotonically with $V_{a}$. Combined with \eqref{tehzenxiangliang0814} and \eqref{liangge0211}, it can be readily deduced that the amplitude of bifurcated solutions decreases monotonically with $V_{a}$ as $R$ is in the neighborhood of $R_{c}$ and $R>R_{c}$.  

\begin{figure}[H]
    \centering
    {\includegraphics[width=4.5in]{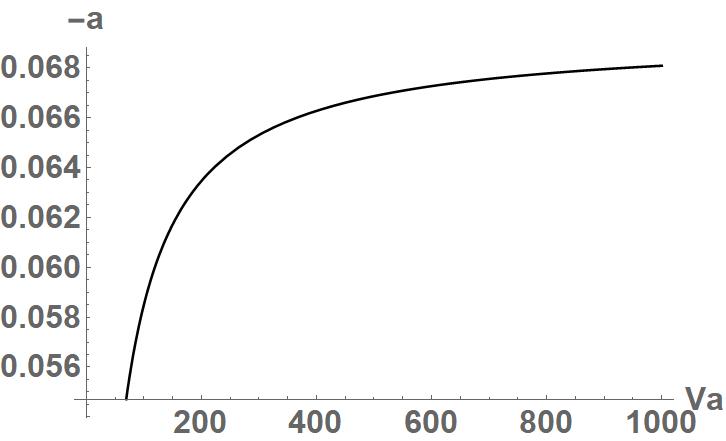}}
\caption{$L=3.2$, $R=R_{c}$ and $V_{a}\in [0.1,1000]$
}.
\label{vadeyingxiang0714}
\end{figure}

\section{Conclusions}\label{jielun0719}
In summary, this work establishes a rigorous center-manifold reduction for the Darcy–B\'{e}nard convection problem with non-zero Prandtl number, overcoming the loss of compactness of the linearized operator and non-Lipschitz nonlinearity. Beyond the specific DBC model, the methodological framework developed here—particularly the construction of inequalities that exploit partial dissipativity to compensate for the lack of full analyticity—offers a general approach to center-manifold theory for systems with partial dissipativity. A natural direction for future work is the application of this framework to more complex models, such as the Darcy–Forchheimer–B\'{e}nard system, where similar issues arise.

\end{document}